\documentclass[a4paper,11pt]{article}

\usepackage{amsmath}
\usepackage{amsthm}
\usepackage{amsfonts}
\usepackage{amssymb}
\usepackage{t-angles}

\usepackage[all]{xy}
\usepackage{amscd}

\newcommand{\al}{\alpha}
\newcommand{\si}{\sigma}

\newcommand{\id}{\mathrm{id}}

\newcommand{\ot}{\otimes}

\newcommand{\trl}{\triangleleft}
\newcommand{\trr}{\triangleright}

\def\ppr{\rightharpoonup}
\def\ppl{\leftharpoonup}

\newcommand{\li}{{}_{1}}
\newcommand{\lii}{{}_{2}}

\newcommand{\lmo}{{}_{(0)}} 

\newcommand{\loo}{{}_{(0)}}
\newcommand{\loi}{{}_{(-1)}}

\newcommand{\lmoo}{{}_{(0)}}

\newcommand{\lmi}{{}_{(1)}}

\newcommand{\lmoi}{{}_{(-1)}}

\newcommand{\mo}{{}_{(0)}}
\newcommand{\mi}{{}_{(1)}}

\newcommand{\moi}{{}_{(-1)}}

\newcommand{\boo}{{}_{[0]}}
\newcommand{\bi}{{}_{[1]}}
\newcommand{\bii}{{}_{[2]}}
\newcommand{\boi}{{}_{[-1]}}

\newcommand{\ppi}{{}_{<1>}}
\newcommand{\pii}{{}_{<2>}}
\newcommand{\qi}{{}_{\{1\}}}
\newcommand{\qii}{{}_{\{2\}}}

\def\rbiprod{{\cdot\kern-.33em\triangleright\!\!\!<}}
\def\lbiprod{{>\!\!\!\triangleleft\kern-.33em\cdot\, }}

\def\lrbiprod{{\ \cdot\kern-.60em\triangleright\kern-.33em\triangleleft\kern-.33em\cdot\, }}

\def\lprod{{>\!\!\!\triangleleft\kern-.33em\ \, }}

\newcommand{\lrcoprod}{{\,\blacktriangleright\!\!\blacktriangleleft\, }}

\allowdisplaybreaks
\newtheorem{theorem}{Theorem}[section]
\newtheorem{lemma}[theorem]{Lemma}

\newtheorem{corollary}[theorem]{Corollary}
\theoremstyle{definition}
\newtheorem{definition}[theorem]{Definition}

\title{Extending structures for alternative bialgebras}
\author{Tao Zhang, Fang Yang}
\date{}

\begin{document}
 \maketitle

 \setcounter{section}{0}

\begin{abstract}
We introduce the concept of braided  alternative  bialgebra. The theory of cocycle bicrossproducts for alternative bialgebras is developed.
As an application, the extending problem for alternative bialgebra is solved by using some non-abelian cohomology theory.

\end{abstract}

\tableofcontents

\section{Introduction}

As a generalization of associative algebras, various aspects of alternative algebra have been investigated by mathematicians. The theory of representations of alternative algebras was given by R. D. Schafer in \cite{Schafer} and N. Jacobson in \cite{Ja54}. The concept of alternative bialgebras was introduced by Goncharov in \cite{Gon} which is related to classical Yang--Baxter equation on alternative algebras.
This new type of algebraic structure was further studied by Ni and Bai in \cite{NB} and by Sun in \cite{Sun}.

The theory of extending structure for many types of algebras were well  developed by A. L. Agore and G. Militaru in \cite{AM1,AM2,AM3,AM4,AM5,AM6}.
Let $A$ be an algebra and $E$ a vector space containing $A$ as a subspace.
The extending problem is to describe and classify all algebra structures on $E$ such that $A$ is a subalgebra of $E$.
They show that associated to any extending structure of $A$ by a complement space $V$, there is a  unified product on the direct sum space  $E\cong A\oplus V$.
Recently, extending structures for  alternative algebras and pre-alternative algebras were studied  in \cite{Z5}.
Extending structures for 3-Lie algebras, Lie bialgebras,  infinitesimal bialgebras and anti-flexible bialgebras were studied  in \cite{Hong,Z2,Z3,Z4}.

As a continue of our paper \cite{Z5} and \cite{Z3,Z4}, the aim of this paper is to study  extending structures for  alternative bialgebras.
For this purpose, we will introduce the concept of braided alternative bialgebras which a generalization of Goncharov's alternative bialgebras.  Then we give the construction of cocycle bicrossproducts for alternative bialgebras.
We will show that these new concept  and  construction  will play a key role in considering extending problem for alternative bialgebras.
As an application, we solve the extending problem for alternative bialgebras by using some non-abelian cohomology theory.

This paper is organized as follows. In Section 2, we recall some
definitions and fix some notations. In Section 3, we introduced the concept of braided alternative bialgebras and proved the bosonisation
theorem associating braided alternative bialgebras to ordinary alternative bialgebras.
In Section 4, we define the notion of matched pairs of  braided alternative bialgebras and construct cocycle bicrossproduct alternative bialgebras through two generalized braided alternative bialgebras.
In Section 5, we studied the extending problems for alternative bialgebras and proof that they can be classified by some non-abelian cohomology theory.

Throughout the following of this paper, all vector spaces will be over a fixed field of character zero. 
An  algebra or a   coalgebra   is  denoted by $(A, \cdot)$ or $(A, \Delta)$.
The identity map of a vector space $V$ is denoted by $\id_V: V\to V$ or simply $\id: V\to V$. The flip map $\tau: V\ot V\to V\ot V$ is defined by $\tau(u\ot v)=v\ot u$ for all $u, v\in V$.

\section{Preliminaries}

\begin{definition} An alternative algebra is a vector space $A$ with a multiplication $\cdot : A\ot A\rightarrow A : (x, y)\mapsto x\cdot y$ such that the following identities hold:
\begin{equation}\label{eq:LB1}
 ({a}, {b}, {c})=-({b}, {a}, {c}), \quad  ({a}, {b}, {c})=-({a}, {c}, {b}),
 \end{equation}
 where $({a}, {b}, {c})=({a}\cdot {b})\cdot {c}-{a}\cdot({b}\cdot {c})$ is the associator of the elements ${a}, {b}, {c}\in A$.

In the following, we always omit $``\cdot"$ in calculation of this paper for convenience.

Note that the above identities are equivalent to the following identities:
\begin{eqnarray}\label{eq:LB2}
 && ({a} {b}){c}-{a}({b} {c})+({b} {a}){c}-{b}({a} {c})=0,\\
\label{eq:LB3}
&& ({a} {b}){c}-{a}({b} {c})+({a} {c}){b}-{a}({c} {b})=0.
 \end{eqnarray}
\end{definition}
\begin{definition}  An alternative coalgebra $A$ is a vector space equipped with a  comultiplication $\Delta: A\rightarrow A\ot A$ such that the following conditions are satisfied,
\begin{equation}\label{eq:LB4}
 (\Delta\ot\id)\Delta(a)-(\id\ot\Delta)\Delta(a)=-\tau_{12}\big((\Delta\ot\id)\Delta(a)-(\id\ot\Delta)\Delta(a)\big),
 \end{equation}
 \begin{equation}\label{eq:LB5}
 (\Delta\ot\id)\Delta(a)-(\id\ot\Delta)\Delta(a)=-\tau_{23}\big((\Delta\ot\id)\Delta(a)-(\id\ot\Delta)\Delta(a)\big),
 \end{equation}
where $\tau_{12}(a\ot b\ot c)=b\ot a\ot c$ and $\tau_{23}(a\ot b\ot c)=a\ot c\ot b$. We denote this alternative coalgebra by $(A, \Delta)$.
\end{definition}
\begin{definition}(\cite{Gon}) \label{dfnlb}
An alternative bialgebra is a vector space $A$ equipped simultaneously with a multiplication $\cdot : A\ot A\rightarrow A  $ and a comultiplication $\Delta: A\rightarrow A\ot A$ such that $(A, \, \cdot)$ is an alternative algebra , $(A, \Delta)$ is an alternative coalgbra and the following conditions are satisfied,
\begin{equation}\label{eq:LB6}
 \Delta(ab)=\sum (a\li b\ot a\lii+a\lii b\ot a\li-a\lii\ot b a\li)
 +\sum (b\li\ot a b\lii+b\li\ot b\lii a-a b\li\ot b\lii),
 \end{equation}
 \begin{equation}\label{eq:LB7}
 \Delta(ba)+\tau\Delta(b a)=\sum(a\li\ot b a\lii+b a\lii\ot a\li)
 +\sum (b\li a\ot b\lii+b\lii\ot b\li a).
 \end{equation}
\end{definition}

We denote it by $(A, \cdot , \Delta )$. One can also write the above equations as ad-actions on tensors by
\begin{equation}\label{eq:LB8}
\Delta(ab)=\Delta (a)\bullet b+\tau\Delta (a)\bullet b- b\bullet\tau\Delta(a)+a\bullet \Delta(b)+[\Delta(b), a],
 \end{equation}
 \begin{equation}\label{eq:LB9}
\Delta(ba)+\tau\Delta(ba)=b\bullet \Delta (a)+b\cdot\tau\Delta (a)+ \Delta(b)\bullet a+\tau \Delta(b)\cdot a,
 \end{equation}
where $a\cdot \Delta(b): =\sum  a b\li\ot b\lii$ , $\Delta (a)\cdot b: = \sum a\li\ot a\lii b$ , $a\bullet \Delta(b): =\sum b\li\ot a b\lii$ , $\Delta(a)\bullet b: =\sum  a\li b\ot a\lii$ and $[\Delta(b), a]: =\Delta(b)\cdot a-a\cdot\Delta(b)$.

\begin{definition}
Let ${H}$ be an alternative algebra and  $V$ be a vector space. Then $V$ is called an ${H}$-bimodule if there is a pair of linear maps $ \trr :  {H}\otimes V \to V, (x, v) \to x \trr v$ and $\trl :  V\otimes {H} \to V, (v, x) \to v \trl x$  such that the following conditions hold:
\begin{eqnarray}
\label{(10)}  &&(xy) \trr v - x \trr (y \trr v)+(yx)\trr v-y\trr (x\trr v)=0,\\
 \label{(11)}   && v \trl (xy) -(v\trl x)\trl y +x\trr(v\trl y)-(x\trr v)\trl y=0,\\
 \label{(12)}   &&(xy)\trr v-x\trr(y\trr v)+(x\trr v)\trl y-x\trr(v\trl y)  = 0,\\
\label{(13)}    &&(v\trl x)\trl y  - v\trl (xy)+(v\trl y)\trl x-v\trl (yx)=0,
\end{eqnarray}
for all $x, y\in {H}$ and $v\in V.$
\end{definition}
The category of  bimodules over $H$ is denoted  by ${}_{H}\mathcal{M}{}_{H}$.

\begin{definition}
Let ${H}$ be an alternative coalgebra, $V$ be a vector space. Then $V$ is called an ${H}$-bicomodule if there is a pair of linear maps $\phi: V\to {H}\otimes V$ and $\psi: V\to V\otimes {H}$  such that the following conditions hold:
\begin{eqnarray}
\label{(14)} &&\left(\Delta_{H} \otimes \id _{V}\right)\phi(v)-(\id _{H}\ot \phi)\phi(v)=-\tau_{12}\big((\Delta_{H}\ot\id_{V})\phi(v)-\left(\id _{H} \otimes \phi\right) \phi(v)\big),\\
\label{(15)} &&\left(\psi \otimes \id _{H}\right)\psi(v)-(\id _{V}\ot \Delta_{H})\psi(v)=-\tau_{12}\big((\phi\ot\id_{H})\psi(v)-\left(\id _{H} \otimes \psi\right) \phi(v)\big),\\
\label{(16)} &&\left(\Delta_{H} \otimes \id _{V}\right)\phi(v)-(\id _{H}\ot \phi)\phi(v)=-\tau_{23}\big((\phi\ot\id_{H})\psi(v)-\left(\id _{H} \otimes \psi\right) \phi(v)\big),\\
\label{(17)} &&\left(\psi \otimes \id _{H}\right)\psi(v)-(\id _{V}\ot \Delta_{H})\psi(v)=-\tau_{23}\big((\psi\ot\id_{H})\psi(v)-\left(\id _{V} \otimes \Delta_{H}\right) \psi(v)\big).
\end{eqnarray}
If we denote by  $\phi(v)=v\moi\ot v\mo$ and $\psi(v)=v\mo\ot v\mi$, then the above equations can be written as
\begin{eqnarray}
  &&\Delta_{H}\left(v_{(-1)}\right) \otimes v_{(0)}-v\loi\ot\phi(v\loo)=-\tau_{12}\big(\Delta_{H}(v\loi)\ot v\loo-v_{(-1)} \otimes \phi\left(v_{(0)}\right)\big),\\
  &&\psi(v\loo)\ot v\lmi-v_{(0)} \otimes \Delta_{H}\left(v_{(1)}\right)=-\tau_{12}\big(\phi\left(v_{(0)}\right) \otimes v_{(1)}-v\loi\ot\psi(v\loo)\big),\\
  &&\Delta_{H}\left(v_{(-1)}\right) \otimes v_{(0)}-v\loi\ot\phi(v\loo)=-\tau_{23}\big(\phi(v\loo)\ot v\lmi-v_{(-1)} \otimes \psi\left(v_{(0)}\right)\big),\\
  &&\psi(v\loo)\ot v\lmi-v_{(0)} \otimes \Delta_{H}\left(v_{(1)}\right)=-\tau_{23}\big(\psi\left(v_{(0)}\right) \otimes v_{(1)}-v\loo\ot\Delta_{H}(v\lmi)\big).
\end{eqnarray}
\end{definition}
The category of  bicomodules over $H$ is denoted by ${}^{H}\mathcal{M}{}^{H}$.

\begin{definition}
Let ${H}$ and  ${A}$ be alternative algebras. An action of ${H}$ on ${A}$ is a pair of linear maps $\trr:{H}\otimes {A} \to {A}, (x, a) \to x \trr a$ and $\trl: {A}\otimes {H} \to {A}, (a, x) \to a \trl x$  such that $A$ is an $H$-bimodule and  the following conditions hold:
\begin{eqnarray}
  &&(a\trl x)b-a(x\trr b)+(x\trr a)b-x\trr( a b)=0,\\
  &&(a b)\trl x-a(b\trl x)+(b a)\trl x-b(a\trl x)=0 ,\\
  &&(x\trr a)b-x\trr( a b)+(x\trr b)a-x\trr( b a)=0,\\
  &&(a\trl x)b-a(x\trr b)+(a b)\trl x-a(b\trl x)=0,
\end{eqnarray}
for all $x\in {H}$ and $a, b\in {A}.$ In this case, we call $(A, \trr, \trl)$ to be an $H$-bimodule alternative algebra.
\end{definition}

\begin{definition}
Let ${H}$ and  ${A}$ be alternative coalgebras. An coaction of ${H}$ on ${A}$ is a pair of linear maps $\phi: {A}\to {H}\otimes {A}$ and $\psi: {A}\to {A}\otimes {H}$ such that $A$ is an $H$-bicomodule and  the following conditions hold:
\begin{eqnarray}
  &&(\phi\ot\id_A)\Delta_{A}(a)-(\id_H \otimes \Delta_{A})\phi(a)=-\tau_{12}\big((\psi\ot\id_A)\Delta_{A}(a)-(\id_A\otimes \phi) \Delta_{A}(a)\big),\\
  &&(\Delta_{A}\ot\id_H)\psi(a)-(\id_A \otimes \psi)\Delta_{A}(a)=-\tau_{12}\big((\Delta_{A}\ot\id_H)\psi(a)-(\id_A\otimes \psi) \Delta_{A}(a)\big),\\
 &&(\phi\ot\id_A)\Delta_{A}(a)-(\id_H \otimes \Delta_{A})\phi(a)=-\tau_{23}\big((\phi\ot\id_A)\Delta_{A}(a)-(\id_H\otimes\Delta_{A} ) \phi(a)\big),\\
  &&(\Delta_{A}\ot\id_H)\psi(a)-(\id_A \otimes \psi)\Delta_{A}(a)=-\tau_{23}\big((\psi\ot\id_A)\Delta_{A}(a)-(\id_A\otimes \phi) \Delta_{A}(a)\big).
\end{eqnarray}
If we denote by  $\phi(a)=a\moi\ot a\mo$ and $\psi(a)=a\mo\ot a\mi$, then the above equations can be written as
\begin{eqnarray}
  &&\phi(a\li)\ot a\lii-a\loi\ot\Delta_{A}(a\loo)=-\tau_{12}\big(\psi(a\li)\ot a\lii-a\li\ot\phi(a\lii)\big),\\
  &&\Delta_{A}(a\loo)\ot a\lmi-a\li\ot\psi(a\lii)=-\tau_{12}\big(\Delta_{A}(a\loo)\ot a\lmi-a\li\ot\psi(a\lii)\big),\\
  &&\phi(a\li)\ot a\lii-a\loi\ot\Delta_{A}(a\loo)=-\tau_{23}\big(\phi(a\li)\ot a\lii-a\loi\ot\Delta_{A}(a\loo)\big),\\
  &&\Delta_{A}(a\loo)\ot a\lmi-a\li\ot\psi(a\lii)=-\tau_{23}\big(\psi(a\li)\ot a\lii-a\li\ot\phi(a\lii)\big),
\end{eqnarray}
for all $a\in {A}.$ In this case, we call $(A, \phi, \psi)$ to be an $H$-bicomodule alternative coalgebra.
\end{definition}

\begin{definition}
Let $({A}, \cdot)$ be a given alternative  algebra (alternative coalgebra, alternative  bialgebra), $E$ be a vector space.
An extending system of ${A}$ through $V$ is an alternative algebra  (alternative coalgebra, alternative bialgebra) on $E$
such that $V$ is a complement subspace of ${A}$ in $E$, the canonical injection map $i: A\to E, a\mapsto (a, 0)$  or the canonical projection map $p: E\to A, (a, x)\mapsto a$ is an alternative algebra (alternative coalgebra, alternative bialgebra) homomorphism.
The extending problem is to describe and classify up to an isomorphism  the set of all alternative algebra  (alternative coalgebra, alternative  bialgebra) structures that can be defined on $E$.
\end{definition}

We remark that our definition of extending system of ${A}$ through $V$ contains not only extending structure in \cite{AM1,AM2,AM3}
but also the global extension structure in \cite{AM5}.
The reason is that when we consider extending problem for Lie bialgebras, both of them are necessarily used,
this will be clear in the context of next two sections.
In fact, the canonical injection map $i: A\to E$ is an alternative (co)algebra homomorphism if and only if $A$ is an alternative  sub(co)algebra of $E$.

\begin{definition}
Let ${A} $ be an alternative   algebra (alternative coalgebra, alternative  bialgebra), $E$  be an alternative algebra  (alternative coalgebra, alternative  bialgebra) such that
${A} $ is a subspace of $E$ and $V$ a complement of
${A} $ in $E$. For a linear map $\varphi: E \to E$ we
consider the diagram:
\begin{equation}\label{eq:ext1}
\xymatrix{
   0  \ar[r]^{} &A \ar[d]_{\id_A} \ar[r]^{i} & E \ar[d]_{\varphi} \ar[r]^{\pi} &V \ar[d]_{\id_V} \ar[r]^{} & 0 \\
   0 \ar[r]^{} & A \ar[r]^{i'} & {E} \ar[r]^{\pi'} & V \ar[r]^{} & 0.
   }
\end{equation}
where $\pi: E\to V$ are the canonical projection maps and $i: A\to E$ are the inclusion maps.
We say that $\varphi: E \to E$ \emph{stabilizes} ${A}$ if the left square of the diagram \eqref{eq:ext1} is  commutative.
Let $(E, \cdot)$ and $(E, \cdot')$ be two alternative  algebra (alternative coalgebra, alternative  bialgebra) structures on $E$. $(E, \cdot)$ and $(E, \cdot')$ are called \emph{equivalent}, and we denote this by $(E, \cdot) \equiv (E, \cdot')$, if there exists an alternative algebra (alternative coalgebra, alternative  bialgebra) isomorphism $\varphi: (E, \cdot)
\to (E, \cdot')$ which stabilizes ${A} $. Denote by $Extd(E, {A})$ ($CExtd(E, {A} )$, $BExtd(E, {A} )$) the set of equivalent classes of  alternative  algebra (alternative coalgebra, alternative  bialgebra) structures on $E$.
\end{definition}

\section{Braided alternative bialgebras}
In this section, we introduce  the concept of alternative Hopf bimodule and braided alternative bialgebra which will be used in the following sections.

\subsection{Alternative Hopf bimodule and braided alternative bialgebra}
\begin{definition}
Let $H$ be an alternative bialgebra. An alternative Hopf bimodule over $H$ is a space $V$ endowed with maps
\begin{align*}
&\trr: H\otimes V \to V, \quad \trl: V\otimes H \to V,\\
&\phi: V \to H \otimes V, \quad  \psi: V \to V\otimes H,
\end{align*}
denoted by
\begin{eqnarray*}
&& \trr (x \otimes v) = x \trr v, \quad \trl(v\otimes x) = v \triangleleft x, \\
&& \phi (v)=\sum v\lmoi\ot v\loo, \quad \psi (v) = \sum v\loo \ot v\lmi,
\end{eqnarray*}
such that $V$ is simultaneously a bimodule, a  bicomodule over $H$ and satisfying
 the following compatibility conditions
\begin{enumerate}
\item[(HM1)] $\phi(x\trr v)=v\loi\ot(x\trr v\loo)+v\loi\ot(v\loo\trl x)-x\lii\ot(v\trl x\li)-x v\loi\ot v\loo$,
\end{enumerate}
 \begin{enumerate}
\item[(HM2)]  $\psi(v\trl x)=(v\loo\trl x)\ot v\lmi+(v\loo\trl x)\ot v\loi-v\loo\ot x v\loi-(v\trl x\li)\ot x\lii$,
\end{enumerate}
\begin{enumerate}
\item[(HM3)]  $\psi(x\trr v)=(x\li\trr v)\ot x\lii+(x\lii\trr v)\ot x\li+v\loo\ot x v\lmi+v\loo\ot v\lmi x-(x\trr v\loo)\ot v\lmi$,
\item[(HM4)]  $\phi(v \trl x)=v\loi x\ot v\loo+v\lmi x\ot v\loo-v\lmi\ot(x\trr v\loo)+x\li\ot(v\trl x\lii)+x\li\ot(x\lii\trr v)$,
\end{enumerate}
\begin{enumerate}
\item[(HM5)] $\phi(x\trr v)+\tau\psi(x\trr v)=v\loi\ot(x\trr v\loo)+x v\lmi\ot v\loo+x\lii\ot (x\li\trr v)$,

\item[(HM6)] $\phi(v\trl x)+\tau\psi(v\trl x)=x\li\ot(v\trl x\lii)+v\loi x\ot v\loo+v\lmi\ot(v\loo\trl x)$,
\end{enumerate}
for all $x\in H$ and $v\in V$.
\end{definition}

We denote  the  category of  alternative Hopf bimodules over $H$ by ${}^{H}_{H}\mathcal{M}{}^{H}_{H}$.

\begin{definition} Let $H$  be an alternative bialgebra.
If $A$ be an alternative algebra and an alternative coalgebra in ${}^{H}_{H}\mathcal{M}{}^{H}_{H}$, we call $A$ be a \emph{braided alternative bialgebra} if the following conditions are satisfied
\begin{enumerate}
\item[(BB1)]
$\Delta_{A}(a b)=a_{1} b \otimes a_{2} +a\lii b\ot a\li-a\lii\ot b a\li+b\li\ot a b\lii +b\li\ot b\lii a-a b\li\ot b\lii\\
+(a\loi\trr b)\ot a\loo+(a\lmi\trr b)\ot a\loo-a\loo\ot (b\trl a\loi)\\
+b\loo\ot (a\trl b\lmi)+b\loo\ot(b\lmi\trr a)-(a\trl b\loi)\ot b\loo$,
\end{enumerate}
\begin{enumerate}
\item[(BB2)]
$\Delta_{A}(b a)+\tau\Delta_{A}(b a)=a\li\ot b a\lii+b a\lii\ot a\li+b\li a\ot b\lii+b\lii\ot b\li a\\
 +a\loo\ot(b\trl a\lmi)+(b\trl a\lmi)\ot a\loo+b\loo\ot(b\loi\trr a)+(b\loi\trr a)\ot b\loo$.
\end{enumerate}
\end{definition}

Here we say $A$ to be an alternative algebra and an alternative coalgebra in ${}^{H}_{H}\mathcal{M}{}^{H}_{H}$ means that $A$ is simultaneously an $H$-bimodule alternative algebra (alternative coalgebra) and $H$-bicomodule alternative algebra (alternative coalgebra).

Now we construct  alternative bialgebra from braided alternative  bialgebra.
Let $H$  be an alternative bialgebra, $A$ be an  alternative algebra and an alternative  coalgebra in ${}^{H}_{H}\mathcal{M}{}^{H}_{H}$.
We define the multiplication and comultiplication on the direct sum vector space $E: =A \oplus H$ by
\begin{eqnarray}
\label{sp01}
&&(a+ x)(b+ y): =a b+x\trr b+a \trl y+ x y, \\
\label{sp02}
&&\Delta_{E}(a+ x): =\Delta_{A}(a)+\phi(a)+\psi(a)+\Delta_{H}(x).
\end{eqnarray}
This is called biproduct of ${A}$ and ${H}$ which will be  denoted   by $A\lbiprod H$.

\begin{theorem} \label{thm-main0}
Let  $H$ be an alternative bialgebra.
Then the biproduct $A\lbiprod H$ forms an alternative bialgebra if and only if  $A$ is a braided alternative bialgebra in ${}^{H}_{H}\mathcal{M}{}^{H}_{H}$.
\end{theorem}

\begin{proof}
 It is easy to prove that $A\lbiprod H$ is an alternative algebra and an alternative coalgebra with the multiplication \eqref{sp01} and comultiplication  \eqref{sp02} .
Now we show the compatibility conditions:
$$
\begin{aligned}
&\Delta_{E}((a+x)(b+y))\\
=&\Delta_{E} (a+x)\bullet (b+y)+\tau\Delta_{E} (a+x)\bullet (b+y)\\
&- (b+y)\bullet\tau\Delta_{E}(a+x)+(a+x)\bullet \Delta_{E}(b+y)\\
&+[\Delta_{E}(b+y), (a+x)],
 \end{aligned}
$$
 $$
\begin{aligned}
&\Delta_{E}((b+y)(a+x))+\tau\Delta_{E}((b+y)(a+x))\\
=&(b+y)\bullet \Delta_{E} (a+x)+(b+y)\cdot\tau\Delta_{E} (a+x)\\
&+ \Delta_{E}(b+y)\bullet (a+x)+\tau \Delta_{E}(b+y)\cdot (a+x).
 \end{aligned}
$$
By direct computations, for the first equation, the left hand side is equal to
$$\begin{aligned}
&\Delta_{E}((a+x)(b+y))\\
=&\Delta_E(a b+x \trr b+a \trl y+ x y)\\
=&\Delta_A(a b)+\Delta_A(x \trr b)+\Delta_A(a \trl y)+\phi(a b)+\phi(x \trr b)+\phi(a \trl y)\\
&+\psi(a b)+\psi(x \trr b)+\psi(a \trl y)+\Delta_{H}(x y),
\end{aligned}
$$
and the right hand side is equal to
\begin{eqnarray*}
&&\Delta_{E} (a+x)\bullet (b+y)+\tau\Delta_{E} (a+x)\bullet (b+y)- (b+y)\bullet\tau\Delta_{E}(a+x)\\
&&+(a+x)\bullet \Delta_{E}(b+y)+[\Delta_{E}(b+y), (a+x)]\\
&=&\left(a_{1} \otimes a_{2}+a_{(-1)} \otimes a_{(0)}+a_{(0)} \otimes a_{(1)}+x_{1} \otimes x_{2}\right) \bullet(b+ y)\\
&&+\left(a_{2} \otimes a_{1}+a_{(0)} \otimes a_{(-1)}+a_{(1)} \otimes a_{(0)}+x_{2} \otimes x_{1}\right) \bullet(b+ y)\\
&&-(b+y)\bullet\left(a_{2} \otimes a_{1}+a_{(0)} \otimes a_{(-1)}+a_{(1)} \otimes a_{(0)}+x_{2} \otimes x_{1}\right)\\
&&+(a+ x) \bullet\left(b_{1} \otimes b_{2}+b_{(-1)} \otimes b_{(0)}+b_{(0)} \otimes b_{(1)}+y_{1} \otimes y_{2}\right)\\
&&+[b_{1} \otimes b_{2}+b_{(-1)} \otimes b_{(0)}+b_{(0)} \otimes b_{(1)}+y_{1} \otimes y_{2}, a+x]\\
&=&a_{1}b \otimes a_{2}+(a\li\trl y)\ot a\lii+(a\loi \trr b)\ot a\loo +a\loi y\ot a\loo \\
&&+ a\loo b\ot a\lmi+(a\loo\trl y)\ot a\lmi+(x\li\trr b)\ot x\lii+x\li y\ot x\lii\\
&&+a\lii b\ot a\li+(a\lii\trl y)\ot a\li+a\loo b\ot a\loi+(a\loo\trl y)\ot a\loi\\
&&+(a\lmi\trr b)\ot a\loo+a\lmi y\ot a\loo+(x\lii\trr b)\ot x\li+x\lii y\ot x\li\\
&&-a\lii\ot b a\li-a\lii\ot (y\trr a\li)-a\loo\ot(b\trl a\loi)-a\loo\ot ya\loi\\
&&-a\lmi\ot ba\loo-a\lmi\ot(y\trr a\loo)-x\lii\ot(b\trl x\li)-x\lii\ot yx\li\\
&&+b\li\ot a b\lii+b\li\ot (x\trr b\lii)+b\loi\ot a b\loo+b\loi\ot (x\trr b\loo)\\
&&+b\loo\ot (a\trl b\lmi)+b\loo\ot x b\lmi+y\li\ot (a\trl y\lii)+y\li\ot x y\lii\\
&&+b\li\ot b\lii a+b\li\ot (b\lii\trl x)-a b\li\ot b\lii-(x\trr b\li)\ot b\lii\\
&&+b\loi\ot b\loo a+b\loi\ot(b\loo\trl x)-(a\trl b\loi)\ot b\loo-x b\loi\ot b\loo\\
&&+b\loo\ot (b\lmi\trr a)+b\loo\ot b\lmi x-a b\loo\ot b\lmi-(x\trr b\loo)\ot b\lmi\\
&&+y\li\ot (y\lii\trr a)+y\li\ot y\lii x-(a\trl y\li)\ot y\lii-x y\li\ot y\lii.
\end{eqnarray*}
Thus the two sides are equal to each other if and only if

(1) $\Delta_{A}(a b)=a_{1} b \otimes a_{2} +a\lii b\ot a\li-a\lii\ot b a\li+b\li\ot a b\lii +b\li\ot b\lii a-a b\li\ot b\lii$

$\qquad+(a\loi\trr b)\ot a\loo+(a\lmi\trr b)\ot a\loo-a\loo\ot (b\trl a\loi)$

$\qquad+b\loo\ot (a\trl b\lmi)+b\loo\ot(b\lmi\trr a)-(a\trl b\loi)\ot b\loo$,

(2) $\phi(a \trl y)=a\loi y\ot a\loo+a\lmi y\ot a\loo-a\lmi\ot(y\trr a\loo)+y\li\ot(a\trl y\lii)+y\li\ot(y\lii\trr a)$,

(3) $\phi(x\trr b)=-x\lii\ot(b\trl x\li)+b\loi\ot(x\trr b\loo)+b\loi\ot(b\loo\trl x)-x b\loi\ot b\loo$,

(4) $\psi(x\trr b)=(x\li\trr b)\ot x\lii+(x\lii\trr b)\ot x\li+b\loo\ot x b\lmi+b\loo\ot b\lmi x-(x\trr b\loo)\ot b\lmi$,

(5) $\psi(a\trl y)=(a\loo\trl y)\ot a\lmi+(a\loo\trl y)\ot a\loi-a\loo\ot y a\loi-(a\trl y\li)\ot y\lii$,

(6) $\Delta_{A}(x \trr b)=b\li\ot(x\trr b\lii)+b\li\ot (b\lii\trl x)-(x\trr b\li)\ot b\lii$,

(7) $\Delta_{A}(a\trl y)=(a\li\trl y)\ot a\lii+(a\lii\trl y)\ot a\li-a\lii\ot(y\trr a\li)$,

(8) $\phi(ab)=b\loi\ot a b\loo+b\loi\ot b\loo a-a\lmi\ot b a\loo$,

(9) $\psi(a b)=a\loo b\ot a\lmi+a\loo b\ot a\loi-a b\loo\ot b\lmi$.

For the second equation, the left hand side is equal to
\begin{eqnarray*}
&&\Delta_{E}((b+y)(a+x))+\tau\Delta_{E}((b+y)(a+x))\\
&=&\Delta_E( b a+y \trr a+b \trl x+  y x)+\tau\Delta_E( b a+y \trr a+b \trl x+  y x)\\
&=&\Delta_A( ba)+\Delta_A(y \trr a)+\Delta_A(b \trl x)+\phi(b a)+\phi(y \trr a)+\phi(b \trl x)\\
&&+\psi(b a)+\psi(y \trr a)+\psi(b \trl x)+\Delta_{H}(yx)+\tau\Delta_A( ba)+\tau\Delta_A(y \trr a)\\
&&+\tau\Delta_A(b \trl x)+\tau\phi(b a)+\tau\phi(y \trr a)+\tau\phi(b \trl x)+\tau\psi(b a)+\tau\psi(y \trr a)\\
&&+\tau\psi(b \trl x)+\tau\Delta_{H}(yx),
\end{eqnarray*}
and the right hand side is equal to
\begin{eqnarray*}
&&(b+y)\bullet \Delta_{E} (a+x)+(b+y)\cdot\tau\Delta_{E} (a+x)\\
&&+ \Delta_{E}(b+y)\bullet (a+x)+\tau \Delta_{E}(b+y)\cdot (a+x)\\
&=&(b+y)\bullet\left(a_{1} \otimes a_{2}+a_{(-1)} \otimes a_{(0)}+a_{(0)} \otimes a_{(1)}+x_{1} \otimes x_{2}\right)\\
&&+(b+y)\cdot\left(a_{2} \otimes a_{1}+a_{(0)} \otimes a_{(-1)}+a_{(1)} \otimes a_{(0)}+x_{2} \otimes x_{1}\right)\\
&&+\left(b_{1} \otimes b_{2}+b_{(-1)} \otimes b_{(0)}+b_{(0)} \otimes b_{(1)}+y_{1} \otimes y_{2}\right)\bullet(a+x)\\
&&+(b\lii\ot b\li+b\loo\ot b\loi+b\lmi\ot b\loo+y\lii\ot y\li)\cdot (a+x)\\
&=&a\li\ot b a\lii+a\li\ot(y\trr a\lii)+a\loi\ot b a\loo+a\loi\ot (y\trr a\loo)\\
&&+a\loo\ot(b\trl a\lmi)+a\loo\ot y a\lmi+x\li\ot(b\trl x\lii)+x\li\ot y x\lii\\
&&+b a\lii\ot a\li+(y\trr a\lii)\ot a\li+b a\loo\ot a\loi+(y\trr a\loo)\ot a\loi\\
&&+(b\trl a\lmi)\ot a\loo+y a\lmi\ot a\loo+(b\trl x\lii)\ot x\li+y x\lii\ot x\li\\
&&+b\li a\ot b\lii+(b\li\trl x)\ot b\lii+(b\loi\trr a)\ot b\loo+b\loi x\ot b\loo\\
&&+b\loo a\ot b\lmi+(b\loo\trl x)\ot b\lmi+(y\li\trr a)\ot y\lii+y\li x\ot y\lii\\
&&+b\lii\ot b\li a+b\lii\ot (b\li\trl x)+b\loo\ot(b\loi\trr a)+b\loo\ot b\loi x\\
&&+b\lmi\ot b\loo a+b\lmi\ot(b\loo\trl x)+y\lii\ot(y\li\trr a)+y\lii\ot y\li x.
\end{eqnarray*}
Then the two sides are equal to each other if and only if satisfying the following conditions

(10)  $\Delta_{A}(b a)+\tau\Delta_{A}(b a)=a\li\ot b a\lii+b a\lii\ot a\li+b\li a\ot b\lii+b\lii\ot b\li a$

 $\qquad+a\loo\ot(b\trl a\lmi)+(b\trl a\lmi)\ot a\loo+b\loo\ot(b\loi\trr a)+(b\loi\trr a)\ot b\loo$,

(11) $\Delta_{A}(y\trr a)+\tau\Delta_{A}(y\trr a)=a\li\ot (y\trr a\lii)+(y\trr a\lii)\ot a\li$,

(12) $\Delta_{A}(b\trl x)+\tau\Delta_{A}(b\trl x)=b\lii\ot(b\li\trl x)+(b\li\trl x)\ot b\lii$,

(13) $\phi(b a)+\tau\psi(b a)=a\loi\ot b a\loo+b\lmi\ot b\loo a$,

(14) $\psi(ba)+\tau\phi(b a)=b a\loo\ot a\loi+b\loo a\ot b\lmi$,

(15) $\phi(b\trl x)+\tau\psi(b\trl x)=x\li\ot(b\trl x\lii)+b\loi x\ot b\loo+b\lmi\ot(b\loo\trl x)$,

(16) $\phi(y\trr a)+\tau\psi(y\trr a)=a\loi\ot(y\trr a\loo)+y a\lmi\ot a\loo+y\lii\ot (y\li\trr a)$,

(17) $\psi(y\trr a)+\tau\phi(y\trr a)=a\loo\ot y a\lmi+(y\trr a\loo)\ot a\loi+(y\li\trr a)\ot y\lii$,

(18) $\psi(b\trl x)+\tau\phi(b\trl x)=(b\trl x\lii)\ot x\li+(b\loo\trl x)\ot b\lmi+b\loo\ot b\loi x$.

\noindent From (6)--(9) and(11)--(14) we have that $A$ is an alternative algebra and an alternative coalgebra in ${}^{H}_{H}\mathcal{M}{}^{H}_{H}$,
from (2)--(5) and (15)--(18) we get that $A$ is an alternative Hopf bimodule over $H$, and (1) together with (10) are the conditions for $A$ to be a braided alternative bialgebra.
The proof is completed.
\end{proof}

\subsection{From quasitriangular alternative bialgebra to braided  alternative bialgebra}

Let $A$ be an alternative algebra and $M$ an $A$-bimodule. Given an element $r=$ $\sum_{i} u_{i} \otimes v_{i} \in A \otimes A$, there is a derivation $\Delta_{r}: A \rightarrow M$ associated to each element $r \in M$ as follows:
\begin{equation}\label{eq:deltar}
\Delta_{r}(a)=a \star r-r  \star a= u_i a\ot v_i-u_i\ot a v_i.
\end{equation}
Denote by
$$
r_{12}=\sum_{i} u_{i} \otimes v_{i} \otimes 1, \quad r_{13}=\sum_{i} u_{i} \otimes 1 \otimes v_{i}, \quad r_{23}=\sum_{i} 1 \otimes u_{i} \otimes v_{i}.
$$

\begin{lemma}(\cite{Gon})
Let $(A, \cdot )$ be an alternative  algebra and $ {r}\in A\otimes A$. Let $\Delta:A\rightarrow A\otimes A$
be a linear map defined by Eq.~\eqref{eq:deltar}.
Assume that $ {r}$ is  skew-symmetric and $ {r}$ satisfies
\begin{equation}\label{eqYBE}
 {r}_{{23}} {r}_{{12}}- {r}_{{12}} {r}_{{13}}-{r}_{{13}} {r}_{{23}}=0,
\end{equation}
then $(A, \cdot, \Delta_r)$ is an alternative  bialgebra.
\end{lemma}

The equation \eqref{eqYBE} is called the alternative Yang-Baxter equation (AYB).
A quasitriangular alternative bialgebra is a pair $(A, r)$ where $A$ is an alternative algebra and $r \in A \otimes A$ is a solution to (AYB).

\begin{lemma}
Let $(A, \cdot, r )$ be a quasitriangular alternative  bialgebra of $\Delta: =\Delta_r$. Then
\begin{equation}\label{r11}
(\Delta\ot \id)(r)=-r_{13}r_{23},
\end{equation}
\begin{equation}\label{r12}
(\id\ot\Delta)(r)=r_{12}r_{13}.
\end{equation}
\end{lemma}
\begin{proof}
By direct computations, we have
\begin{eqnarray*}
(\Delta\ot \id)(r)&=&\sum_{i} \Delta(u_i)\ot v_i=\sum_{ij} (u_j u_i\ot v_j-u_j \ot u_i v_j)\ot v_i\\
&=&\sum_{ij} (u_j u_i\ot v_j\ot v_i-u_j \ot u_i v_j\ot v_i)\\
&=&r_{12}r_{13}-r_{23}r_{12}\\
&=&-r_{13}r_{23},
\end{eqnarray*}
\begin{eqnarray*}
(\id\ot\Delta)(r)&=&\sum_{i} u_i\ot\Delta(v_i)=\sum_{ij} u_i\ot(u_j v_i\ot v_j-u_j \ot v_i v_j)\\
&=&\sum_{ij} (u_i\ot u_j v_i\ot v_j-u_i\ot u_j \ot v_i v_j)\\
&=&r_{23}r_{12}-r_{13}r_{23}\\
&=&r_{12}r_{13}.
\end{eqnarray*}
Thus the above equality \eqref{r11} and \eqref{r12} hold.
\end{proof}

\begin{theorem}
 Let $(A, \cdot, r)$ be a quasitriangular  alternative bialgebra and $M$ an $A$-bimodule. Then $M$ becomes an  alternative Hopf bimodule over $A$ with the $\phi: M \rightarrow A \otimes M$ and $\psi: M \rightarrow M \otimes A$ given by
\begin{equation}
\phi(m): =-\sum_{i} u_{i} \otimes m\trl v_{i} ,\quad \psi(m): = \sum_{i} u_{i}\trr m \otimes v_{i}.
\end{equation}
\end{theorem}

\begin{proof}
First we prove that $M$ is a bicomodule. By Definition 2.5, we just have to verify that equations \eqref{(14)}--\eqref{(17)} are true.

For the equation  \eqref{(14)}, the left hand side is equal to
\begin{eqnarray*}
&&(\Delta_{r} \ot \id _{A})\phi(m)-(\id _{A}\ot \phi)\phi(m)\\
&=&-\Delta_r(u_i)\ot m\trl v_i+u_i\ot \phi(m\trl v_i)\\
&=&u_i\ot u_j\ot m\trl (v_iv_j)-u_i\ot u_j\ot (m\trl v_i)\trl v_j\\
&=&u_i\ot u_j\ot (m\trl (v_iv_j)- (m\trl v_i)\trl v_j),
\end{eqnarray*}
and the right hand side is equal to
\begin{eqnarray*}
&&-\tau_{12}(\Delta_{r} \ot \id _{A})\phi(m)-(\id _{A}\ot \phi)\phi(m)\\
&=&-\tau_{12}(-\Delta_r(u_i)\ot v_i\trl m+u_i\ot \phi(m\trl v_i))\\
&=&-\tau_{12}(u_i\ot u_j\ot m\trl (v_iv_j)-u_i\ot u_j\ot (m\trl v_i)\trl v_j)\\
&=&-u_j\ot u_i\ot m\trl (v_iv_j)+u_j\ot u_i\ot (m\trl v_i)\trl v_j\\
&=&-u_i\ot u_j\ot m\trl (v_jv_i)+u_i\ot u_j\ot (m\trl v_j)\trl v_i\\
&=&u_i\ot u_j\ot(-m\trl (v_jv_i)+(m\trl v_j)\trl v_i)\\
&\overset{(13)}{=}&u_i\ot u_j\ot (m\trl (v_iv_j)- (m\trl v_i)\trl v_j).
\end{eqnarray*}
Thus the left and the right are equal  to each other.

For the equation  \eqref{(15)}, the left hand side is equal to
\begin{eqnarray*}
&&(\psi \otimes \id _{A})\psi(m)-(\id _{A}\ot \Delta_{r})\psi(m)\\
&=&\psi( u_i\trr m)\ot v_i-( u_i\trr m)\ot \Delta_r(v_i)\\
&=& u_j\trr (u_i\trr m)\ot v_j\ot v_i-(u_ju_i)\trr m\ot v_j\ot v_i\\
&=& (u_j\trr (u_i\trr m)-(u_ju_i)\trr m)\ot v_j\ot v_i,
\end{eqnarray*}
and the right hand side is equal to
\begin{eqnarray*}
&&-\tau_{12}((\phi\ot\id_{H})\psi(v)-(\id _{H} \otimes \psi) \phi(v))\\
&=&-\tau_{12}(\phi(u_i\trr m)\ot v_i+u_i\ot \psi(m\trl v_i))\\
&=&-\tau_{12}(-u_j\ot (u_i\trr m)\trl v_j\ot v_i+u_i\ot  u_j\trr (m\trl v_i)\ot v_j)\\
&=&  (u_i\trr m)\trl v_j\ot u_j\ot v_i-  u_j\trr (m\trl v_i)\ot u_i\ot v_j\\
&=&-(u_i\trr m)\trl u_j\ot v_j\ot v_i+  u_j\trr (m\trl u_i)\ot v_i\ot v_j\\
&=&-(u_i\trr m)\trl u_j\ot v_j\ot v_i+  u_i\trr (m\trl u_j)\ot v_j\ot v_i\\
&=&(-(u_i\trr m)\trl u_j+  u_i\trr (m\trl u_j))\ot v_j\ot v_i\\
&\overset{ \eqref{(11)}}{=}&(-m\trl(u_iu_j)+(m\trl u_i)\trl u_j)\ot v_j\ot v_i\\
&\overset{ \eqref{(13)}}{=}&(m\trl(u_ju_i)-(m\trl u_j)\trl u_i)\ot v_j\ot v_i\\
&\overset{ \eqref{(11)}}{=}&(-u_j\trr(m\trl u_i)+(u_j\trr m)\trl u_i)\ot v_j\ot v_i\\
&\overset{ \eqref{(12)}}{=}&(-(u_ju_i)\trr m+u_j\trr(u_i\trr m))\ot v_j\ot v_i.
\end{eqnarray*}
Thus the left and the right are equal  to each other.

For the equation  \eqref{(16)}, the left hand side is equal to
\begin{eqnarray*}
&&(\Delta_{r} \ot \id _{A})\phi(m)-(\id _{A}\ot \phi)\phi(m)\\
&=&-\Delta_r(u_i)\ot m\trl v_i+u_i\ot \phi(m\trl v_i)\\
&=&u_i\ot u_j\ot m\trl (v_iv_j)-u_i\ot u_j\ot (m\trl v_i)\trl v_j\\
&=&u_i\ot u_j\ot (m\trl (v_iv_j)- (m\trl v_i)\trl v_j),
\end{eqnarray*}
and the right hand side is equal to
\begin{eqnarray*}
&&-\tau_{23}((\phi\ot\id_{A})\psi(m)-(\id _{A} \otimes \psi) \phi(m))\\
&=&-\tau_{23}(\phi(u_i\trr m)\ot v_i+u_i\ot\psi(m\trl v_i))\\
&=&-\tau_{23}(-u_j\ot (u_i\trr m)\trl v_j\ot v_i+u_i\ot u_j\trr(m\trl v_i)\ot v_j)\\
&=& u_j\ot v_i\ot (u_i\trr m)\trl v_j-u_i\ot v_j\ot u_j\trr(m\trl v_i)\\
&=&u_i\ot v_j\ot (u_j\trr m)\trl v_i-u_i\ot v_j\ot u_j\trr(m\trl v_i)\\
&=&-u_i\ot u_j\ot (v_j\trr m)\trl v_i+u_i\ot u_j\ot v_j\trr(m\trl v_i)\\
&=&u_i\ot u_j\ot (-(v_j\trr m)\trl v_i+ v_j\trr(m\trl v_i))\\
&\overset{ \eqref{(12)}}{=}&u_i\ot u_j\ot ((v_jv_i)\trr m-v_j\trr(v_i\trr m))\\
&\overset{ \eqref{(10)}}{=}&u_i\ot u_j\ot (-(v_iv_j)\trr m+v_i\trr(v_j\trr m))\\
&\overset{ \eqref{(12)}}{=}&u_i\ot u_j\ot ((v_i\trr m)\trl v_j-v_i\trr(m\trl v_j))\\
&\overset{ \eqref{(11)}}{=}&u_i\ot u_j\ot (m\trl v_iv_j-(m\trl v_i)\trl v_j).
\end{eqnarray*}
Thus we obtain that the left and the right are equal  to each other.

For the equation  \eqref{(17)}, the left hand side is equal to
\begin{eqnarray*}
&&(\psi \otimes \id _{A})\psi(m)-(\id _{A}\ot \Delta_{r})\psi(m)\\
&=&\psi( u_i\trr m)\ot v_i-( u_i\trr m)\ot \Delta_r(v_i)\\
&=& u_j\trr (u_i\trr m)\ot v_j\ot v_i-(u_ju_i)\trr m\ot v_j\ot v_i\\
&=& (u_j\trr (u_i\trr m)-(u_ju_i)\trr m)\ot v_j\ot v_i,
\end{eqnarray*}
and the right hand side is equal to
\begin{eqnarray*}
&&-\tau_{23}((\psi\ot\id_{H})\psi(v)-(\id _{V} \otimes \Delta_{H}) \psi(v))\\
&=&-\tau_{23}(\psi( u_i\trr m)\ot v_i-( u_i\trr m)\ot \Delta_r(v_i))\\
&=&-\tau_{23}(u_j\trr (u_i\trr m)\ot v_j\ot v_i-(u_ju_i)\trr m\ot v_j\ot v_i)\\
&=&-u_j\trr (u_i\trr m)\ot v_i\ot v_j+(u_ju_i)\trr m\ot v_i\ot v_j\\
&=&-u_i\trr (u_j\trr m)\ot v_j\ot v_i+(u_iu_j)\trr m\ot v_j\ot v_i\\
&\overset{ \eqref{(10)}}{=}&(u_j\trr (u_i\trr m)-(u_ju_i)\trr m)\ot v_j\ot v_i.
\end{eqnarray*}
Thus we obtain that the left and the right are equal to each other.

Next we check that $M$ is an alternative Hopf bimodule as follows. By definition 3.1, we just have to verify equations (HM1)--(HM6) are true.

For the equation (HM1), the left hand side is equal to
\begin{eqnarray*}
&&\phi(x\trr m)=-u_i\ot (x\trr m)\trl v_i,
\end{eqnarray*}
the right hand side is equal to
\begin{eqnarray*}
&&-v_i\ot m\trl(u_ix)+xv_i\ot m\trl u_i-u_i\ot x\trr(m\trl v_i)-u_i\ot(m\trl v_i)\trl x+xu_i\ot m\trl v_i\\
&=&-v_i\ot m\trl(u_ix)-u_i\ot x\trr(m\trl v_i)-u_i\ot(m\trl v_i)\trl x-xu_i\ot m\trl v_i+xu_i\ot m\trl v_i\\
&=&-v_i\ot m\trl(u_ix)-u_i\ot x\trr(m\trl v_i)-u_i\ot(m\trl v_i)\trl x\\
&=&u_i\ot m\trl(v_ix)-u_i\ot x\trr(m\trl v_i)-u_i\ot(m\trl v_i)\trl x\\
&=&u_i\ot (m\trl(v_ix)- x\trr(m\trl v_i)-(m\trl v_i)\trl x)\\
&\overset{ \eqref{(11)}}{=}&u_i\ot (m\trl(v_ix)-(m\trl v_i)\trl x+m\trl(xv_i)-(m\trl x)\trl v_i-(x\trr m)\trl v_i)\\
&\overset{ \eqref{(13)}}{=}&-u_i\ot (x\trr m)\trl v_i.
\end{eqnarray*}

For the equation (HM2), the left hand side is equal to
\begin{eqnarray*}
&&\psi(m\trl x)=u_i\trr(m\trl x)\ot v_i,
\end{eqnarray*}
the right hand side is equal to
\begin{eqnarray*}
&&(u_i\trr m)\trl x\ot v_i-(m\trl v_i)\trl x\ot u_i+m\trl v_i\ot xu_i-m\trl(u_ix)\ot v_i+m\trl u_i\ot xv_i\\
&=&(u_i\trr m)\trl x\ot v_i-(m\trl v_i)\trl x\ot u_i-m\trl u_i\ot xv_i-m\trl(u_ix)\ot v_i+m\trl u_i\ot xv_i\\
&=&(u_i\trr m)\trl x\ot v_i-(m\trl v_i)\trl x\ot u_i-m\trl(u_ix)\ot v_i\\
&=&(u_i\trr m)\trl x\ot v_i+(m\trl u_i)\trl x\ot v_i-m\trl(u_ix)\ot v_i\\
&\overset{ \eqref{(11)}}{=}&((u_i\trr m)\trl x+(m\trl u_i)\trl x-(m\trl u_i)\trl x+u_i\trr(m\trl u_i)-(u_i\trr m)\trl x)\ot v_i\\
&=&u_i\trr(m\trl u_i)\ot v_i.
\end{eqnarray*}

For the equation (HM3), the left hand side is equal to
\begin{eqnarray*}
&&\psi(x\trr m)=u_i\trr(x\trr m)\ot v_i,
\end{eqnarray*}
the right hand side is equal to
\begin{eqnarray*}
&&(u_ix)\trr m\ot v_i-u_i\trr m\ot xv_i+v_i\trr m\ot u_ix-(xv_i)\trr m\ot u_i+u_i\trr m\ot xv_i\\
&&+u_i\trr m\ot v_ix-x\trr(u_i\trr m)\ot v_i\\
&=&(u_ix)\trr m\ot v_i+v_i\trr m\ot u_ix-(xv_i)\trr m\ot u_i-v_i\trr m\ot u_ix-x\trr(u_i\trr m)\ot v_i\\
&=&(u_ix)\trr m\ot v_i+(xu_i)\trr m\ot v_i-x\trr(u_i\trr m)\ot v_i\\
&=&((u_ix)\trr m+(xu_i)\trr m-x\trr(u_i\trr m))\ot v_i\\
&\overset{ \eqref{(10)}}{=}&u_i\trr(x\trr m)\ot v_i.
\end{eqnarray*}

For the equation (HM4), the left hand side is equal to
\begin{eqnarray*}
&&\phi(m\trl x)=-u_i\ot (m\trl x)\trl v_i,
\end{eqnarray*}
the right hand side is equal to
\begin{eqnarray*}
&&-u_ix\ot m\trl v_i+v_ix\ot u_i\trr m-v_i\ot x\trr(u_i\trr m)+u_ix\ot m\trl v_i-u_i\ot m\trl (xv_i)\\
&&+u_ix\ot v_i\trr m-u_i\ot(xv_i)\trr m\\
&=&v_ix\ot u_i\trr m-v_i\ot x\trr(u_i\trr m)-u_i\ot m\trl (xv_i)-v_ix\ot u_i\trr m-u_i\ot(xv_i)\trr m\\
&=&-v_i\ot x\trr(u_i\trr m)-u_i\ot m\trl (xv_i)-u_i\ot(xv_i)\trr m\\
&=&u_i\ot x\trr(v_i\trr m)-u_i\ot m\trl (xv_i)-u_i\ot(xv_i)\trr m\\
&=&u_i\ot (x\trr(v_i\trr m)- m\trl (xv_i)-(xv_i)\trr m)\\
&\overset{ \eqref{(12)}}{=}&u_i\ot (x\trr(v_i\trr m)-(m\trl x)\trl v_i+x\trr(m\trl v_i)-(x\trr m)\trl v_i-x\trr(v_i\trr m)\\
&&+(x\trr m)\trl v_i-x\trr(m\trl v_i))\\
&=&-u_i\ot (m\trl x)\trl v_i.
\end{eqnarray*}

For the equation (HM5), the left hand side is equal to
\begin{eqnarray*}
&&\phi(x\trr m)+\tau\psi(x\trr m)\\
&=&-u_i\ot (x\trr m)\trl v_i+v_i\ot u_i\trr(x\trr m)\\
&=&v_i\ot (x\trr m)\trl u_i+v_i\ot u_i\trr(x\trr m)\\
&=&v_i\ot ((x\trr m)\trl u_i+ u_i\trr(x\trr m)),
\end{eqnarray*}
the right hand side is equal to
\begin{eqnarray*}
&&-u_i\ot x\trr(m\trl v_i)+xv_i\ot u_i\trr m+v_i\ot (u_ix)\trr m-xv_i\ot u_i\trr m\\
&=&-u_i\ot x\trr(m\trl v_i)+v_i\ot (u_ix)\trr m\\
&=&v_i\ot x\trr(m\trl u_i)+v_i\ot (u_ix)\trr m\\
&\overset{ \eqref{(12)}}{=}&v_i\ot (x\trr(m\trl u_i)+ u_i\trr(x\trr m)-(u_i\trr m)\trl x+u_i\trr(m\trl x))\\
&\overset{ \eqref{(11)}}{=}&v_i\ot ((x\trr m)\trl u_i+ u_i\trr(x\trr m)+(m\trl x)\trl u_i-m\trl(xu_i)+(m\trl u_i)\trl x-m\trl(u_ix))\\
&\overset{ \eqref{(13)}}{=}&v_i\ot ((x\trr m)\trl u_i+ u_i\trr(x\trr m)).
\end{eqnarray*}

For the equation (HM6), the left hand side is equal to
\begin{eqnarray*}
&&\phi(m\trl x)+\tau\psi(m\trl x)\\
&=&-u_i\ot (m\trl x)\trl v_i+v_i\ot u_i\trr(m\trl x)\\
&=&v_i\ot (m\trl x)\trl u_i+v_i\ot u_i\trr(m\trl x)\\
&=&v_i\ot ((m\trl x)\trl u_i+ u_i\trr(m\trl x)),
\end{eqnarray*}
the right hand side is equal to
\begin{eqnarray*}
&&u_ix\ot m\trl v_i-u_i\ot m\trl(xv_i)-u_ix\ot m\trl v_i+v_i\ot (u_i\trr m)\trl x\\
&=&-u_i\ot m\trl(xv_i)+v_i\ot (u_i\trr m)\trl x\\
&=&v_i\ot (m\trl(xu_i)+ (u_i\trr m)\trl x)\\
&\overset{ \eqref{(11)}}{=}&v_i\ot ((m\trl x)\trl u_i+u_i\trr(m\trl x)+m\trl(xu_i)-(m\trl x)\trl u_i+ m\trl(u_ix)-(m\trl u_i)\trl x)\\
&\overset{ \eqref{(13)}}{=}&v_i\ot ((m\trl x)\trl u_i+u_i\trr(m\trl x)).
\end{eqnarray*}
Thus we obtain that the equations (HM1)--(HM6) are hold.
\end{proof}

\begin{theorem}\label{r-bialg2}
 Let $(A, \cdot, r)$ be a quasitriangular  alternative bialgebra. Then $A$ becomes a braided  alternative bialgebra over itself with $M=A$ and $\phi: M \rightarrow A \otimes M$ and $\psi: M \rightarrow M \otimes A$ are given by
 \begin{equation}
\phi(a): =-\sum_{i} u_{i} \otimes a v_{i} ,\quad \psi(a): = \sum_{i} u_{i}a \otimes v_{i}.
\end{equation}
\end{theorem}

\begin{proof}For braided terms in (BB1), we have
\begin{eqnarray*}
&&(a\loi\trr b)\ot a\loo+(a\lmi\trr b)\ot a\loo-a\loo\ot (b\trl a\loi)+b\loo\ot (a\trl b\lmi)\\
&&+b\loo\ot(b\lmi\trr a)-(a\trl b\loi)\ot b\loo\\
&=&-u_ib\ot av_i+v_ib\ot u_ia+av_i\ot bu_i+u_ib\ot av_i+u_ib\ot v_ia+au_i\ot bv_i\\
&=&v_ib\ot u_ia+av_i\ot bu_i-v_ib\ot u_ia-av_i\ot bu_i\\
&=&0.
\end{eqnarray*}
For braided terms in (BB2), we have
\begin{eqnarray*}
&&a\loo\ot(b\trl a\lmi)+(b\trl a\lmi)\ot a\loo+b\loo\ot(b\loi\trr a)+(b\loi\trr a)\ot b\loo\\
&=&u_ia\ot bv_i+bv_i\ot u_ia-bv_i\ot u_ia-u_ia\ot bv_i\\
&=&0.
\end{eqnarray*}
Thus $(A, \cdot, r)$ be an (braided) alternative bialgebra in the category of alternative Hopf bimodule over itself.
\end{proof}

In the above Theorem \ref{r-bialg2}, what we have obtained is a braided   alternative bialgebra with zero braided terms.
It is an open question for us whether there exists a braided  alternative bialgebra with nonzero braided terms coming from  a quasitriangular  alternative bialgebra.

\subsection{From braided alternative bialgebras to braided Lie bialgebras}

We can construct Lie bialgebras from alternative bialgebras as follows.
\begin{definition}
Let $(A, \cdot, \Delta)$ be an alternative bialgebra. Then $(A, [,], \delta)$ is a Lie bialgebra by $[a, b]=ab-ba$ and $\delta(a)=\Delta(a)-\tau\Delta(a)$ for any $a, b\in A$ if and only if
\begin{eqnarray}
-a\lii b\ot a\li+a\lii\ot ba\li=b\li\ot b\lii a-ab\li\ot b\lii.
\end{eqnarray}
\end{definition}
\begin{proof}Denote $\delta(a)=a\bi\ot a\bii=a\li\ot a\lii-a\lii\ot a\li$ and $\Delta(a)=a\li\ot a\lii$. $(A, [,], \delta)$ is a Lie bialgebra
if the following identity holds for any $a, b\in A$
\begin{eqnarray*}
\delta ( [a, b ] )=&[ a, b_{[1]} ] \otimes  b_{[2]}+b\bi \otimes [ a, b_{[2]}]+ [a\bi,  b] \otimes  a_{[2]}+ a\bi \otimes [a_{[2]},  b].
\end{eqnarray*}
Then by the define above, we have the left hand side equal to
\begin{eqnarray*}
&&\delta ( [a, b ] )\\
&=&\Delta([a, b])-\tau\Delta([a, b])\\
&=&\Delta(ab)-\Delta(ba)-\tau\Delta(ab)+\tau\Delta(ba),
\end{eqnarray*}
and the right hand side equal to
\begin{eqnarray*}
&&[ a, b_{[1]} ] \otimes  b_{[2]}+b\bi \otimes [ a, b_{[2]}]+ [a\bi,  b] \otimes  a_{[2]}+ a\bi \otimes [a_{[2]},  b]\\
&=&ab_{[1]}\ot b_{[2]}-b_{[1]}a\ot b_{[2]}+b_{[1]}\ot ab_{[2]}-b_{[1]}\ot b_{[2]}a\\
&&+a_{[1]}b\ot a_{[2]}- b a_{[1]}\ot a_{[2]}+a_{[1]} \otimes a_{[2]}b-a_{[1]} \otimes ba_{[2]}\\
&=&ab\li\ot b\lii-ab\lii\ot b\li-b\li a\ot b\lii+b\lii a\ot b\li+b_{1}\ot ab_{2}-b_{2}\ot ab_{1}\\
&&-b_{1}\ot b_{2}a+b_{2}\ot b_{1}a+a_{1}b\ot a_{2}-a_{2}b\ot a_{1}- b a_{1}\ot a_{2}+b a_{2}\ot a_{1}\\
&&+a_{1} \otimes a_{2}b-a_{2} \otimes a_{1}b-a_{1} \otimes ba_{2}+a_{2} \otimes ba_{1}\\
&=&\Delta(ab)-\Delta(ba)-\tau\Delta(ab)+\tau\Delta(ba)\\
&&+3(ab\li\ot b\lii +b\lii a\ot b\li-b\lii\ot ab\li-b\li\ot b\lii a-a\lii b\ot a\li\\
&&+a\lii\ot ba\li-b a\li\ot a\lii+a\li\ot a\lii b)\\
&=&\Delta(ab)-\Delta(ba)-\tau\Delta(ab)+\tau\Delta(ba)\\
&&+6(-a\lii b\ot a\li+a\lii\ot ba\li-b\li\ot b\lii a+ab\li\ot b\lii)\\
&=&\Delta(ab)-\Delta(ba)-\tau\Delta(ab)+\tau\Delta(ba).
\end{eqnarray*}
Thus the proof is complicated.
\end{proof}

\begin{definition}Let $H$ be simultaneously a Lie alegbra and a Lie coalgebra.
 If $V$ is  a left $H$-Lie module and left $H$-Lie comodule,  and satisfying the following condition:
\begin{enumerate}
\item[] $\phi (x\trr v)= [\alpha_H(x), v_{(-1)}] \otimes v_{(0)}+v_{(-1)} \otimes x \trr v_{(0)}+x\li \otimes x_{2} \trr v,$
\end{enumerate}
\noindent then $(V,\al_V)$is called a left Yetter-Drinfeld  module over $H$.
\end{definition}
We denote  the  category of Yetter-Drinfeld  modules over $H$ by ${}^{H}_{H}\mathcal{M}$.

\begin{definition}(\cite{Som,Maj})
Let $H$ be a  Lie algebra and Lie coalgebra and $A$ is a left Yetter-Drinfeld module over $H$, we call $A$ a \emph{braided Lie
bialgebra} in ${}^{H}_{H}\mathcal{M}$, if the following condition is
satisfied
\begin{eqnarray}
 \notag\delta ( [a, b ] )&=& [ a, b_{[1]} ] \otimes  b\bii+y\bi \otimes [ a, b\bii]+ [a\bi,  b] \otimes  a\bii+ a\bi \otimes [a_{[2]},  b]\\
&&+(a_{[-1]} \trr b) \otimes  a_{[0]}+ b_{[0]}\otimes  (b_{[-1]} \trr  a)-  (b_{[-1]} \trr   a) \otimes  b_{[0]} - a_{[0]}\otimes  (a_{[-1]}\trr  b).
\end{eqnarray}
\end{definition}
Next, we will prove that we can also construct a braided Lie bialgebra from a braided  alternative bialgebra.

\begin{theorem}
Let $H$ be an alternative bialgebra. If $(A, ~\cdot, ~\Delta)$ is a braided alternative bialgebra in ${}^{H}_{H}\mathcal{M}{}^{H}_{H}$ . Define the commutative bracket $[-, -]=\mu\circ(\id-\tau)$ and the cocommutative cobracket $\delta=(\id-\tau)\circ\Delta$, then we have  $(A,~[-, -], ~\delta)$ is a braided Lie bialgebra if and only if
\begin{eqnarray}
\label{(44)}\notag -a\lii b\ot a\li+a\lii\ot ba\li-(a\lmi\trr b)\ot a\loo+a\loo\ot(b\trl a\loi)\\
=b\li\ot b\lii a-ab\li\ot b\lii+b\loo\ot(b\lmi\trr a)-(a\trl b\loi)\ot b\loo.
\end{eqnarray}
\end{theorem}
\begin{proof}Firstly, it is easy to get that $A$ is obviously a Lie algebra and a Lie coalgebra by the new
multiplication and comultiplication above. In order to proof $A$ is a braided Lie bialgebra, we need to define
the left $H$-module and left $H$-comodule over $A$ by:
\begin{align*}
&\trr_{L}=\trr-\trl: H\ot A \to A  , ~~~\phi_{L}=\phi-\tau\psi: A\to H\ot A; ~\phi_{L}(a)=a_{[-1]}\ot a_{[0]},
\end{align*}
that is
\begin{align*}
&x\trr_{L} a=x\trr a- a\trl x, ~~~~\phi_{L}(a)=a_{[-1]}\ot a_{[0]}=a_{(-1)}\ot a_{(0)}-a_{(1)}\ot a_{(0)},
\end{align*}
 for any $x, y$ in $A$. All we need to do is to prove that
\begin{eqnarray}\label{bialg11}
\delta ( [a, b ] )&=& [ a, b_{[1]} ] \otimes  b_{[2]}+b_{[1]}\otimes [ a, b_{[2]}]+ [a_{[1]},  b] \otimes  a_{[2]}+ a_{[1]} \otimes [a_{[2]},  b]\\
 \notag&&+(a_{[-1]} \trr_{L}  b) \otimes  a_{[0]}+ b_{[0]}\otimes  (b_{[-1]} \trr_{L}  a)-  (b_{[-1]} \trr_{L}   a) \otimes  b_{[0]} - a_{[0]}\otimes  (a_{[-1]}\trr_{L}  b).
\end{eqnarray}
For direct calculations, we have  the left hand side equal to
\begin{eqnarray*}
\delta ( [a, b ] )&=&\delta(ab-ba)\\
&=&\Delta(ab)-\Delta(ba)-\tau\Delta(ab)+\tau\Delta(ba),
\end{eqnarray*}
and the right hand side equal to
\begin{eqnarray*}
&&[ a, b\bi ] \otimes  b\bii+b\bi\otimes [ a, b\bii]+ [a\bi,  b] \otimes  a\bii+ a\bi \otimes [a\bii,  b]\\
&&+(a_{[-1]} \trr_{L}  b) \otimes  a_{[0]}+ b_{[0]}\otimes  (b_{[-1]} \trr_{L}  a)-  (b_{[-1]} \trr_{L}   a) \otimes  b_{[0]} - a_{[0]}\otimes  (a_{[-1]}\trr_{L}  b)\\
&=&ab\bi\ot b\bii-b\bi a\ot b\bii+b\bi\ot ab\bii-b\bi\ot b\bii a\\
&&+a\bi b\ot a\bii- b a\bi\ot a\bii+a\bi \otimes a\bii b-a\bi \otimes b a\bii\\
&&+(a_{(-1)} \trr_{L}  b) \otimes  a_{(0)}-(a_{(1)} \trr_{L}  b) \otimes  a_{(0)}
+ b_{(0)}\otimes  (b_{(-1)} \trr_{L}  a)-b_{(0)}\otimes  (b_{(1)} \trr_{L}  a)\\
&&-  (b_{(-1)} \trr_{L}   a) \otimes  b_{(0)}+ (b_{(1)} \trr_{L}   a) \otimes  b_{(0)} - a_{(0)}\otimes  (a_{(-1)}\trr_{L}  b)+a_{(0)}\otimes  (a_{(1)}\trr_{L}  b)\\
&=&ab\li\ot b\lii-ab\lii\ot b\li-b\li a\ot b\lii+b\lii a\ot b\li+b_{1}\ot ab_{2}-b_{2}\ot ab_{1}\\
&&-b_{1}\ot b_{2}a+b_{2}\ot b_{1}a+a_{1}b\ot a_{2}-a_{2}b\ot a_{1}- b a_{1}\ot a_{2}+b a_{2}\ot a_{1}\\
&&+a_{1} \otimes a_{2}b-a_{2} \otimes a_{1}b-a_{1} \otimes ba_{2}+a_{2} \otimes ba_{1}+(a_{(-1)} \trr  b) \otimes  a_{(0)}-(b\trl a_{(-1)}) \otimes  a_{(0)}\\
&&-(a_{(1)}\trr b) \otimes  a_{(0)}+(b\trl a_{(1)}) \otimes  a_{(0)}
+ b_{(0)}\otimes  (b_{(-1)} \trr a)-b_{(0)}\otimes (a\trl b_{(-1)}) \\
&&-b_{(0)}\otimes  (b_{(1)} \trr  a)+b_{(0)}\otimes (a\trl b_{(1)})
-  (b_{(-1)} \trr  a) \otimes  b_{(0)}+(a\trl b_{(-1)}) \otimes  b_{(0)}\\
&&+ (b_{(1)} \trr a) \otimes  b_{(0)}-(a\trl b_{(1)})\otimes  b_{(0)}
 - a_{(0)}\otimes  (a_{(-1)}\trr b)+a_{(0)}\otimes  (b\trl a_{(-1)})\\
&& +a_{(0)}\otimes  (a_{(1)}\trr   b)-a_{(0)}\otimes (b\trl a_{(1)})\\
&=&\Delta(ab)-\Delta(ba)-\tau\Delta(ab)+\tau\Delta(ba)\\
&&-2a_{1} b \otimes a_{2} -2a\lii b\ot a\li+2a\lii\ot b a\li-2b\li\ot a b\lii -2b\li\ot b\lii a+2a b\li\ot b\lii\\
&&-2(a\loi\trr b)\ot a\loo-2(a\lmi\trr b)\ot a\loo+2a\loo\ot (b\trl a\loi)\\
&&-2b\loo\ot (a\trl b\lmi)-2b\loo\ot(b\lmi\trr a)+2(a\trl b\loi)\ot b\loo\\
&&+b\li\ot a b\lii+a b\lii\ot b\li+a\li b\ot a\lii+a\lii\ot a\li b\\
&&+b\loo\ot(a\trl b\lmi)+(a\trl b\lmi)\ot b\loo+a\loo\ot(a\loi\trr b)+(a\loi\trr b)\ot a\loo\\
&&-a\li\ot b a\lii-b a\lii\ot a\li-b\li a\ot b\lii-b\lii\ot b\li a\\
&&-a\loo\ot(b\trl a\lmi)-(b\trl a\lmi)\ot a\loo-b\loo\ot(b\loi\trr a)-(b\loi\trr a)\ot b\loo\\
&&+2b_{1} a \otimes b_{2} +2b\lii a\ot b\li-2b\lii\ot a b\li+2a\li\ot b a\lii +2a\li\ot a\lii b-2b a\li\ot a\lii\\
&&+2(b\loi\trr a)\ot b\loo+2(b\lmi\trr a)\ot b\loo-2b\loo\ot (a\trl b\loi)\\
&&+2a\loo\ot (b\trl a\lmi)+2a\loo\ot(a\lmi\trr b)-2(b\trl a\loi)\ot a\loo\\
&&+ab\li\ot b\lii-ab\lii\ot b\li-b\li a\ot b\lii+b\lii a\ot b\li+b_{1}\ot ab_{2}-b_{2}\ot ab_{1}\\
&&-b_{1}\ot b_{2}a+b_{2}\ot b_{1}a+a_{1}b\ot a_{2}-a_{2}b\ot a_{1}- b a_{1}\ot a_{2}+b a_{2}\ot a_{1}\\
&&+a_{1} \otimes a_{2}b-a_{2} \otimes a_{1}b-a_{1} \otimes ba_{2}+a_{2} \otimes ba_{1}+(a_{(-1)} \trr  b) \otimes  a_{(0)}-(b\trl a_{(-1)}) \otimes  a_{(0)}\\
&&-(a_{(1)}\trr b) \otimes  a_{(0)}+(b\trl a_{(1)}) \otimes  a_{(0)}
+ b_{(0)}\otimes  (b_{(-1)} \trr a)-b_{(0)}\otimes (a\trl b_{(-1)}) \\
&&-b_{(0)}\otimes  (b_{(1)} \trr  a)+b_{(0)}\otimes (a\trl b_{(1)})
-  (b_{(-1)} \trr  a) \otimes  b_{(0)}+(a\trl b_{(-1)}) \otimes  b_{(0)}\\
&&+ (b_{(1)} \trr a) \otimes  b_{(0)}-(a\trl b_{(1)})\otimes  b_{(0)}
 - a_{(0)}\otimes  (a_{(-1)}\trr b)+a_{(0)}\otimes  (b\trl a_{(-1)})\\
&& +a_{(0)}\otimes  (a_{(1)}\trr   b)-a_{(0)}\otimes (b\trl a_{(1)})\\
&=&\Delta(ab)-\Delta(ba)-\tau\Delta(ab)+\tau\Delta(ba)\\
&&+3(ab\li\ot b\lii +b\lii a\ot b\li-b\lii\ot ab\li-b\li\ot b\lii a-a\lii b\ot a\li-b a\li\ot a\lii+a\li\ot a\lii b\\
&&+a\lii\ot ba\li-(a\lmi\trr b)\ot a\loo+a\loo\ot(b\trl a\loi)-b\loo\ot(b\lmi\trr a)+(a\trl b\loi)\ot b\loo\\
&&+(b\lmi\trr a)\ot b\loo-b\loo\ot (a\trl b\loi)+a\loo\ot(a\lmi\trr b)-(b\trl a\loi)\ot a\loo)\\
&=&\Delta(ab)-\Delta(ba)-\tau\Delta(ab)+\tau\Delta(ba)\\
&&+6(-a\lii b\ot a\li+a\lii\ot ba\li-b\li\ot b\lii a+ab\li\ot b\lii-(a\lmi\trr b)\ot a\loo+a\loo\ot(b\trl a\loi)\\
&&-b\loo\ot(b\lmi\trr a)+(a\trl b\loi)\ot b\loo)\\
&=&\Delta(ab)-\Delta(ba)-\tau\Delta(ab)+\tau\Delta(ba).
\end{eqnarray*}
Thus \eqref{bialg11} holds and the proof is completed.
\end{proof}

\section{Cocycle bicrossproduct of  braided alternative bialgebras}
\subsection{Matched pair of braided alternative bialgebras}

In this section, we construct alternative bialgebra from the double cross biproduct of a matched pair of braided alternative bialgebras.

Let $A, H$ be both alternative  algebras and alternative  coalgebras.   For any $a, b\in A$, $x, y\in H$,  we denote maps
\begin{align*}
&\ppr: H \otimes A \to A,\quad \ppl: A\otimes H\to A,\\
&\trr: A\otimes H \to H,\quad \trl: H\otimes A \to H,\\
&\phi: A \to H \otimes A,\quad  \psi: A \to A\otimes H,\\
&\rho: H  \to A\otimes H,\quad  \gamma: H \to H \otimes A,
\end{align*}
by
\begin{eqnarray*}
&& \ppr (x \otimes a) = x \ppr a, \quad \ppl(a\otimes x) = a \ppl x, \\
&& \trr (a \otimes x) = a \trr x, \quad \trl(x \otimes a) = x \triangleleft a, \\
&& \phi (a)=\sum a\lmoi\ot a\loo, \quad \psi (a) = \sum a\loo \ot a\lmi,\\
&& \rho (x)=\sum x\boi\ot x\boo, \quad \gamma (x) = \sum x\boo \ot x\bi.
\end{eqnarray*}

\begin{definition}(\cite{NB})
A \emph{matched pair} of alternative  algebras is a system $(A, \, {H},\, \trl, \, \trr, \, \ppl, \, \ppr)$ consisting
of two alternative  algebras $A$ and ${H}$ and four bilinear maps $\triangleleft : {H}\otimes A\to {H}$, $\trr : {A} \otimes H
\to H$, $\ppl: A \otimes {H} \to A$, $\ppr: H\otimes {A} \to {A}$ such that $({H}, \trr, \trl)$ is an $A$-bimodule,  $(A, \ppr, \ppl)$ is an ${H}$-bimodule and satisfying the following compatibilities for all $a, b\in A$, $x, y \in {H}$:
\begin{enumerate}
\item[(AM1)] $x\ppr  (ab)+a(x\ppr b)+a\ppl(x\trl b)=(x\ppr a+a\ppl x) b+ (x\trl a+a\trr x)\ppr b$,

\item[(AM2)]$x\ppr  (ab+ b a)=(x\ppr a)b+(x\trl a)\ppr b+(x\ppr b)a+(x\trl b)\ppr a$,

\item[(AM3)] $  (ab) \ppl x+(a\ppl x)b+(a\trr x)\ppr b= a(b\ppl x+x\ppr b)+a\ppl(b\trr x+x\trl b)$,

\item[(AM4)] $ (ab+b a )\ppl x=a(b \ppl x) + a \ppl ( b \trr x)+b(a\ppl x)+b\ppl(a\trr x)$,

\item[(AM5)]$a\trr (x y)+x(a\trr y)+x\trl(a\ppl y)=(x\trl a+a\trr x)y+(x\ppr a+a\ppl x)\trr y$,

\item[(AM6)] $ a\trr (x y+ y x)= (a\trr x)y+(a\ppl x)\trr y+(a\trr y)x+(a\ppl y)\trr x$,

\item[(AM7)] $  (x y)\trl a+(x\trl a)y+(x\ppr a)\trr y=x(y\trl a+a\trr y)+x\trl(y\ppr a+a\ppl y) $,

\item[(AM8)] $(x y+y x)\trl a=x(y\trl a)+x\trl(y\ppr a)+y(x\trl a)+y\trl(x\ppr a)$.
\end{enumerate}
\end{definition}

\begin{lemma}(\cite{NB})
Let  $(A, \, {H},\, \trl, \, \trr, \, \ppl, \, \ppr)$be a matched pair of alternative  algebras.
Then $A \, \bowtie {H}: = A \oplus  {H}$, as a vector space, with the multiplication defined for any $a, b\in A$ and $x, y\in {H}$ by
\begin{equation}
(a +x) (b+ y) : = (ab+ a \ppl y + x\ppr b)+( a\trr y + x\trl b + xy )
\end{equation}
is an alternative algebra called the \emph{bicrossed product} associated to the matched pair of alternative  algebras $A$ and ${H}$.
\end{lemma}

Now we introduce the notion of matched pairs of alternative coalgebras, which is the dual version of  matched pairs of alternative algebras.

\begin{definition} A \emph{matched pair} of alternative  coalgebras is a system $(A, \, {H}, \, \phi, \, \psi, \, \rho, \, \gamma)$ consisting
of two alternative  coalgebras $A$ and ${H}$ and four bilinear maps
$\phi: {A}\to H\otimes A$, $\psi: {A}\to A \otimes H$, $\rho: H\to A\otimes {H}$, $\gamma: H \to {H} \ot {A}$
such that $({H}, \rho, \gamma)$ is an $A$-bicomodule,  $(A, \phi, \psi)$ is an ${H}$-bicomodule and satisfying the following compatibility conditions for any $a\in A$, $x\in {H}$:
\begin{enumerate}
\item[(CM1)]
$\phi(a\li)\ot a\lii+\gamma(a\loi)\ot a\loo-a_{(-1)} \otimes \Delta_{A}\left(a_{(0)}\right)\\
=-\tau_{12}\big(\psi(a\li)\ot a\lii+\rho(a\loi)\ot a\loo-a\li\ot\phi(a\lii)-a\loo\ot\gamma(a\lmi)\big)$,

\item[(CM2)]
$\Delta_{A}(a\loo)\ot a\lmi-a\li\ot\psi(a\lii)-a\loo\ot\rho(a\lmi)\\
=-\tau_{12}\big(\Delta_{A}(a\loo)\ot a\lmi-a\li\ot\psi(a\lii)-a\loo\ot\rho(a\lmi)\big)$,

\item[(CM3)]
$\Delta_{H}(x\boo)\ot x\bi-x\li\ot\gamma(x\lii)-x\boo\ot\phi(x\bi)\\
=-\tau_{12}\big(\Delta_{H}(x\boo)\ot x\bi-x\li\ot\gamma(x\lii)-x\boo\ot\phi(x\bi)\big)$,

\item[(CM4)]
$ \rho(x\li)\ot x\lii+\psi(x\boi)\ot x\boo-x\boi\ot\Delta_{H}(x\boo)\\
=-\tau_{12}\big(\gamma(x\li)\ot x\lii+\phi(x\boi)\ot x\boo-x\li\ot\rho(x\lii)-x\boo\ot\psi(x\bi)\big)$,

\item[(CM5)]
$\phi(a\li)\ot a\lii+\gamma(a\loi)\ot a\loo-a_{(-1)} \otimes \Delta_{A}(a_{(0)})\\
=-\tau_{23}\big(\phi(a\li)\ot a\lii+\gamma(a\loi)\ot a\loo-a_{(-1)} \otimes \Delta_{A}(a_{(0)})\big)$,

\item[(CM6)]
$\Delta_{A}(a\loo)\ot a\lmi-a\li\ot\psi(a\lii)-a\loo\ot\rho(a\lmi)\\
=-\tau_{23}\big(\psi(a\li)\ot a\lii+\rho(a\loi)\ot a\loo-a\li\ot\phi(a\lii)-a\loo\ot\gamma(a\lmi)\big)$,

\item[(CM7)]
$\Delta_{H}(x\boo)\ot x\bi-x\li\ot\gamma(x\lii)-x\boo\ot\phi(x\bi)\\
=-\tau_{23}\big(\gamma(x\li)\ot x\lii+\phi(x\boi)\ot x\boo-x\li\ot\rho(x\lii)-x\boo\ot\psi(x\bi)\big)$,

\item[(CM8)]
$\rho(x\li)\ot x\lii+\psi(x\boi)\ot x\boo-x\boi\ot\Delta_{H}(x\boo)\\
=-\tau_{23}\big(\rho(x\li)\ot x\lii+\psi(x\boi)\ot x\boo-x\boi\ot\Delta_{H}(x\boo)\big)$.
\end{enumerate}
\end{definition}

\begin{lemma}\label{lem1} Let $(A, \, {H}, \, \phi, \, \psi, \, \rho, \, \gamma)$ be a matched pair of alternative  coalgebras. We define $E=A\lrcoprod H$ as the vector space $A\oplus H$ with   comultiplication
$$\Delta_{E}(a)=(\Delta_{A}+\phi+\psi)(a), \quad\Delta_{E}(x)=(\Delta_{H}+\rho+\gamma)(x),$$
that is
$$\Delta_{E}(a)=\sum a\li \ot a\lii+\sum a\loi \ot a\loo+\sum a\mo\ot a\mi, $$
$$\Delta_{E}(x)=\sum x\li \ot x\lii+\sum  x\boi \ot x\boo+\sum x\boo \ot x\bi.$$
Then  $A\lrcoprod H$ is an alternative  coalgebra which is called the \emph{bicrossed coproduct} associated to the matched pair of alternative  coalgebras $A$ and $H$.
\end{lemma}

The proof of the above Lemma \ref{lem1} is omitted since it is by direct computations.
In the following of this section, we construct alternative bialgebra from the double cross biproduct of a pair of braided alternative bialgebras.
First we generalize the concept of Hopf module to the case of $A$ is not necessarily an alternative bialgebra.
But by abuse of notation, we also call it alternative Hopf module.

\begin{definition}
Let $A$ be simultaneously an alternative algebra and an alternative  coalgebra.
 If $H$ is  an   $A$-bimodule,  $A$-bicomodule and satisfying
\begin{enumerate}
\item[(HM1')]  $\rho(a \trr x)=-a\lii\ot(x\trl a\li)+x\boi\ot(a\trr x\boo)+x\boi\ot(x\boo\trl a)-a x_{[-1]} \otimes x_{[0]}$,
\end{enumerate}
\begin{enumerate}
\item[(HM2')]  $\gamma(x\trl a)=(x\boo\trl a)\ot x\bi+(x\boo\trl a)\ot x\boi-x\boo\ot a x\boi-\left(x\trl a_{1}\right) \otimes a_{2}$,
\end{enumerate}
\begin{enumerate}
\item[(HM3')]  $\rho(x\trl a)=x\boi a\ot x\boo+x\bi a\ot x\boo-x\bi\ot(a\trr x\boo)+a\li\ot(x\trl a\lii)+a\li\ot(a\lii\trr x)$,
\item[(HM4')] $\gamma(a \trr x)=(a\li\trr x)\ot a\lii+(a\lii\trr x)\ot a\li+x\boo\ot a x\bi+x\boo\ot x\bi a-\left(a\trr x_{[0]}\right)\otimes x_{[1]}$,
\end{enumerate}
\begin{enumerate}
\item[(HM5')] $\rho(a\trr x)+\tau\gamma(a\trr x)=x\boi\ot(a\trr x\boo)+a x\bi\ot x\boo+a\lii\ot (a\li\trr x)$,

\item[(HM6')] $\rho(x\trl a)+\tau\gamma(x\trl a)=a\li\ot(x\trl a\lii)+x\boi a\ot x\boo+x\bi\ot(x\boo\trl a)$,
\end{enumerate}
\noindent then $H$
is called an alternative Hopf bimodule over $A$.
\end{definition}




We denote  the  category of  alternative Hopf bimodules over $A$ by ${}^{A}_{A}\mathcal{M}{}^{A}_{A}$.

\begin{definition}
Let $A$ be an alternative algebra and  alternative coalgebra and  $H$ is an alternative Hopf bimodule over $A$. If $H$ is an alternative algebra and an alternative coalgebra in ${}^{A}_{A}\mathcal{M}^{A}_{A}$, then we call $H$ be a \emph{braided alternative  bialgebra} over $A$, if the following conditions are satisfied:
\begin{enumerate}
\item[(BB3)]
$\Delta_{H}(x y)=x\li y\ot x\lii+x\lii y\ot x\li-x\lii\ot y x\li+y\li\ot x y\lii+y\li\ot y\lii x-x y\li\ot y\lii\\
+(x\boi\trr y)\ot x\boo+(x\bi\trr y)\ot x\boo-x\boo\ot(y\trl x\boi)+y\boo\ot(x\trl y\bi)\\
+y\boo\ot(y\bi\trr x)-(x\trl y\boi)\ot y\boo,$
\item[(BB4)]
$\Delta_{H}(y x)+\tau\Delta_{H}(y x)=x\li\ot y x\lii+y x\lii\ot x\li+y\li x\ot y\lii+y\lii\ot y\li x\\
+x\boo\ot(y\trl x\bi)+(y\trl x\bi)\ot x\boo+y\boo\ot(y\boi\trr x)+(y\boi\trr x)\ot y\boo.$
\end{enumerate}
\end{definition}

\begin{definition}\label{def:dmp}
Let $A, H$ be both alternative  algebras and alternative  coalgebras. If  the following conditions hold:
\begin{enumerate}
\item[(DM1)]  $\phi(a b)=-a\lmi\ot b a\loo+b\loi\ot a b\loo+b\loi\ot b\loo a\\
+(a\loi\trl b)\ot a\loo+(a\lmi\trl b)\ot a\loo-(a\trr b\loi)\ot b\loo$,
\item[(DM2)] $\psi(a b)=a\loo b\ot a\lmi+a\loo b\ot a\loi-a b\loo\ot b\lmi\\
-a\loo\ot (b\trr a\loi)+b\loo\ot(a\trr b\lmi)+b\loo\ot(b\lmi\trl a)$,
\item[(DM3)] $\rho(x y)=(x\bi\ppl y)\ot x\boo+(x\boi\ppl y)\ot x\boo-x_{[1]} \otimes y x_{[0]}  \\
-\left(x \ppr y_{[-1]}\right) \otimes y_{[0]}+y\boi\ot x y\boo+y\boi\ot y\boo x$,
\item[(DM4)] $\gamma(x y)=x\boo y\ot x\bi+x\boo y\ot x\boi-x_{[0]}\otimes (y\ppr x\boi)\\
-xy_{[0]}\otimes y_{[1]}+y\boo\ot (x\ppr y\bi)+y\boo\ot(y\bi\ppl x)$,
\item[(DM5)] $\Delta_{A}(x \ppr b)=(x\boo\ppr b)\ot x\bi+(x\boo\ppr b)\ot x\boi-x\bi\ot(b\ppl x\boo)\\
-\left(x \ppr b_{1}\right) \otimes b_{2}+b\li\ot(x\ppr b\lii)+b\li\ot(b\lii\ppl x)$,
\item[(DM6)] $\Delta_{A}(a\ppl y)=(a\li\ppl y)\ot a\lii+\left(a_{2} \ppl y\right)\ot a\li-a\lii\ot(y\ppr a\li)\\
-\left(a\ppl y_{[0]}\right) \otimes y_{[1]}+y\boi\ot(a\ppl y\boo)+y\boi\ot(y\boo\ppr a)$,
\item[(DM7)] $\Delta_{H}(a \trr y)=(a\loo\trr y)\ot a\lmi+(a\loo\trr y)\ot a\loi-a\lmi\ot(y\trl a\loo)\\
+y\li\ot(a\trr y\lii)+y\li\ot(y\lii\trl a)-\left(a \trr y_{1}\right) \otimes y_{2}$,
\item[(DM8)] $\Delta_{H}(x \trl b)=(x\li\trl b)\ot x\lii+(x\lii\trl b)\ot x\li-x\lii\ot(b\trr x\li)\\
+b\loi\ot(x\trl b\loo)+b\loi\ot(b\loo\trr x)-\left(x\trl b_{(0)}\right) \otimes b_{(1)}$,
\item[(DM9)]
$\phi(x \ppr b)+\gamma(x\trl b)=(x\boo\trl b)\ot x\bi+(x\boo\trl b)\ot x\boi$\\
$-x\lii\ot(b\ppl x\li)-x\boo\ot b x\boi+b\loi\ot(x\ppr b\loo)$\\
$-\left(x\trl b_{1}\right) \otimes b_{2}+b\loi\ot(b\loo\ppl x)-x b_{(-1)} \otimes b_{(0)}$,
\item[(DM10)]
$\psi(a\ppl y)+\rho(a \trr y)=(a\loo\ppl y)\ot a\lmi+(a\loo\ppl y)\ot a\loi$\\
$-a\lii\ot(y\trl a\li)-a\loo\ot y a\loi+y\boi\ot(a\trr y\boo)$\\
$-\left(a\ppl y_{1}\right) \otimes y_{2}+y\boi\ot(y\boo\trl a)-a y_{[-1]} \otimes y_{[0]}$,
\item[(DM11)]
$\psi(x \ppr b)+\rho(x\trl b)=(x\li\ppr b)\ot x\lii+x\boi b\ot x\boo+(x\lii\ppr b)\ot x\li\\
+x\bi b\ot x\boo-x\bi\ot(b\trr x\boo)+b\li\ot(x\trl b\lii)+b\loo\ot x b\lmi\\
+b\li\ot(b\lii\trr x)+b\loo\ot b\lmi x-(x \ppr b_{(0)}) \otimes b_{(1)}$,
\item[(DM12)]
$\phi(a\ppl y)+\gamma(a \trr y)=(a\li\trr y)\ot a\lii+a\loi y\ot a\loo+(a\lii\trr y)\ot a\li\\
+a\lmi y\ot a\loo-a\lmi\ot (y\ppr a\loo)+y\li\ot(a\ppl y\lii)+y\boo\ot a y\bi\\
+y\li\ot(y\lii\ppr a)+y\boo\ot y\bi a-\left(a\trr y_{[0]}\right)\otimes y_{[1]}$,
\item[(DM13)]  $\phi(b a)+\tau\psi(b a)=a\loi\ot b a\loo+(b\trr a\lmi)\ot a\loo\\
+(b\loi\trl a)\ot b\loo+b\lmi\ot b\loo a$,
\item[(DM14)] $\psi(b a)+\tau\phi(b a)=a\loo\ot(b\trr a\lmi)+b a\loo \ot a\loi\\
+ b\loo a\ot b\lmi+b\loo\ot(b\loi\trl a)$,
\item[(DM15)] $\rho(y x)+\tau\gamma(y x)=x\boi\ot y x\boo+(y\ppr x\bi)\ot x\boo\\
+(y\boi\ppl x)\ot y\boo+y\bi\ot y\boo x$,
\item[(DM16)] $\gamma(y x)+\tau\rho(y x)=y x\boo\ot x\boi+y\boo x\ot y\bi\\
+y\boo\ot(y\boi\ppl x)+x\boo\ot(y\ppr x\bi)$,
\item[(DM17)] $\Delta_{A}(b \ppl x)+\tau\Delta_{A}(b \ppl x)=x\boi\ot(b\ppl x\boo)+(b\ppl x\boo)\ot x\boi\\
+(b\li\ppl x)\ot b\lii+b\lii\ot(b\li\ppl x)$,
\item[(DM18)] $\Delta_{A}(y\ppr a)+\tau\Delta_{A}(y\ppr a)=a\li\ot(y\ppr a\lii)+(y\ppr a\lii)\ot a\li\\
+(y\boo\ppr a)\ot y\bi+y\bi\ot(y\boo\ppr a)$,
\item[(DM19)] $\Delta_{H}(y\trl a)+\tau\Delta_{H}(y\trl a)=a\loi\ot(y\trl a\loo)+(y\trl a\loo)\ot a\loi\\
+(y\li\trl a)\ot y\lii+y\lii\ot(y\li\trl a)$,
\item[(DM20)] $\Delta_{H}(b \trr x)+\tau\Delta_{H}(b \trr x)=x\li\ot(b\trr x\lii)+(b\trr x\lii)\ot x\li\\
+(b\loo\trr x)\ot b\lmi+b\lmi\ot (b\loo\trr x)$,
\item[(DM21)]
$\phi(y \ppr a)+\tau\psi(y \ppr a)+\gamma(y\trl a)+\tau\rho(y\trl a)\\
=a\loi\ot(y\ppr a\loo)+(y\trl a\lii)\ot a\li+y a\lmi\ot a\loo\\
+(y\boo\trl a)\ot y\bi+y\lii\ot(y\li\ppr a)+y\boo\ot y\boi a$,
\item[(DM22)]
$\phi(b \ppl x)+\tau\psi(b \ppl x)+\gamma(b\trr x)+\tau\rho(b\trr x)\\
=x\li\ot(b\ppl x\lii)+x\boo\ot b x\bi+(b\trr x\boo)\ot x\boi\\
+(b\li\trr x)\ot b\lii+b\loi x\ot b\loo+b\lmi\ot(b\loo\ppl x)$,
\item[(DM23)]
$\psi(y \ppr a)+\tau\phi(y \ppr a)+\rho(y\trl a)+\tau\gamma(y\trl a)\\
=a\li\ot(y\trl a\lii)+a\loo\ot y a\lmi+(y\ppr a\loo)\ot a\loi\\
+(y\li\ppr a)\ot y\lii+y\boi a\ot y\boo+y\bi\ot(y\boo\trl a)$,
\item[(DM24)]
$\psi(b \ppl x)+\tau\phi(b \ppl x)+\rho(b\trr x)+\tau\gamma(b\trr x)\\
=x\boi\ot(b\trr x\boo)+(b\ppl x\lii)\ot x\li+b x\bi\ot x\boo\\
+(b\loo\ppl x)\ot b\lmi+b\lii\ot(b\li\trr x)+b\loo\ot b\loi x$,
\end{enumerate}
\noindent then $(A, H)$ is called a \emph{double matched pair}.
\end{definition}

\begin{theorem}\label{main1}
Let $(A, H)$ be a double matched pair of alternative  algebras and alternative  coalgebras,
$A$ is  a  braided alternative bialgebra in
${}^{H}_{H}\mathcal{M}^{H}_{H}$, $H$ is  a braided alternative bialgebra in
${}^{A}_{A}\mathcal{M}^{A}_{A}$. If we define the double cross biproduct of $A$ and
$H$, denoted by $A\lrbiprod H$, $A\lrbiprod H=A\bowtie H$ as an alternative
algebra, $A\lrbiprod H=A\lrcoprod H$ as an alternative  coalgebra, then
$A\lrbiprod H$ become an alternative bialgebra if and only if  $(A, H)$ forms a double matched pair.
\end{theorem}

The proof of the above Theorem \ref{main1} is omitted since it is a special case of Theorem \ref{main2} in next subsection. 

\subsection{Cocycle bicrossproduct  alternative bialgebras}\label{subsetion-4.2}

In this section, we construct cocycle bicrossproduct alternative bialgebras, which is a generalization of double cross biproduct.

Let $A, H$ be both alternative  algebras and alternative  coalgebras.   For $a, b\in A$, $x, y\in H$,  we denote maps
\begin{align*}
&\sigma: H\otimes H \to A, \quad \theta: A\otimes A \to H,\\
&P: A  \to H\otimes H, \quad  Q: H \to A\otimes A,
\end{align*}
by
\begin{eqnarray*}
&& \sigma (x,y)  \in  A, \quad \theta(a, b) \in H,\\
&& P(a)=\sum a\ppi\ot  a\pii, \quad Q(x) = \sum x\qi \ot x\qii.
\end{eqnarray*}

A bilinear map $\si: H\ot H\to A$ is called a cocycle on $H$ if
\begin{enumerate}
\item[(CC1)] $\sigma(x y, z)+\sigma(x, y)\ppl z-x \ppr \sigma(y, z)-{\sigma}(x, y z)\\
=-\sigma(y, x)\ppl z-\sigma(y x, z)+y\ppr\sigma(x, z)+\sigma(y, xz),$
\item[(CC2)] $\sigma(x y, z)+\sigma(x, y)\ppl z-x \ppr \sigma(y, z)-{\sigma}(x, y z)\\
=-\sigma(x, z)\ppl y-\sigma(x z, y)+x\ppr\sigma(z, y)+\sigma(x, zy).$
\end{enumerate}

A bilinear map $\theta: A\ot A\to H$ is called a cocycle on $A$ if
\begin{enumerate}
\item[(CC3)] $\theta(a b, c)+\theta(a, b) \triangleleft c-a\trr \theta(b, c)-\theta(a, b c)\\
=-\theta(b a, c)-\theta(b, a) \triangleleft c+b\trr \theta(a, c)+\theta(b, a c),$
\item[(CC4)] $\theta(a b, c)+\theta(a, b) \triangleleft c-a\trr \theta(b, c)-\theta(a, b c)\\
=-\theta(a c, b)-\theta(a, c) \triangleleft b+a\trr \theta(c, b)+\theta(a, cb).$
\end{enumerate}

A bilinear map $P: A\to H\ot H$ is called a cycle on $A$ if
\begin{enumerate}
\item[(CC5)]  $\Delta_H(a\ppi)\ot a\pii+P(a\lmoo)\ot a\mi-a\lmoi\ot P(a\lmoo)-a\ppi\ot \Delta_H(a\pii)\\
=-\tau_{12}\big(\Delta_H(a\ppi)\ot a\pii+P(a\lmoo)\ot a\mi-a\lmoi\ot P(a\lmoo)-a\ppi\ot \Delta_H(a\pii)\big)$,
\item[(CC6)]  $\Delta_H(a\ppi)\ot a\pii+P(a\lmoo)\ot a\mi-a\lmoi\ot P(a\lmoo)-a\ppi\ot \Delta_H(a\pii)\\
=-\tau_{23}\big(\Delta_H(a\ppi)\ot a\pii+P(a\lmoo)\ot a\mi-a\lmoi\ot P(a\lmoo)-a\ppi\ot \Delta_H(a\pii)\big)$.
\end{enumerate}

A bilinear map $Q: H\to A\ot A$ is called a cycle on $H$ if
\begin{enumerate}
\item[(CC7)]  $ \Delta_A(x\qi)\ot x\qii+Q(x\boo)\ot x\bi-x\qi\ot \Delta_A(x\qii)-x\boi\ot Q(x\boo)\\
=-\tau_{12}\big(\Delta_A(x\qi)\ot x\qii+Q(x\boo)\ot x\bi-x\qi\ot \Delta_A(x\qii)-x\boi\ot Q(x\boo)\big)$,
\item[(CC8)]  $ \Delta_A(x\qi)\ot x\qii+Q(x\boo)\ot x\bi-x\qi\ot \Delta_A(x\qii)-x\boi\ot Q(x\boo)\\
=-\tau_{23}\big(\Delta_A(x\qi)\ot x\qii+Q(x\boo)\ot x\bi-x\qi\ot \Delta_A(x\qii)-x\boi\ot Q(x\boo)\big)$.
\end{enumerate}

In the following definitions, we introduce   the concept of cocycle
alternative  algebras and cycle alternative  coalgebras, which are  in fact not really
ordinary alternative  algebras and alternative  coalgebras, but generalized ones.

\begin{definition}
(i): Let $\si$ be a cocycle on a vector space  $H$ equipped with multiplication $H \ot H \to H$, satisfying the
following cocycle associative identities:
\begin{enumerate}
\item[(CC9)] $(x y) z+\sigma(x, y) \trr z-x(y z)-x \triangleleft \sigma(y, z)\\
=-(y x) z-\sigma(y, x) \trr z+y(x z)+y \triangleleft \sigma(x, z)$,
\item[(CC10)] $(x y) z+\sigma(x, y) \trr z-x(y z)-x \triangleleft \sigma(y, z)\\
=-(x z) y-\sigma(x, z) \trr y+x(z y)+x \triangleleft \sigma(z, y)$.
\end{enumerate}
Then  $H$ is called a  cocycle $\si$-alternative algebra which is denoted by $(H, \sigma)$.

(ii): Let $\theta$ be a cocycle on a vector space $A$  equipped with a multiplication $A \ot A \to A$, satisfying the
following cocycle associative identities:
\begin{enumerate}
\item[(CC11)] $(a b) c+\theta(a, b) \ppr c-a(b c)-a\ppl \theta(b, c)\\
=-(b a) c-\theta(b, a) \ppr c+b(a c)+b\ppl \theta(a, c)$,
\item[(CC12)] $(a b) c+\theta(a, b) \ppr c-a(b c)-a\ppl \theta(b, c)\\
=-(a c) b-\theta(a, c) \ppr b+a(c b)+a\ppl \theta(c, b)$.
\end{enumerate}
Then  $A$ is called a cocycle  $\theta$-alternative  algebra which is denoted by $(A, \theta)$.

(iii) Let $P$ be a cycle on a vector space  $H$ equipped with a comultiplication $\Delta: H \to H \ot H$, satisfying the
following cycle coassociative identities:
\begin{enumerate}
\item[(CC13)] $\Delta_H(x\li)\ot x\lii+ P(x\boi) \ot x\boo-x_1\ot \Delta_H(x_2)-x\boo\ot P(x\bi)\\
=-\tau_{12}\big(\Delta_H(x\li)\ot x\lii+ P(x\boi) \ot x\boo-x_1\ot \Delta_H(x_2)-x\boo\ot P(x\bi)\big)$,
\item[(CC14)] $\Delta_H(x\li)\ot x\lii+ P(x\boi) \ot x\boo-x_1\ot \Delta_H(x_2)-x\boo\ot P(x\bi)\\
=-\tau_{23}\big(\Delta_H(x\li)\ot x\lii+ P(x\boi) \ot x\boo-x_1\ot \Delta_H(x_2)-x\boo\ot P(x\bi)\big)$.
\end{enumerate}
\noindent Then  $H$ is called a  cycle $P$-alternative coalgebra which is denoted by $(H, P)$.

(iv) Let $Q$ be a cycle on a vector space  $A$ equipped with a   commutativity  map $\Delta: A \to A \ot A$, satisfying the
following cycle coassociative identities:
\begin{enumerate}
\item[(CC15)] $\Delta_A(a\li)\ot a\lii+Q(a\lmoi)\ot a\lmoo-a\li\ot \Delta_A(a\lii)-a\mo\ot Q(a\mi)\\
=-\tau_{12}\big(\Delta_A(a\li)\ot a\lii+Q(a\lmoi)\ot a\lmoo-a\li\ot \Delta_A(a\lii)-a\mo\ot Q(a\mi)\big)$,
\item[(CC16)] $\Delta_A(a\li)\ot a\lii+Q(a\lmoi)\ot a\lmoo-a\li\ot \Delta_A(a\lii)-a\mo\ot Q(a\mi)\\
=-\tau_{23}\big(\Delta_A(a\li)\ot a\lii+Q(a\lmoi)\ot a\lmoo-a\li\ot \Delta_A(a\lii)-a\mo\ot Q(a\mi)\big)$.
\end{enumerate}
\noindent Then  $A$ is called a  cycle $Q$-alternative coalgebra which is denoted by $(A, Q)$.
\end{definition}

\begin{definition}
A  \emph{cocycle cross product system } is a pair of $\theta$-alternative algebra $A$ and $\sigma$-alternative algebra $H$,
where $\si: H\ot H\to A$ is a cocycle on $H$, $\theta: A\ot A\to H$ is a cocycle on $A$ and the following conditions are satisfied:
\begin{enumerate}
\item[(CP1)] $(ab+b a)\ppl x+\sigma(\theta(a, b), x)+\sigma(\theta(b, a), x)\\
=a(b\ppl x)+a\ppl(b\trr x)+b(a\ppl x)+b\ppl(a\trr x)$,

\item[(CP2)] $x\ppr (a b)+a(x\ppr b)+a\ppl(x\trl b)+\sigma(x, \theta(a, b))\\
=(x\ppr a+a\ppl x)b+(x\trl a+a\trr x)\ppr b$,

\item[(CP3)] $ a\ppl (x y)+a\sigma(x, y)+x\ppr(a\ppl y)+\sigma(x, a\trr y)\\
=(a\ppl x+x\ppr a)\ppl y+\sigma(a\trr x, y)+\sigma(x\trl a, y)$,

\item[(CP4)] $(x y+y x)\ppr a+(\sigma(x, y)+\sigma(y, x))a\\
=x\ppr(y\ppr a)+\sigma(x, y\trl a)+y\ppr(x\ppr a)+\sigma(y, x\trl a)$,

\item[(CP5)] $(ab +ba)\trr x+(\theta(a, b) +\theta(b, a))x\\
=a\trr (b\trr x)+\theta(a, b\ppl x)+b\trr (a\trr x)+\theta(b, a\ppl x)$,

\item[(CP6)] $x \trl (a b)+a\trr(x\trl b)+\theta(a, x\ppr b)+x\theta(a, b)\\
=(x \trl a+a\trr x) \trl b+\theta(x \ppr a, b)+\theta(a\ppl x, b)$,

\item[(CP7)] $ a \trr (x y)+x(a\trr y)+x\trl(a\ppl y)+\theta(a, \sigma(x, y))\\
=(a \trr x+x\trl a) y+(a\ppl x+x\ppr a) \trr y$,

\item[(CP8)] $(x y +y x )\trl a+\theta(\sigma(x, y), a)+\theta(\sigma(y, x), a)\\
=x \trl(y \ppr a)+x(y \trl a)+y \trl(x \ppr a)+y(x \trl a)$,

\item[(CP9)] $x \ppr (a b+ b a)+\sigma(x, \theta(a, b))+\sigma(x, \theta(b, a))\\
=(x\ppr a) b+(x\trl a) \ppr b+(x\ppr b) a+(x\trl b) \ppr a$,

\item[(CP10)] $(ab)\ppl x+(a\ppl x)b+(a\trr x)\ppr b+\sigma(\theta(a, b), x)\\
=a(b\ppl x+x\ppr b)+a\ppl (b \trr x+x\trl b)$,

\item[(CP11)] $(x y) \ppr a+(x\ppr a)\ppl y+\sigma(x\trl a, y)+\sigma(x, y)a\\
=x\ppr(y\ppr a+a\ppl y)+\sigma(x, y\trl a)+\sigma(x, a\trr y)$,

\item[(CP12)] $a\ppl (x y+ y x)+a( \sigma(x, y)+ \sigma(y, x))\\
=(a\ppl x)\ppl y+\sigma(a \trr x, y)+(a\ppl y)\ppl x+\sigma(a \trr y, x)$,

\item[(CP13)] $x \trl (a b+ b a)+x(\theta(a, b)+\theta(b, a))\\
=(x\trl a)\trl b+\theta(x\ppr a, b) +(x\trl b)\trl a+\theta(x\ppr b, a) $,

\item[(CP14)] $(a b)\trr x+\theta(a, b)x+(a\trr x)\trl b+\theta(a\ppl x, b)\\
=a\trr(b\trr x+x\trl b)+\theta(a, b\ppl x)+\theta(a, x\ppr b)$,

\item[(CP15)] $(x y)\trl a+(x\trl a)y+(x\ppr a)\trr y+\theta(\sigma(x, y), a)\\
=x(y\trl a+a\trr y)+x\trl(y\ppr a+a\ppl y)$,

\item[(CP16)] $a\trr (x y+ y x)+\theta(a, \sigma(x, y))+\theta(a, \sigma(y, x))\\
=(a\trr x)y+(a\ppl x)\trr y+(a\trr y)x+(a\ppl y)\trr x$.
\end{enumerate}
\end{definition}

\begin{lemma}
Let $(A, H)$ be  a  cocycle cross product system.
If we define $E=A_{\sigma}\#_{\theta} H$ as the vector space $A\oplus H$ with the   multiplication
\begin{align}
(a+x)(b+ y)=\big(ab+x\ppr b+a\ppl y+\sigma(x, y)\big)+ \big(xy+x\trl b+a\trr y+\theta(a, b)\big).
\end{align}
Then $E=A_{\sigma}\#_{\theta} H$ forms  an alternative algebra  which is called the cocycle cross product alternative algebra.
\end{lemma}

\begin{proof} First, we need to check the first equation
$$\begin{aligned}
&\big((a+ x) (b+ y)\big) (c+ z)-(a+ x)\big( (b+ y) (c+ z)\big)\\
  =& -\big( (b+y) (a+x)\big) (c+z)+(b+y)\big ((a+x) (c+z)\big).\\
\end{aligned}
$$
By direct computations, the left hand side is equal to
\begin{eqnarray*}
&&\big((a+ x) (b+ y)\big) (c+ z)-(a+ x)\big( (b+ y) (c+ z)\big)\\
&=&\big(a b+x \ppr b+a \ppl y+\sigma(x, y)+ x y+x\trl b+a\trr y+\theta(a, b)\big) (c+ z)\\
&&-(a+ x)(b c+y\ppr c+b \ppl z+\sigma(y, z)+ y z+y\trl c+b\trr z+\theta(b, c))\\
&=&(a b)c+(x\ppr b)c+(a\ppl y)c+\sigma(x, y)c+(xy)\ppr c+(x\trl b)\ppr c\\
&&+(a\trr y)\ppr c+\theta(a, b)\ppr c+(a b)\ppl z+(x\ppr b)\ppl z+(a\ppl y)\ppl z\\
&&+\sigma(x, y)\ppl z+\sigma( x y, z)+\sigma(x\trl b, z)+\sigma(a\trr y, z)+\sigma(\theta(a, b), z)\\
&&+(xy)z+(x\trl b)z+(a\trr y)z+\theta(a, b)z+(x y)\trl c+(x\trl b)\trl c\\
&&+(a\trr y)\trl c+\theta(a, b)\trl c+(a b)\trr z+(x \ppr b)\trr z+(a \ppl y)\trr z\\
&&+\sigma(x, y)\trr z+\theta(a b, c)+\theta(x \ppr b, c)+\theta(a \ppl y, c)+\theta(\sigma(x, y), c)\\
&&-a(b c)-a(y\ppr c)-a(b\ppl z)-a\sigma(y, z)-x\ppr( b c)\\
&&-x\ppr (y\ppr c)-x\ppr(b\ppl z)-x\ppr\sigma(y, z)-a\ppl (yz)\\
&&-a\ppl(y\trl c)-a\ppl(b\trr z)-a\ppl\theta(b, c)-\sigma(x, y z)-\sigma(x, y\trl c)\\
&&-\sigma(x, b\trr z)-\sigma(x, \theta(b, c))-x(yz)-x(y\trl c)-x(b\trr z)\\
&&-x\theta(b, c)-x\trl (b c)-x\trl(y\ppr c)-x\trl(b \ppl z)-x\trl\sigma(y,z)\\
&&-a\trr (y z)-a\trr(y\trl c)-a\trr(b\trr z)-a\trr\theta(b, c)-\theta(a, b c)\\
&&-\theta(a, y\ppr c)-\theta(a, b \ppl z)-\theta(a,\sigma(y, z)),
\end{eqnarray*}
and the right hand side is equal to
\begin{eqnarray*}
&&-\big( (b+y) (a+x)\big) (c+z)+(b+y)\big ((a+x) (c+z)\big)\\
&=&-\big(b a+y \ppr a+b \ppl x+\sigma(y, x)+ y x+y\trl a+b\trr x+\theta(b, a)\big) (c+ z)\\
&&+(b+ y)(a c+x\ppr c+a \ppl z+\sigma(x, z)+ x z+x\trl c+a\trr z+\theta(a, c))\\
&=&-(b a)c-(y\ppr a)c-(b\ppl x)c-\sigma(y, x)c-(yx)\ppr c-(y\trl a)\ppr c\\
&&-(b\trr x)\ppr c-\theta(b, a)\ppr c-(b a)\ppl z-(y\ppr a)\ppl z-(b\ppl x)\ppl z\\
&&-\sigma(y, x)\ppl z-\sigma( y x, z)-\sigma(y\trl a, z)-\sigma(b\trr x, z)-\sigma(\theta(b, a), z)\\
&&-(yx)z-(y\trl a)z-(b\trr x)z-\theta(b, a)z-(y x)\trl c-(y\trl a)\trl c\\
&&-(b\trr x)\trl c-\theta(b, a)\trl c-(b a)\trr z-(y \ppr a)\trr z-(b \ppl x)\trr z\\
&&-\sigma(y, x)\trr z-\theta(ba, c)-\theta(y \ppr a, c)-\theta(b \ppl x, c)-\theta(\sigma(y, x), c)\\
&&+b(a c)+b(x\ppr c)+b(a\ppl z)+b\sigma(x, z)+y\ppr (a c)\\
&&+y\ppr (x\ppr c)+y\ppr(a\ppl z)+y\ppr\sigma(x, z)+b\ppl (xz)\\
&&+b\ppl(x\trl c)+b\ppl(a\trr z)+b\ppl\theta(a, c)+\sigma(y, x z)+\sigma(y, x\trl c)\\
&&+\sigma(y, a\trr z)+\sigma(y, \theta(a,c))+y(xz)+y(x\trl c)+y(a\trr z)\\
&&+y\theta(a, c)+y\trl (a c)+y\trl(x\ppr c)+y\trl(a \ppl z)+y\trl\sigma(x, z)\\
&&+b\trr (x z)+b\trr(x\trl c)+b\trr(a\trr z)+b\trr\theta(a, c)+\theta(b, a c)\\
&&+\theta(b, x\ppr c)+\theta(b, a \ppl z)+\theta(b, \sigma(x, z)).
\end{eqnarray*}
Thus the two sides are equal to each other if and only if (CP1)--(CP8) hold.

Next,we check the second equation
 $$\begin{aligned}
&\big((a+ x) (b+ y)\big) (c+z)-(a+ x) \big((b+ y) (c+ z)\big)\\
  =& -\big( (a+x)(c+z)\big) (b+y)+(a+x) \big( (c+z)(b+y)\big).\\
\end{aligned}
$$
By direct computations, the left hand side is equal to
\begin{eqnarray*}
&&\big((a+ x) (b+ y)\big) (c+ z)-(a+ x)\big( (b+ y) (c+ z)\big)\\
&=&\big(a b+x \ppr b+a \ppl y+\sigma(x, y)+ x y+x\trl b+a\trr y+\theta(a, b)\big) (c+ z)\\
&&-(a+ x)(b c+y\ppr c+b \ppl z+\sigma(y, z)+ y z+y\trl c+b\trr z+\theta(b, c))\\
&=&(a b)c+(x\ppr b)c+(a\ppl y)c+\sigma(x, y)c+(xy)\ppr c+(x\trl b)\ppr c\\
&&+(a\trr y)\ppr c+\theta(a,b)\ppr c+(a b)\ppl z+(x\ppr b)\ppl z+(a\ppl y)\ppl z\\
&&+\sigma(x, y)\ppl z+\sigma( x y, z)+\sigma(x\trl b, z)+\sigma(a\trr y, z)+\sigma(\theta(a, b), z)\\
&&+(xy)z+(x\trl b)z+(a\trr y)z+\theta(a, b)z+(x y)\trl c+(x\trl b)\trl c\\
&&+(a\trr y)\trl c+\theta(a, b)\trl c+(a b)\trr z+(x \ppr b)\trr z+(a \ppl y)\trr z\\
&&+\sigma(x, y)\trr z+\theta(a b, c)+\theta(x \ppr b, c)+\theta(a \ppl y, c)+\theta(\sigma(x, y), c)\\
&&-a(b c)-a(y\ppr c)-a(b\ppl z)-a\sigma(y, z)-x\ppr (b c)\\
&&-x\ppr (y\ppr c)-x\ppr(b\ppl z)-x\ppr\sigma(y, z)-a\ppl (yz)\\
&&-a\ppl(y\trl c)-a\ppl(b\trr z)-a\ppl\theta(b, c)-\sigma(x, y z)-\sigma(x, y\trl c)\\
&&-\sigma(x, b\trr z)-\sigma(x, \theta(b, c))-x(yz)-x(y\trl c)-x(b\trr z)\\
&&-x\theta(b, c)-x\trl (b c)-x\trl(y\ppr c)-x\trl(b \ppl z)-x\trl\sigma(y, z)\\
&&-a\trr (y z)-a\trr(y\trl c)-a\trr(b\trr z)-a\trr\theta(b, c)-\theta(a, b c)\\
&&-\theta(a, y\ppr c)-\theta(a, b \ppl z)-\theta(a, \sigma(y,z)),
\end{eqnarray*}
and the right hand side is equal to
\begin{eqnarray*}
&& -\big( (a+x)(c+z)\big) (b+y)+(a+x) \big( (c+z)(b+y)\big)\\
&=&-\big(a c+x \ppr c+a \ppl z+\sigma(x, z)+ x z+x\trl c+a\trr z+\theta(a, c)\big) (b+ y)\\
&&+(a+ x)(c b+z\ppr b+c \ppl y+\sigma(z, y)+ z y+z\trl b+c\trr y+\theta(c, b))\\
&=&-(a c)b-(x\ppr c)b-(a\ppl z)b-\sigma(x, z)b-(xz)\ppr b-(x\trl c)\ppr b\\
&&-(a\trr z)\ppr b-\theta(a, c)\ppr b-(a c)\ppl y-(x\ppr c)\ppl y-(a\ppl z)\ppl y\\
&&-\sigma(x, z)\ppl y-\sigma( x z, y)-\sigma(x\trl c, y)-\sigma(a\trr z, y)-\sigma(\theta(a, c), y)\\
&&-(xz)y-(x\trl c)y-(a\trr z)y-\theta(a, c)y-(x z)\trl b-(x\trl c)\trl b\\
&&-(a\trr z)\trl b-\theta(a, c)\trl b-(a c)\trr y-(x \ppr c)\trr y-(a \ppl z)\trr y\\
&&-\sigma(x, z)\trr y-\theta(a c, b)-\theta(x \ppr c, b)-\theta(a \ppl z, b)-\theta(\sigma(x, z), b)\\
&&+a(c b)+a(z\ppr b)+a(c\ppl y)+a\sigma(z, y)+x\ppr (c b)\\
&&+x\ppr (z\ppr b)+x\ppr(c\ppl y)+x\ppr\sigma(z, y)+a\ppl (zy)\\
&&+a\ppl(z\trl b)+a\ppl(c\trr y)+a\ppl\theta(c, b)+\sigma(x, z y)+\sigma(x, z\trl b)\\
&&+\sigma(x, c\trr y)+\sigma(x, \theta(c, b))+x(zy)+x(z\trl b)+x(c\trr y)\\
&&+x\theta(c, b)+x\trl (c b)+x\trl(z\ppr b)+x\trl(c \ppl y)+x\trl\sigma(z, y)\\
&&+a\trr (z y)+a\trr(z\trl b)+a\trr(c\trr y)+a\trr\theta(c, b)+\theta(a, c b)\\
&&+\theta(a, z\ppr b)+\theta(a, c \ppl y)+\theta(a, \sigma(z,y)).
\end{eqnarray*}
Thus the two sides are equal to each other if and only if (CP9)--(CP16) hold.
\end{proof}

\begin{definition}
A  \emph{cycle cross coproduct system } is a pair of   $P$-alternative coalgebra $A$ and  $Q$-alternative coalgebra $H$ ,  where $P: A\to H\ot H$ is a cycle on $A$,  $Q: H\to A\ot A$ is a cycle over $H$ such that following conditions are satisfied:
\begin{enumerate}
\item[(CCP1)] $\phi(a\li)\ot a\lii+\gamma(a\loi)\ot a\loo-a\loi\ot \Delta_{A}(a\loo)-a\ppi\ot Q(a\pii)\\
=-\tau_{12}\big(\psi(a\li)\ot a\lii+\rho(a\loi)\ot a\loo-a\li\ot\phi(a\lii)-a\loo\ot\gamma(a\lmi)\big)$,

\item[(CCP2)] $P(a\li)\ot a\lii+\Delta_{H}(a\loi)\ot a\loo-a\loi\ot\phi(a\loo)-a\ppi\ot\gamma(a\pii)\\
=-\tau_{12}\big(P(a\li)\ot a\lii+\Delta_{H}(a\loi)\ot a\loo-a\loi\ot\phi(a\loo)-a\ppi\ot\gamma(a\pii)\big)$,

\item[(CCP3)] $\Delta_{A}(a\loo)\ot a\lmi+Q(a\ppi)\ot a\pii-a\li\ot\psi(a\lii)-a\loo\ot\rho(a\lmi)\\
=-\tau_{12}\big(\Delta_{A}(a\loo)\ot a\lmi+Q(a\ppi)\ot a\pii-a\li\ot\psi(a\lii)-a\loo\ot\rho(a\lmi)\big)$,

\item[(CCP4)] $\psi(a\loo)\ot a\lmi+\rho(a\ppi)\ot a\pii-a\li\ot P(a\lii)-a\loo\ot\Delta_{H}(a\lmi)\\
=-\tau_{12}\big(\phi(a\loo)\ot a\lmi+\gamma(a\ppi)\ot a\pii-a\loi\ot\psi(a\loo)-a\ppi\ot\rho(a\pii)\big)$,

\item[(CCP5)] $\gamma(x\boo)\ot x\bi+\phi(x\qi)\ot x\qii-x\li\ot Q(x\lii)-x\boo\ot\Delta_{A}(x\bi)\\
=-\tau_{12}\big(\rho(x\boo)\ot x\bi+\psi(x\qi)\ot x\qii-x\boi\ot\gamma(x\boo)-x\qi\ot\phi(x\qii)\big)$,

\item[(CCP6)] $\Delta_{H}(x\boo)\ot x\bi+P(x\qi)\ot x\qii-x\li\ot \gamma(x\lii)-x\boo\ot\phi(x\bi)\\
=-\tau_{12}\big(\Delta_{H}(x\boo)\ot x\bi+P(x\qi)\ot x\qii-x\li\ot \gamma(x\lii)-x\boo\ot\phi(x\bi)\big)$,

\item[(CCP7)] $Q(x\li)\ot x\lii+\Delta_{A}(x\boi)\ot x\boo-x\boi\ot \rho(x\boo)-x\qi\ot\psi(x\qii)\\
=-\tau_{12}\big(Q(x\li)\ot x\lii+\Delta_{A}(x\boi)\ot x\boo-x\boi\ot \rho(x\boo)-x\qi\ot\psi(x\qii)\big)$,

\item[(CCP8)] $\rho(x\li)\ot x\lii+\psi(x\boi)\ot x\boo-x\boi\ot\Delta_{H}(x\boo)-x\qi\ot P(x\qii)\\
=-\tau_{12}\big(\gamma(x\li)\ot x\lii+\phi(x\boi)\ot x\boo-x\li\ot\rho(x\lii)-x\boo\ot\psi(x\bi)\big)$,

\item[(CCP9)] $\phi(a\li)\ot a\lii+\gamma(a\loi)\ot a\loo-a\loi\ot \Delta_{A}(a\loo)-a\ppi\ot Q(a\pii)\\
=-\tau_{23}\big(\phi(a\li)\ot a\lii+\gamma(a\loi)\ot a\loo-a\loi\ot \Delta_{A}(a\loo)-a\ppi\ot Q(a\pii)\big)$,

\item[(CCP10)] $P(a\li)\ot a\lii+\Delta_{H}(a\loi)\ot a\loo-a\loi\ot\phi(a\loo)-a\ppi\ot\gamma(a\pii)\\
=-\tau_{23}\big(\phi(a\loo)\ot a\lmi+\gamma(a\ppi)\ot a\pii-a\loi\ot\psi(a\loo)-a\ppi\ot\rho(a\pii)\big)$,

\item[(CCP11)] $\Delta_{A}(a\loo)\ot a\lmi+Q(a\ppi)\ot a\pii-a\li\ot\psi(a\lii)-a\loo\ot\rho(a\lmi)\\
=-\tau_{23}\big(\psi(a\li)\ot a\lii+\rho(a\loi)\ot a\loo-a\li\ot\phi(a\lii)-a\loo\ot\gamma(a\lmi)\big)$,

\item[(CCP12)] $\psi(a\loo)\ot a\lmi+\rho(a\ppi)\ot a\pii-a\li\ot P(a\lii)-a\loo\ot\Delta_{H}(a\lmi)\\
=-\tau_{23}\big(\psi(a\loo)\ot a\lmi+\rho(a\ppi)\ot a\pii-a\li\ot P(a\lii)-a\loo\ot\Delta_{H}(a\lmi)\big)$,

\item[(CCP13)] $\gamma(x\boo)\ot x\bi+\phi(x\qi)\ot x\qii-x\li\ot Q(x\lii)-x\boo\ot\Delta_{A}(x\bi)\\
=-\tau_{23}\big(\gamma(x\boo)\ot x\bi+\phi(x\qi)\ot x\qii-x\li\ot Q(x\lii)-x\boo\ot\Delta_{A}(x\bi)\big)$,

\item[(CCP14)] $\Delta_{H}(x\boo)\ot x\bi+P(x\qi)\ot x\qii-x\li\ot \gamma(x\lii)-x\boo\ot\phi(x\bi)\\
=-\tau_{23}\big(\phi(x\boi)\ot x\boo+\gamma(x\li)\ot x\lii-x\li\ot\rho(x\lii)-x\boo\ot\psi(x\bi)\big)$,

\item[(CCP15)] $Q(x\li)\ot x\lii+\Delta_{A}(x\boi)\ot x\boo-x\boi\ot \rho(x\boo)-x\qi\ot\psi(x\qii)\\
=-\tau_{23}\big(\rho(x\boo)\ot x\bi+\psi(x\qi)\ot x\qii-x\boi\ot\gamma(x\boo)-x\qi\ot\phi(x\qii)\big)$,

\item[(CCP16)] $\rho(x\li)\ot x\lii+\psi(x\boi)\ot x\boo-x\boi\ot\Delta_{H}(x\boo)-x\qi\ot P(x\qii)\\
=-\tau_{23}\big(\rho(x\li)\ot x\lii+\psi(x\boi)\ot x\boo-x\boi\ot\Delta_{H}(x\boo)-x\qi\ot P(x\qii)\big)$.
\end{enumerate}
\end{definition}

\begin{lemma}\label{lem2} Let $(A, H)$ be  a  cycle cross coproduct system. If we define $E=A^{P}\# {}^{Q} H$ as the vector
space $A\oplus H$ with the   comultiplication
$$\Delta_{E}(a)=(\Delta_{A}+\phi+\psi+P)(a), \quad \Delta_{E}(x)=(\Delta_{H}+\rho+\gamma+Q)(x), $$
that is
$$\Delta_{E}(a)= a\li \ot a\lii+ a\moi \ot a\mo+a\mo\ot a\mi+a\ppi\ot a\pii,$$
$$\Delta_{E}(x)= x\li \ot x\lii+ x\boi \ot x\boo+x\boo \ot x\bi+x\qi\ot x\qii,$$
then  $A^{P}\# {}^{Q} H$ forms an alternative  coalgebra which we will call it the cycle cross coproduct  alternative coalgebra.
\end{lemma}

\begin{proof}
First, we have to check
$$
\begin{aligned}
 &(\Delta_E\ot\id)\Delta_E(a+x)-(\id\ot\Delta_E)\Delta_E(a+x)\\
 =&-\tau_{12}\big((\Delta_E\ot\id)\Delta_E(a+x)-(\id\ot\Delta_E)\Delta_E(a+x)\big).
 \end{aligned}
 $$
By direct computations, the left hand side is equal to
\begin{eqnarray*}
&&(\Delta_E\ot\id)\Delta_E(a+x)-(\id\ot\Delta_E)\Delta_E(a+x)\\
&=&(\Delta_E\ot\id)(a_{1} \otimes a_{2}+a_{(-1)} \otimes a_{(0)}+a_{(0)} \otimes a_{(1)}+a\ppi\ot a\pii+x_{1} \otimes x_{2}\\
&&+x_{[-1]} \otimes x_{[0]}+x_{[0]} \otimes x_{[1]}+x\qi\ot x\qii)-(\id\ot\Delta_E)(a_{1} \otimes a_{2}+a_{(-1)} \otimes a_{(0)}\\
&&+a_{(0)} \otimes a_{(1)}+a\ppi\ot a\pii+x_{1} \otimes x_{2}+x_{[-1]} \otimes x_{[0]}+x_{[0]} \otimes x_{[1]}+x\qi\ot x\qii)\\
&=&\Delta_{A}\left(a_{1}\right) \otimes a_{2}+\phi\left(a_{1}\right) \otimes a_{2}+\psi\left(a_{1}\right) \otimes a_{2}
+P(a\li)\ot a\lii\\
&&+\Delta_{H}\left(a_{(-1)}\right) \otimes a_{(0)}+\rho\left(a_{(-1)}\right) \otimes a_{(0)}+\gamma\left(a_{(-1)}\right) \otimes a_{(0)}
+Q(a\loi)\ot a\loo\\
&&+\Delta_{A}\left(a_{(0)}\right) \otimes a_{(1)}+\phi\left(a_{(0)}\right) \otimes a_{(1)}+\psi\left(a_{(0)}\right) \otimes a_{(1)}
+P(a\loo)\ot a\lmi\\
&&+\Delta_{H}(a\ppi)\ot a\pii+\rho(a\ppi)\ot a\pii+\gamma(a\ppi)\ot a\pii+Q(a\ppi)\ot a\pii\\
&&+\Delta_{H}\left(x_{1}\right) \otimes x_{2}+\rho\left(x_{1}\right) \otimes x_{2}+\gamma\left(x_{1}\right) \otimes x_{2}
+Q(x\li)\ot x\lii\\
&&+\Delta_{A}\left(x_{[-1])}\right) \otimes x_{[0]}+\phi\left(x_{[-1]}\right) \otimes x_{[0]}+\psi\left(x_{[-1]}\right) \otimes x_{[0]}
+P(x\boi)\ot x\boo\\
&&+\Delta_{H}\left(x_{[0]}\right) \otimes x_{[1]}+\rho\left(x_{[0]}\right) \otimes x_{[1]}+\gamma\left(x_{[0]}\right) \otimes x_{[1]}
+Q(x\boo)\ot x\bi\\
&&+\Delta_{A}(x\qi)\ot x\qii+\phi(x\qi)\ot x\qii+\psi(x\qi)\ot x\qii+P(x\qi)\ot x\qii\\
&&-a_{1} \otimes \Delta_{A}\left(a_{2}\right)-a_{1} \otimes \phi\left(a_{2}\right)-a_{1} \otimes \psi\left(a_{2}\right)
-a\li\ot P(a\lii)\\
&&-a_{(-1)} \otimes \Delta_{A}\left(a_{(0)}\right)-a_{(-1)} \otimes \phi\left(a_{(0)}\right)-a_{(-1)} \otimes \psi\left(a_{(0)}\right)
-a\loi\ot P(a\loo)\\
&&-a_{(0)} \otimes \Delta_{H}\left(a_{(1)}\right)-a_{(0)} \otimes \rho\left(a_{(1)}\right)-a_{(0)} \otimes \gamma\left(a_{(1)}\right)
-a\loo\ot Q(a\lmi)\\
&&-a\ppi\ot \Delta_{H}(a\pii)-a\ppi\ot \rho(a\pii)-a\ppi\ot\gamma(a\pii)-a\ppi\ot Q(a\pii)\\
&&-x_{1} \otimes \Delta_{H}\left(x_{2}\right)-x_{1} \otimes \rho\left(x_{2}\right)-x_{1} \otimes \gamma\left(x_{2}\right)
-x\li\ot Q(x\lii)\\
&&-x_{[-1]} \otimes \Delta_{H}\left(x_{[0]}\right)-x_{[-1]} \otimes \rho\left(x_{[0]}\right)-x_{[-1]} \otimes \gamma\left(x_{[0]}\right)
-x\boi\ot Q(x\boo)\\
&&-x_{[0]} \otimes \Delta_{A}\left(x_{[1]}\right)-x_{[0]} \otimes \phi\left(x_{[1]}\right)-x_{[0]} \otimes \psi\left(x_{[1]}\right)
-x\boo\ot P(x\bi)\\
&&-x\qi\ot\Delta_{A}(x\qii)-x\qi\ot\phi(x\qii)-x\qi\ot\psi(x\qii)-x\qi\ot P(x\qii),
\end{eqnarray*}
and the right hand side is equal to
\begin{eqnarray*}
&& -\tau_{12}\big((\Delta_E\ot\id)\Delta_E(a+x)-(\id\ot\Delta_E)\Delta_E(a+x)\big)\\
&=&-\tau_{12}(\Delta_E\ot\id)(a_{1} \otimes a_{2}+a_{(-1)} \otimes a_{(0)}+a_{(0)} \otimes a_{(1)}
+a\ppi\ot a\pii+x_{1} \otimes x_{2}\\
&&+x_{[-1]} \otimes x_{[0]}+x_{[0]} \otimes x_{[1]}+x\qi\ot x\qii)
+\tau_{12}(\id\ot\Delta_E)(a_{1} \otimes a_{2}+a_{(-1)} \otimes a_{(0)}\\
&&+a_{(0)} \otimes a_{(1)}+a\ppi\ot a\pii+x_{1} \otimes x_{2}+x_{[-1]} \otimes x_{[0]}+x_{[0]} \otimes x_{[1]}+x\qi\ot x\qii)\\
&=&-\tau_{12}\big(\Delta_{A}\left(a_{1}\right) \otimes a_{2}\big)-\tau_{12}\big(\phi\left(a_{1}\right) \otimes a_{2}\big)-\tau_{12}\big(\psi\left(a_{1}\right) \otimes a_{2}\big)-\tau_{12}\big(P(a\li)\ot a\lii\big)\\
&&-\tau_{12}\big(\Delta_{H}\left(a_{(-1)}\right) \otimes a_{(0)}\big)-\tau_{12}\big(\rho\left(a_{(-1)}\right) \otimes a_{(0)}\big)
-\tau_{12}\big(\gamma\left(a_{(-1)}\right) \otimes a_{(0)}\big)\\
&&-\tau_{12}\big(Q(a\loi)\ot a\loo\big)-\tau_{12}\big(\Delta_{A}\left(a_{(0)}\right) \otimes a_{(1)}\big)-\tau_{12}\big(\phi\left(a_{(0)}\right) \otimes a_{(1)}\big)\\
&&-\tau_{12}\big(\psi\left(a_{(0)}\right) \otimes a_{(1)}\big)-\tau_{12}\big(P(a\loo)\ot a\lmi\big)-\tau_{12}\big(\Delta_{H}(a\ppi)\ot a\pii\big)\\
&&-\tau_{12}\big(\rho(a\ppi)\ot a\pii\big)-\tau_{12}\big(\gamma(a\ppi)\ot a\pii\big)-\tau_{12}\big(Q(a\ppi)\ot a\pii\big)\\
&&-\tau_{12}\big(\Delta_{H}\left(x_{1}\right) \otimes x_{2}\big)-\tau_{12}\big(\rho\left(x_{1}\right) \otimes x_{2}\big)
-\tau_{12}\big(\gamma\left(x_{1}\right) \otimes x_{2}\big)-\tau_{12}\big(Q(x\li)\ot x\lii\big)\\
&&-\tau_{12}\big(\Delta_{A}\left(x_{[-1])}\right) \otimes x_{[0]}\big)-\tau_{12}\big(\phi\left(x_{[-1]}\right) \otimes x_{[0]}\big)
-\tau_{12}\big(\psi\left(x_{[-1]}\right) \otimes x_{[0]}\big)\\
&&-\tau_{12}\big(P(x\boi)\ot x\boo\big)-\tau_{12}\big(\Delta_{H}\left(x_{[0]}\right) \otimes x_{[1]}\big)-\tau_{12}\big(\rho\left(x_{[0]}\right) \otimes x_{[1]}\big)\\
&&-\tau_{12}\big(\gamma\left(x_{[0]}\right) \otimes x_{[1]}\big)-\tau_{12}\big(Q(x\boo)\ot x\bi\big)-\tau_{12}\big(\Delta_{A}(x\qi)\ot x\qii\big)\\
&&-\tau_{12}\big(\phi(x\qi)\ot x\qii\big)-\tau_{12}\big(\psi(x\qi)\ot x\qii\big)-\tau_{12}\big(P(x\qi)\ot x\qii\big)\\
&&+\tau_{12}\big(a_{1} \otimes \Delta_{A}\left(a_{2}\right)\big)+\tau_{12}\big(a_{1} \otimes \phi\left(a_{2}\right)\big)
+\tau_{12}\big(a_{1} \otimes \psi\left(a_{2}\right)\big)+\tau_{12}\big(a\li\ot P(a\lii)\big)\\
&&+\tau_{12}\big(a_{(-1)} \otimes \Delta_{A}\left(a_{(0)}\right)\big)+\tau_{12}\big(a_{(-1)} \otimes \phi\left(a_{(0)}\right)\big)
+\tau_{12}\big(a_{(-1)} \otimes \psi\left(a_{(0)}\right)\big)\\
&&+\tau_{12}\big(a\loi\ot P(a\loo)\big)+\tau_{12}\big(a_{(0)} \otimes \Delta_{H}\left(a_{(1)}\right)\big)+\tau_{12}\big(a_{(0)} \otimes \rho\left(a_{(1)}\right)\big)\\
&&+\tau_{12}\big(a_{(0)} \otimes \gamma\left(a_{(1)}\right)\big)+\tau_{12}\big(a\loo\ot Q(a\lmi)\big)+\tau_{12}\big(a\ppi\ot \Delta_{H}(a\pii)\big)\\
&&+\tau_{12}\big(a\ppi\ot \rho(a\pii)\big)+\tau_{12}\big(a\ppi\ot\gamma(a\pii)\big)+\tau_{12}\big(a\ppi\ot Q(a\pii)\big)\\
&&+\tau_{12}\big(x_{1} \otimes \Delta_{H}\left(x_{2}\right)\big)+\tau_{12}\big(x_{1} \otimes \rho\left(x_{2}\right)\big)
+\tau_{12}\big(x_{1} \otimes \gamma\left(x_{2}\right)\big)+\tau_{12}\big(x\li\ot Q(x\lii)\big)\\
&&+\tau_{12}\big(x_{[-1]} \otimes \Delta_{H}\left(x_{[0]}\right)\big)+\tau_{12}\big(x_{[-1]} \otimes \rho\left(x_{[0]}\right)\big)
+\tau_{12}\big(x_{[-1]} \otimes \gamma\left(x_{[0]}\right)\big)\\
&&+\tau_{12}\big(x\boi\ot Q(x\boo)\big)+\tau_{12}\big(x_{[0]} \otimes \Delta_{A}\left(x_{[1]}\right)\big)+\tau_{12}\big(x_{[0]} \otimes \phi\left(x_{[1]}\right)\big)\\
&&+\tau_{12}\big(x_{[0]} \otimes \psi\left(x_{[1]}\right)\big)+\tau_{12}\big(x\boo\ot P(x\bi)\big)+\tau_{12}\big(x\qi\ot\Delta_{A}(x\qii)\big)\\
&&+\tau_{12}\big(x\qi\ot\phi(x\qii)\big)
+\tau_{12}\big(x\qi\ot\psi(x\qii)\big)+\tau_{12}\big(x\qi\ot P(x\qii)\big).
\end{eqnarray*}
Thus the two sides are equal to each other if and only if (CCP1)--(CCP8) hold.

Next, we need to check
  $$
 \begin{aligned}
 &(\Delta_E\ot\id)\Delta_E(a+x)-(\id\ot\Delta_E)\Delta_E(a+x)\\
 =&-\tau_{23}\big((\Delta_E\ot\id)\Delta_E(a+x)-(\id\ot\Delta_E)\Delta_E(a+x)\big).
 \end{aligned}
 $$
By direct computations, the left hand side is equal to
\begin{eqnarray*}
&&(\Delta_E\ot\id)\Delta_E(a+x)-(\id\ot\Delta_E)\Delta_E(a+x)\\
&=&(\Delta_E\ot\id)(a_{1} \otimes a_{2}+a_{(-1)} \otimes a_{(0)}+a_{(0)} \otimes a_{(1)}+a\ppi\ot a\pii+x_{1} \otimes x_{2}\\
&&+x_{[-1]} \otimes x_{[0]}+x_{[0]} \otimes x_{[1]}+x\qi\ot x\qii)-(\id\ot\Delta_E)(a_{1} \otimes a_{2}+a_{(-1)} \otimes a_{(0)}\\
&&+a_{(0)} \otimes a_{(1)}+a\ppi\ot a\pii+x_{1} \otimes x_{2}+x_{[-1]} \otimes x_{[0]}+x_{[0]} \otimes x_{[1]}+x\qi\ot x\qii)\\
&=&\Delta_{A}\left(a_{1}\right) \otimes a_{2}+\phi\left(a_{1}\right) \otimes a_{2}+\psi\left(a_{1}\right) \otimes a_{2}
+P(a\li)\ot a\lii\\
&&+\Delta_{H}\left(a_{(-1)}\right) \otimes a_{(0)}+\rho\left(a_{(-1)}\right) \otimes a_{(0)}+\gamma\left(a_{(-1)}\right) \otimes a_{(0)}
+Q(a\loi)\ot a\loo\\
&&+\Delta_{A}\left(a_{(0)}\right) \otimes a_{(1)}+\phi\left(a_{(0)}\right) \otimes a_{(1)}+\psi\left(a_{(0)}\right) \otimes a_{(1)}
+P(a\loo)\ot a\lmi\\
&&+\Delta_{H}(a\ppi)\ot a\pii+\rho(a\ppi)\ot a\pii+\gamma(a\ppi)\ot a\pii+Q(a\ppi)\ot a\pii\\
&&+\Delta_{H}\left(x_{1}\right) \otimes x_{2}+\rho\left(x_{1}\right) \otimes x_{2}+\gamma\left(x_{1}\right) \otimes x_{2}
+Q(x\li)\ot x\lii\\
&&+\Delta_{A}\left(x_{[-1])}\right) \otimes x_{[0]}+\phi\left(x_{[-1]}\right) \otimes x_{[0]}+\psi\left(x_{[-1]}\right) \otimes x_{[0]}
+P(x\boi)\ot x\boo\\
&&+\Delta_{H}\left(x_{[0]}\right) \otimes x_{[1]}+\rho\left(x_{[0]}\right) \otimes x_{[1]}+\gamma\left(x_{[0]}\right) \otimes x_{[1]}
+Q(x\boo)\ot x\bi\\
&&+\Delta_{A}(x\qi)\ot x\qii+\phi(x\qi)\ot x\qii+\psi(x\qi)\ot x\qii+P(x\qi)\ot x\qii\\
&&-a_{1} \otimes \Delta_{A}\left(a_{2}\right)-a_{1} \otimes \phi\left(a_{2}\right)-a_{1} \otimes \psi\left(a_{2}\right)
-a\li\ot P(a\lii)\\
&&-a_{(-1)} \otimes \Delta_{A}\left(a_{(0)}\right)-a_{(-1)} \otimes \phi\left(a_{(0)}\right)-a_{(-1)} \otimes \psi\left(a_{(0)}\right)
-a\loi\ot P(a\loo)\\
&&-a_{(0)} \otimes \Delta_{H}\left(a_{(1)}\right)-a_{(0)} \otimes \rho\left(a_{(1)}\right)-a_{(0)} \otimes \gamma\left(a_{(1)}\right)
-a\loo\ot Q(a\lmi)\\
&&-a\ppi\ot \Delta_{H}(a\pii)-a\ppi\ot \rho(a\pii)-a\ppi\ot\gamma(a\pii)-a\ppi\ot Q(a\pii)\\
&&-x_{1} \otimes \Delta_{H}\left(x_{2}\right)-x_{1} \otimes \rho\left(x_{2}\right)-x_{1} \otimes \gamma\left(x_{2}\right)
-x\li\ot Q(x\lii)\\
&&-x_{[-1]} \otimes \Delta_{H}\left(x_{[0]}\right)-x_{[-1]} \otimes \rho\left(x_{[0]}\right)-x_{[-1]} \otimes \gamma\left(x_{[0]}\right)
-x\boi\ot Q(x\boo)\\
&&-x_{[0]} \otimes \Delta_{A}\left(x_{[1]}\right)-x_{[0]} \otimes \phi\left(x_{[1]}\right)-x_{[0]} \otimes \psi\left(x_{[1]}\right)
-x\boo\ot P(x\bi)\\
&&-x\qi\ot\Delta_{A}(x\qii)-x\qi\ot\phi(x\qii)-x\qi\ot\psi(x\qii)-x\qi\ot P(x\qii),
\end{eqnarray*}
and the right hand side is equal to
\begin{eqnarray*}
&& -\tau_{23}\big((\Delta_E\ot\id)\Delta_E(a+x)-(\id\ot\Delta_E)\Delta_E(a+x)\big)\\
&=&-\tau_{23}(\Delta_E\ot\id)(a_{1} \otimes a_{2}+a_{(-1)} \otimes a_{(0)}+a_{(0)} \otimes a_{(1)}
+a\ppi\ot a\pii+x_{1} \otimes x_{2}\\
&&+x_{[-1]} \otimes x_{[0]}+x_{[0]} \otimes x_{[1]}+x\qi\ot x\qii)
+\tau_{23}(\id\ot\Delta_E)(a_{1} \otimes a_{2}+a_{(-1)} \otimes a_{(0)}\\
&&+a_{(0)} \otimes a_{(1)}+a\ppi\ot a\pii+x_{1} \otimes x_{2}+x_{[-1]} \otimes x_{[0]}+x_{[0]} \otimes x_{[1]}+x\qi\ot x\qii)\\
&=&-\tau_{23}\big(\Delta_{A}\left(a_{1}\right) \otimes a_{2}\big)-\tau_{23}\big(\phi\left(a_{1}\right) \otimes a_{2}\big)
-\tau_{23}\big(\psi\left(a_{1}\right) \otimes a_{2}\big)-\tau_{23}\big(P(a\li)\ot a\lii\big)\\
&&-\tau_{23}\big(\Delta_{H}\left(a_{(-1)}\right) \otimes a_{(0)}\big)-\tau_{23}\big(\rho\left(a_{(-1)}\right) \otimes a_{(0)}\big)
-\tau_{23}\big(\gamma\left(a_{(-1)}\right) \otimes a_{(0)}\big)\\
&&-\tau_{23}\big(Q(a\loi)\ot a\loo\big)-\tau_{23}\big(\Delta_{A}\left(a_{(0)}\right) \otimes a_{(1)}\big)-\tau_{23}\big(\phi\left(a_{(0)}\right) \otimes a_{(1)}\big)\\
&&-\tau_{23}\big(\psi\left(a_{(0)}\right) \otimes a_{(1)}\big)-\tau_{23}\big(P(a\loo)\ot a\lmi\big)-\tau_{23}\big(\Delta_{H}(a\ppi)\ot a\pii\big)\\
&&-\tau_{23}\big(\rho(a\ppi)\ot a\pii\big)-\tau_{23}\big(\gamma(a\ppi)\ot a\pii\big)-\tau_{23}\big(Q(a\ppi)\ot a\pii\big)\\
&&-\tau_{23}\big(\Delta_{H}\left(x_{1}\right) \otimes x_{2}\big)-\tau_{23}\big(\rho\left(x_{1}\right) \otimes x_{2}\big)
-\tau_{23}\big(\gamma\left(x_{1}\right) \otimes x_{2}\big)-\tau_{23}\big(Q(x\li)\ot x\lii\big)\\
&&-\tau_{23}\big(\Delta_{A}\left(x_{[-1])}\right) \otimes x_{[0]}\big)-\tau_{23}\big(\phi\left(x_{[-1]}\right) \otimes x_{[0]}\big)
-\tau_{23}\big(\psi\left(x_{[-1]}\right) \otimes x_{[0]}\big)\\
&&-\tau_{23}\big(P(x\boi)\ot x\boo\big)-\tau_{23}\big(\Delta_{H}\left(x_{[0]}\right) \otimes x_{[1]}\big)-\tau_{23}\big(\rho\left(x_{[0]}\right) \otimes x_{[1]}\big)\\
&&-\tau_{23}\big(\gamma\left(x_{[0]}\right) \otimes x_{[1]}\big)-\tau_{23}\big(Q(x\boo)\ot x\bi\big)-\tau_{23}\big(\Delta_{A}(x\qi)\ot x\qii\big)\\
&&-\tau_{23}\big(\phi(x\qi)\ot x\qii\big)
-\tau_{23}\big(\psi(x\qi)\ot x\qii\big)-\tau_{23}\big(P(x\qi)\ot x\qii\big)\\
&&+\tau_{23}\big(a_{1} \otimes \Delta_{A}\left(a_{2}\right)\big)+\tau_{23}\big(a_{1} \otimes \phi\left(a_{2}\right)\big)
+\tau_{23}\big(a_{1} \otimes \psi\left(a_{2}\right)\big)+\tau_{23}\big(a\li\ot P(a\lii)\big)\\
&&+\tau_{23}\big(a_{(-1)} \otimes \Delta_{A}\left(a_{(0)}\right)\big)+\tau_{23}\big(a_{(-1)} \otimes \phi\left(a_{(0)}\right)\big)
+\tau_{23}\big(a_{(-1)} \otimes \psi\left(a_{(0)}\right)\big)\\
&&+\tau_{23}\big(a\loi\ot P(a\loo)\big)+\tau_{23}\big(a_{(0)} \otimes \Delta_{H}\left(a_{(1)}\right)\big)+\tau_{23}\big(a_{(0)} \otimes \rho\left(a_{(1)}\right)\big)\\
&&+\tau_{23}\big(a_{(0)} \otimes \gamma\left(a_{(1)}\right)\big)+\tau_{23}\big(a\loo\ot Q(a\lmi)\big)+\tau_{23}\big(a\ppi\ot \Delta_{H}(a\pii)\big)\\
&&+\tau_{23}\big(a\ppi\ot \rho(a\pii)\big)
+\tau_{23}\big(a\ppi\ot\gamma(a\pii)\big)+\tau_{23}\big(a\ppi\ot Q(a\pii)\big)\\
&&+\tau_{23}\big(x_{1} \otimes \Delta_{H}\left(x_{2}\right)\big)+\tau_{23}\big(x_{1} \otimes \rho\left(x_{2}\right)\big)
+\tau_{23}\big(x_{1} \otimes \gamma\left(x_{2}\right)\big)+\tau_{23}\big(x\li\ot Q(x\lii)\big)\\
&&+\tau_{23}\big(x_{[-1]} \otimes \Delta_{H}\left(x_{[0]}\right)\big)+\tau_{23}\big(x_{[-1]} \otimes \rho\left(x_{[0]}\right)\big)
+\tau_{23}\big(x_{[-1]} \otimes \gamma\left(x_{[0]}\right)\big)\\
&&+\tau_{23}\big(x\boi\ot Q(x\boo)\big)+\tau_{23}\big(x_{[0]} \otimes \Delta_{A}\left(x_{[1]}\right)\big)+\tau_{23}\big(x_{[0]} \otimes \phi\left(x_{[1]}\right)\big)\\
&&+\tau_{23}\big(x_{[0]} \otimes \psi\left(x_{[1]}\right)\big)+\tau_{23}\big(x\boo\ot P(x\bi)\big)+\tau_{23}\big(x\qi\ot\Delta_{A}(x\qii)\big)\\
&&+\tau_{23}\big(x\qi\ot\phi(x\qii)\big)
+\tau_{23}\big(x\qi\ot\psi(x\qii)\big)+\tau_{23}\big(x\qi\ot P(x\qii)\big).
\end{eqnarray*}
Thus the two sides are equal to each other if and only if (CCP9)--(CCP16) hold.
\end{proof}

\begin{definition}\label{cocycledmp}
Let $A$ and $H$ be both  alternative algebras and  alternative  coalgebras. If  the following conditions hold:
\begin{enumerate}
\item[(CDM1)]  $\phi(ab)+\gamma(\theta(a, b))\\
 =-a\lmi\ot b a\loo+b\loi\ot a b\loo+b\loi\ot b\loo a+(a\loi\trl b)\ot a\loo\\
+(a\lmi\trl b)\ot a\loo-(a\trr b\loi)\ot b\loo+b\ppi\ot(a\ppl b\pii)+b\ppi\ot(b\pii\ppr a)\\
-a\pii\ot(b\ppl a\ppi)+\theta(a\li, b)\ot a\lii+\theta(a\lii, b)\ot a\li-\theta(a, b\li)\ot b\lii$,
\item[(CDM2)] $\psi(a b)+\rho(\theta(a, b))\\
=a\loo b\ot a\lmi+a\loo b\ot a\loi-a b\loo\ot b\lmi-a\loo\ot (b\trr a\loi)+b\loo\ot(a\trr b\lmi)\\
+b\loo\ot(b\lmi\trl a)+(a\ppi\ppr b)\ot a\pii+(a\pii\ppr b)\ot a\ppi\\
-(a\ppl b\ppi)\ot b\pii+b\li\ot\theta(a, b\lii)+b\li\ot\theta(b\lii, a)-a\lii\ot\theta(b, a\li)$,
\item[(CDM3)] $\rho(x y)+\psi(\sigma(x, y))\\
=(x\bi\ppl y)\ot x\boo+(x\boi\ppl y)\ot x\boo-x_{[1]} \otimes y x_{[0]}-\left(x \ppr y_{[-1]}\right) \otimes y_{[0]}  \\
+y\boi\ot x y\boo+y\boi\ot y\boo x+\sigma(x\li, y)\ot x\lii+\sigma(x\lii, y)\ot x\li-\sigma(x, y\li)\ot y\lii\\
+y\qi\ot (y\qii\trr x)+y\qi\ot (x\trl y\qii)-x\qii\ot(y\trl x\qi)$,
\item[(CDM4)] $\gamma(x y)+\phi(\sigma(x, y))\\
=x\boo y\ot x\bi+x\boo y\ot x\boi-x_{[0]}\otimes (y\ppr x\boi)-xy_{[0]}\otimes y_{[1]}+y\boo\ot (x\ppr y\bi)\\
+y\boo\ot(y\bi\ppl x)+(x\qi\trr y)\ot x\qii+(x\qii\trr y)\ot x\qi-x\lii\ot\sigma(y, x\li)\\
+y\li\ot\sigma(x, y\lii)+y\li\ot\sigma(y\lii, x)-(x\trl y\qi)\ot y\qii$,
\item[(CDM5)] $\Delta_{A}(x \ppr b)+Q(x\trl  b)$ \\
$=(x\boo\ppr b)\ot x\bi+(x\boo\ppr b)\ot x\boi-x\bi\ot(b\ppl x\boo)-\left(x \ppr b_{1}\right) \otimes b_{2}\\
+b\li\ot(x\ppr b\lii)+b\li\ot(b\lii\ppl x)+b\loo\ot\sigma(x, b\lmi)+b\loo\ot\sigma(b\lmi, x)\\
-\sigma(x, b\loi)\ot b\loo+x\qi b\ot x\qii+x\qii b\ot x\qi-x\qii\ot b x\qi$,

\item[(CDM6)] $\Delta_{A}(a\ppl y)+Q(a\trr y)$\\
$=(a\li\ppl y)\ot a\lii+\left(a_{2} \ppl y\right)\ot a\li-a\lii\ot(y\ppr a\li)-\left(a\ppl y_{[0]}\right) \otimes y_{[1]}\\
+y\boi\ot(a\ppl y\boo)+y\boi\ot(y\boo\ppr a)+\sigma(a\loi, y)\ot a\loo+\sigma(a\lmi, y)\ot a\loo\\
-a\loo\ot\sigma(y, a\loi)+y\qi\ot a y\qii+y\qi\ot y\qii a-a y\qi\ot y\qii$,

\item[(CDM7)] $\Delta_{H}(a \trr y)+P(a\ppl y)$\\
$=(a\loo\trr y)\ot a\lmi+(a\loo\trr y)\ot a\loi-a\lmi\ot(y\trl a\loo)+y\li\ot(a\trr y\lii)\\
+y\li\ot(y\lii\trl a)-\left(a \trr y_{1}\right) \otimes y_{2}+y\boo\ot\theta(a, y\bi)+y\boo\ot\theta(y\bi, a)\\
-\theta(a, y\boi)\ot y\boo+a\ppi y\ot a\pii+a\pii y\ot a\ppi-a\pii\ot y a\ppi$,

\item[(CDM8)] $\Delta_{H}(x \trl b)+P(x\ppr b)$\\
$=(x\li\trl b)\ot x\lii+(x\lii\trl b)\ot x\li-x\lii\ot(b\trr x\li)+b\loi\ot(x\trl b\loo)\\
+b\loi\ot(b\loo\trr x)-\left(x\trl b_{(0)}\right) \otimes b_{(1)}+\theta(x\boi, b)\ot x\boo+\theta(x\bi, b)\ot x\boo\\
-x\boo\ot\theta(b, x\boi)+b\ppi\ot x b\pii+b\ppi\ot b\pii x-x b\ppi\ot b\pii$,

\item[(CDM9)]$\Delta_{H}(\theta(a, b))+P(a, b)$\\
$=\theta(a\loo, b)\ot a\lmi+\theta(a\loo, b)\ot a\loi-\theta(a, b\loo)\ot b\lmi+b\loi\ot\theta(a, b\loo)\\
+b\loi\ot\theta(b\loo, a)-a\lmi\ot\theta(b, a\loo)+(a\ppi\trl b)\ot a\pii+(a\pii\trl b)\ot a\ppi\\
-(a\trr b\ppi)\ot b\pii+b\ppi\ot(a\trr b\pii)+b\ppi\ot(b\pii\trl a)-a\pii\ot(b\trr a\ppi)$,

\item[(CDM10)]$\Delta_{A}(\sigma(x, y))+Q(x, y)$\\
$=\sigma(x\boo, y)\ot x\bi+\sigma(x\boo, y)\ot x\boi-\sigma(x, y\boo)\ot y\bi+y\boi\ot\sigma(x, y\boo)\\
+y\boi\ot\sigma(y\boo, x)-x\bi\ot\sigma(y, x\boo)+(x\qi\ppl y)\ot x\qii+(x\qii\ppl y)\ot x\qi\\
-x\qii\ot(y\ppr x\qi)+y\qi\ot(x\ppr y\qii)+y\qi\ot(y\qii\ppl x)-(x\ppr y\qi)\ot y\qii$,

\item[(CDM11)]
 $\phi(x \ppr b)+\gamma(x\trl b)\\
 =(x\boo\trl b)\ot x\bi+(x\boo\trl b)\ot x\boi-x\lii\ot(b\ppl x\li)-x\boo\ot b x\boi$\\
$+b\loi\ot(x\ppr b\loo)-\left(x\trl b_{1}\right) \otimes b_{2}+b\loi\ot(b\loo\ppl x)-x b_{(-1)} \otimes b_{(0)}$\\
$+\theta(x\qi, b)\ot x\qii+\theta(x\qii, b)\ot x\qi+b\ppi\ot\sigma(x, b\pii)+b\ppi\ot\sigma(b\pii, x)$,

\item[(CDM12)]
$\psi(a\ppl y)+\rho(a \trr y)\\
=(a\loo\ppl y)\ot a\lmi+(a\loo\ppl y)\ot a\loi-a\lii\ot(y\trl a\li)-a\loo\ot y a\loi$\\
$+y\boi\ot(a\trr y\boo)-\left(a\ppl y_{1}\right) \otimes y_{2}+y\boi\ot(y\boo\trl a)-a y_{[-1]} \otimes y_{[0]}$\\
$+\sigma(a\ppi, y)\ot a\pii+\sigma(a\pii, y)\ot a\ppi+y\qi\ot\theta(a, y\qii)+y\qi\ot\theta(y\qii, a)$,

\item[(CDM13)]
$\psi(x \ppr b)+\rho(x\trl b)\\
=(x\li\ppr b)\ot x\lii+x\boi b\ot x\boo+(x\lii\ppr b)\ot x\li+x\bi b\ot x\boo\\
-x\bi\ot(b\trr x\boo)+b\li\ot(x\trl b\lii)+b\loo\ot x b\lmi+b\li\ot(b\lii\trr x)\\
+b\loo\ot b\lmi x-(x \ppr b_{(0)}) \otimes b_{(1)}-x\qii\ot\theta(b, x\qi)-\sigma(x, b\ppi)\ot b\pii$,

\item[(CDM14)]
$\phi(a\ppl y)+\gamma(a \trr y)\\
=(a\li\trr y)\ot a\lii+a\loi y\ot a\loo+(a\lii\trr y)\ot a\li+a\lmi y\ot a\loo\\
-a\lmi\ot (y\ppr a\loo)+y\li\ot(a\ppl y\lii)+y\boo\ot a y\bi+y\li\ot(y\lii\ppr a)\\
+y\boo\ot y\bi a-\left(a\trr y_{[0]}\right)\otimes y_{[1]}-\theta(a, y\qi)\ot y\qii-a\pii\ot\sigma(y, a\ppi)$,

\item[(CDM15)]  $\phi(b a)+\gamma(\theta(b, a))+\tau\psi(b a)+\tau\rho(\theta(b, a))\\
=a\loi\ot b a\loo+(b\trr a\lmi)\ot a\loo+(b\loi\trl a)\ot b\loo+b\lmi\ot b\loo a\\
+a\ppi\ot(b\ppl a\pii)+b\pii\ot(b\ppi\ppr a)+\theta(b, a\lii)\ot a\li+\theta(b\li, a)\ot b\lii$,

\item[(CDM16)] $\psi(b a)+\rho(\theta(b,a))+\tau\phi(b a)+\tau\gamma(\theta(b, a))\\
=a\loo\ot(b\trr a\lmi)+b a\loo \ot a\loi+ b\loo a\ot b\lmi+b\loo\ot(b\loi\trl a)\\
+a\li\ot\theta(b, a\lii)+b\lii\ot\theta(b\li, a)+(b\ppl a\pii)\ot a\ppi+(b\ppi\ppr a)\ot b\pii$,

\item[(CDM17)] $\rho(y x)+\psi(\sigma(y, x))+\tau\gamma(y x)+\tau\phi(\sigma(y, x))\\
=x\boi\ot y x\boo+(y\ppr x\bi)\ot x\boo+(y\boi\ppl x)\ot y\boo+y\bi\ot y\boo x\\
+x\qi\ot(y\trl x\qii)+y\qii\ot(y\qi\trr x)+\sigma(y, x\lii)\ot x\li+\sigma(y\li, x)\ot y\lii$,

\item[(CDM18)] $\gamma(y x)+\phi(\sigma(y, x))+\tau\rho(y x)+\tau\psi(\sigma(y, x))\\
=y x\boo\ot x\boi+y\boo x\ot y\bi+y\boo\ot(y\boi\ppl x)+x\boo\ot(y\ppr x\bi)\\
+x\li\ot\sigma(y, x\lii)+y\lii\ot\sigma(y\li, x)+(y\trl x\qii)\ot x\qi+(y\qi\trr x)\ot y\qii$,

\item[(CDM19)] $\Delta_{A}(b \ppl x)+Q(b\trr x)+\tau\Delta_{A}(b \ppl x)+\tau Q(b\trr x)\\
=x\boi\ot(b\ppl x\boo)+(b\ppl x\boo)\ot x\boi+(b\li\ppl x)\ot b\lii+b\lii\ot(b\li\ppl x)\\
+x\qi\ot b x\qii+b x\qii\ot x\qi+\sigma(b\loi, x)\ot b\loo+b\loo\ot\sigma(b\loi, x)$,

\item[(CDM20)] $\Delta_{A}(y\ppr a)+Q(y\trl a)+\tau\Delta_{A}(y\ppr a)+\tau Q(y\trl a)\\
=a\li\ot(y\ppr a\lii)+(y\ppr a\lii)\ot a\li+(y\boo\ppr a)\ot y\bi+y\bi\ot(y\boo\ppr a)\\
+a\loo\ot\sigma(y, a\lmi)+\sigma(y, a\lmi)\ot a\loo+y\qi a\ot y\qii+y\qii\ot y\qi a$,

\item[(CDM21)] $\Delta_{H}(y\trl a)+P(y\ppr a)+\tau\Delta_{H}(y\trl a)+\tau P(y\ppr a)\\
=a\loi\ot(y\trl a\loo)+(y\trl a\loo)\ot a\loi+(y\li\trl a)\ot y\lii+y\lii\ot(y\li\trl a)\\
+a\ppi\ot y a\pii+y a\pii\ot a\ppi+\theta(y\boi, a)\ot y\boo+y\boo\ot\theta(y\boi, a)$,

\item[(CDM22)] $\Delta_{H}(b \trr x)+P(b\ppl x)+\tau\Delta_{H}(b \trr x)+\tau P(b\ppl x)\\
=x\li\ot(b\trr x\lii)+(b\trr x\lii)\ot x\li+(b\loo\trr x)\ot b\lmi+b\lmi\ot (b\loo\trr x)\\
+x\boo\ot\theta(b, x\bi)+\theta(b, x\bi)\ot x\boo+b\ppi x\ot b\pii+b\pii\ot b\ppi x$,

\item[(CDM23)]
$\phi(y \ppr a)+\tau\psi(y \ppr a)+\gamma(y\trl a)+\tau\rho(y\trl a)\\
=a\loi\ot(y\ppr a\loo)+(y\trl a\lii)\ot a\li+y a\lmi\ot a\loo+(y\boo\trl a)\ot y\bi\\
+y\lii\ot(y\li\ppr a)+y\boo\ot y\boi a+a\ppi\ot\sigma(y, a\pii)+\theta(y\qi, a)\ot y\qii$,

\item[(CDM24)]
$\phi(b \ppl x)+\tau\psi(b \ppl x)+\gamma(b\trr x)+\tau\rho(b\trr x)\\
=x\li\ot(b\ppl x\lii)+x\boo\ot b x\bi+(b\trr x\boo)\ot x\boi+(b\li\trr x)\ot b\lii\\
+b\loi x\ot b\loo+b\lmi\ot(b\loo\ppl x)+\theta(b, x\qii)\ot x\qi+b\pii\ot\sigma(b\ppi, x)$,

\item[(CDM25)]
$\psi(y \ppr a)+\tau\phi(y \ppr a)+\rho(y\trl a)+\tau\gamma(y\trl a)\\
=a\li\ot(y\trl a\lii)+a\loo\ot y a\lmi+(y\ppr a\loo)\ot a\loi+(y\li\ppr a)\ot y\lii\\
+y\boi a\ot y\boo+y\bi\ot(y\boo\trl a)+\sigma(y, a\pii)\ot a\ppi+y\qii\ot\theta(y\qi, a)$,

\item[(CDM26)]
$\psi(b \ppl x)+\tau\phi(b \ppl x)+\rho(b\trr x)+\tau\gamma(b\trr x)\\
=x\boi\ot(b\trr x\boo)+(b\ppl x\lii)\ot x\li+b x\bi\ot x\boo+(b\loo\ppl x)\ot b\lmi\\
+b\lii\ot(b\li\trr x)+b\loo\ot b\loi x+x\qi\ot\theta(b, x\qii)+\sigma(b\ppi, x)\ot b\pii$,

\item[(CDM27)]
$\Delta_{A}(\sigma(y, x))+Q(y, x)+\tau\Delta_{A}(\sigma(y, x))+\tau Q(y,x)\\
=x\boi\ot\sigma(y, x\boo)+\sigma(y, x\boo)\ot x\boi+x\qi\ot(y\ppr x\qii)+(y\ppr x\qii)\ot x\qi\\
+\sigma(y\boo, x)\ot y\bi+y\bi\ot\sigma(y\boo, x)+(y\qi\ppl x)\ot y\qii+y\qii\ot(y\qi\ppl x)$,

\item[(CDM28)]
$P(b,a)+\Delta_{H}(\theta(b,a))+\tau P(b, a)+\tau\Delta_{H}(\theta(b, a))\\
=a\loi\ot\theta(b, a\loo)+\theta(b, a\loo)\ot a\loi+a\ppi\ot(b\trr a\pii)+(b\trr a\pii)\ot a\ppi\\
+\theta(b\loo, a)\ot b\lmi+b\lmi\ot\theta(b\loo, a)+(b\ppi\trl a)\ot b\pii+b\pii\ot(b\ppi\trl a)$,
\end{enumerate}
\noindent then $(A, H)$ is called a \emph{cocycle double matched pair}.
\end{definition}

\begin{definition}\label{cocycle-braided}
(i) A \emph{cocycle braided alternative bialgebra} $A$ is simultaneously a cocycle alternative algebra $(A, \theta)$ and  a cycle alternative coalgebra $(A, Q)$ satisfying the  conditions
\begin{enumerate}
\item[(CBB1)] $\Delta_{A}(ab)+Q(\theta(a,b))\\
=a_{1} b\otimes a_{2}+a\lii b\ot a\li-a\lii\ot b a\li+b\li\ot a b\lii\\
+b\li\ot b\lii a-a b\li\ot b\lii+(a\loi\ppr b)\ot a\loo+(a\lmi\ppr b)\ot a\loo\\
-a\loo\ot(b\ppl a\loi)+b\loo\ot(a\ppl b\lmi)+b\loo\ot(b\lmi\ppr a)-(a\ppl b\loi)\ot b\loo,$
    \end{enumerate}
\begin{enumerate}
\item[(CBB2)] $\Delta_{A}(ba)+Q(\theta(b,a))+\tau\Delta_{A}(ba)+\tau Q(\theta(b,a))\\
=a_{1} \otimes b a_{2}+b a\lii \ot a\li+b\li a\ot b\lii+b\lii\ot b\li a\\
+a\loo\ot(b\ppl a\lmi)+(b\ppl a\lmi)\ot a\loo+(b\loi\ppr a)\ot b\loo+b\loo\ot(b\loi\ppr a).$
    \end{enumerate}
(ii) A \emph{cocycle braided alternative bialgebra} $H$ is simultaneously a cocycle alternative algebra $(H, \sigma)$ and a cycle alternative coalgebra $(H, P)$ satisfying the conditions
\begin{enumerate}
\item[(CBB3)] $\Delta_{H}(xy)+P(\sigma(x, y))\\
=x\li y\ot x\lii+x\lii y\ot x\li-x\lii\ot y x\li+y\li\ot x y\lii\\
+y\li\ot y\lii x-x y\li\ot y\lii+(x\boi\trr y)\ot x\boo+(x\bi\trr y)\ot x\boo\\
-x\boo\ot(y\trl x\boi)+y\boo\ot(x\trl y\bi)+y\boo\ot(y\bi\trr x)-(x\trl y\boi)\ot y\boo,$
    \end{enumerate}
\begin{enumerate}
\item[(CBB4)] $\Delta_{H}(yx)+P(\sigma(y, x))+\tau\Delta_{H}(yx)+\tau P(\sigma(y, x))\\
=x\li \ot y x\lii+y x\lii \ot x\li+y\li x\ot y\lii+y\lii\ot y\li x\\
+x\boo\ot(y\trl x\bi)+(y\trl x\bi)\ot x\boo+(y\boi\trr x)\ot y\boo+y\boo\ot(y\boi\trr x).$
    \end{enumerate}
\end{definition}

The next theorem says that we can obtain an ordinary alternative bialgebra from two cocycle braided alternative bialgebras.

\begin{theorem}\label{main2}
Let $A$, $H$ be  cocycle braided alternative bialgebras, $(A, H)$ be a cocycle cross product system and a cycle cross coproduct system.
Then the cocycle cross product alternative  algebra and cycle cross coproduct alternative  coalgebra fit together to become an ordinary
alternative bialgebra if and only if $(A, H)$ forms  a cocycle double matched pair. We will call it the cocycle bicrossproduct alternative bialgebra and denote it by $A^{P}_{\sigma}\# {}^{Q}_{\theta}H$.
\end{theorem}

\begin{proof}  First, we need to check the first compatibility condition
$$
\begin{aligned}
&\Delta_{E}((a+x)(b+y))\\
=&\Delta_{E} (a+x)\bullet (b+y)+\tau\Delta_{E} (a+x)\bullet (b+y)\\
&- (b+y)\bullet\tau\Delta_{E}(a+x)+(a+x)\bullet \Delta_{E}(b+y)\\
&+[\Delta_{E}(b+y), (a+x)].
 \end{aligned}
 $$
The left hand side is equal to
\begin{eqnarray*}
&&\Delta_{E}((a+x)(b+y))\\
&=&\Delta_{E}(a b+x \ppr b+a\ppl y+\sigma(x, y)+ x y+x \trl b+a\trr y+\theta(a, b))\\
&=&\Delta_A(a b)+\Delta_A(x \ppr b)+\Delta_{A}(a\ppl y)+\Delta_{A}(\sigma(x, y))\\
&&+\phi(a b)+\phi(x \ppr b)+\phi(a\ppl y)+\phi(\sigma(x, y))\\
&&+\psi(a b)+\psi(x \ppr b)+\psi(a\ppl y)+\psi(\sigma(x, y))\\
&&+P(ab)+P(x \ppr b)+P(a\ppl y)+P(\sigma(x, y))\\
&&+\Delta_{H}(x y)+\Delta_{H}(x \trl b)+\Delta_{H}(a \trr y)+\Delta_{H}(\theta(a, b))\\
&&+\rho(x y)+\rho(x \trl b)+\rho(a \trr y)+\rho(\theta(a, b))\\
&&+\gamma(x y)+\gamma(x \trl b)+\gamma(a \trr y)+\gamma(\theta(a, b))\\
&&+Q(xy)+Q(x \trl b)+Q(a\trr y)+Q(\theta(a, b)),
\end{eqnarray*}
and the right hand side is equal to
\begin{eqnarray*}
&&\Delta_{E} (a+x)\bullet (b+y)+\tau\Delta_{E} (a+x)\bullet (b+y)- (b+y)\bullet\tau\Delta_{E}(a+x)\\
&&+(a+x)\bullet \Delta_{E}(b+y)+[\Delta_{E}(b+y), (a+x)]\\
&=&(a_{1} \otimes a_{2}+a_{(-1)} \otimes a_{(0)}+a_{(0)} \otimes a_{(1)}+a\ppi\ot a\pii+x_{1} \otimes x_{2}+x_{[-1]} \otimes x_{[0]}\\
&&+x_{[0]}\ot x_{[1]}+x\qi\ot x\qii) \bullet(b+ y)+(a_{2} \otimes a_{1}+a_{(0)} \otimes a_{(-1)}+a_{(1)} \otimes a_{(0)}\\
&&+a\pii\ot a\ppi+x_{2} \otimes x_{1}+x\boo\ot x\boi+x\bi\ot x\boo+x\qii\ot x\qi) \bullet(b+ y)\\
&&-(b+y)\bullet(a_{2} \otimes a_{1}+a_{(0)} \otimes a_{(-1)}+a_{(1)} \otimes a_{(0)}+a\pii\ot a\ppi+x_{2} \otimes x_{1}\\
&&+x\boo\ot x\boi+x\bi\ot x\boo+x\qii\ot x\qi)+(a+ x) \bullet(b_{1} \otimes b_{2}+b_{(-1)} \otimes b_{(0)}\\
&&+b_{(0)} \otimes b_{(1)}+b\ppi\ot b\pii+y_{1} \otimes y_{2}+y\boi\ot y\boo+y\boo\ot y\bi+y\qi\ot y\qii)\\
&&+[b_{1} \otimes b_{2}+b_{(-1)} \otimes b_{(0)}+b_{(0)} \otimes b_{(1)}+b\ppi\ot b\pii+y_{1} \otimes y_{2}+y\boi\ot y\boo\\
&&+y\boo\ot y\bi+y\qi\ot y\qii, a+x]\\
&=&a_{1}b \otimes a_{2}+(a\li\ppl y)\ot a\lii+(a\li\trr y)\ot a\lii+\theta(a\li, b)\ot a\lii\\
&&+(a\loi \ppr b)\ot a\loo+\sigma(a\loi, y)\ot a\loo+a\loi y\ot a\loo+(a\loi\trl b)\ot a\loo\\
&&+a\loo b\ot a\lmi+(a\loo\ppl y)\ot a\lmi+(a\loo\trr y)\ot a\lmi+\theta(a\loo, b)\ot a\lmi\\
&&+(a\ppi\ppr b)\ot a\pii+\sigma(a\ppi, y)\ot a\pii+a\ppi y\ot a\pii+(a\ppi\trl b)\ot a\pii\\
&&+(x\li\ppr b)\ot x\lii+\sigma(x\li, y)\ot x\lii+x\li y\ot x\lii+(x\li\trl b)\ot x\lii\\
&&+x\boi b\ot x\boo+(x\boi\ppl y)\ot x\boo+(x\boi\trr y)\ot x\boo+\theta(x\boi, b)\ot x\boo\\
&&+(x\boo\ppr b)\ot x\bi+\sigma(x\boo, y)\ot x\bi+x\boo y\ot x\bi+(x\boo\trl b)\ot x\bi\\
&&+x\qi b\ot x\qii+(x\qi\ppl y)\ot x\qii+(x\qi\trr y)\ot x\qii+\theta(x\qi, b)\ot x\qii\\
&&+a\lii b\ot a\li+(a\lii\ppl y)\ot a\li+(a\lii\trr y)\ot a\li+\theta(a\lii, b)\ot a\li\\
&&+a\loo b\ot a\loi+(a\loo\ppl y)\ot a\loi+(a\loo\trr y)\ot a\loi+\theta(a\loo, b)\ot a\loi\\
&&+(a\lmi\ppr b)\ot a\loo+\sigma(a\lmi, y)\ot a\loo+a\lmi y\ot a\loo+(a\li\trl b)\ot a\loo\\
&&+(a\pii\ppr b)\ot a\ppi+\sigma(a\pii, y)\ot a\ppi+a\pii y\ot a\ppi+(a\pii\trl b)\ot a\ppi\\
&&+(x\lii\ppr b)\ot x\li+\sigma(x\lii, y)\ot x\li+x\lii y\ot x\li+(x\lii\trl b)\ot x\li\\
&&+(x\boo\ppr b)\ot x\boi+\sigma(x\boo, y)\ot x\boi+x\boo y\ot x\boi+(x\boo\trl b)\ot x\boi\\
&&+x\bi b\ot x\boo+(x\bi\ppl y)\ot x\boo+(x\bi\trr y)\ot x\boo+\theta(x\bi, b)\ot x\boo\\
&&+x\qii b\ot x\qi+(x\qii\ppl y)\ot x\qi+(x\qii\trr y)\ot x\qi+\theta(x\qii, b)\ot x\qi\\
&&-a\lii\ot b a\li-a\lii\ot (y\ppr a\li)-a\lii\ot(y\trl a\li)-a\lii\ot\theta(b, a\li)\\
&&-a\loo\ot (b\ppl a\loi)-a\loo\ot\sigma(y, a\loi)-a\loo\ot y a\loi-a\loo\ot(b\trr a\loi)\\
&&-a\lmi\ot b a\loo-a\lmi\ot(y\ppr a\loo)-a\lmi\ot (y\trl a\loo)-a\lmi\ot\theta(b, a\loo)\\
&&-a\pii\ot(b\ppl a\ppi)-a\pii\ot\sigma(y, a\ppi)-a\pii\ot y a\ppi-a\pii\ot(b\trr a\ppi)\\
&&-x\lii\ot(b\ppl x\li)-x\lii\ot\sigma(y, x\li)-x\lii\ot y x\li-x\lii\ot (b\trr x\li)\\
&&-x\boo\ot b x\boi-x\boo\ot(y\ppr x\boi)-x\boo\ot(y\trl x\boi)-x\boo\ot\theta(b, x\boi)\\
&&-x\bi\ot(b\ppl x\boo)-x\bi\ot\sigma(y, x\boo)-x\bi\ot y x\boo-x\bi\ot(b\trr x\boo)\\
&&-x\qii\ot b x\qi-x\qii\ot(y\ppr x\qi)-x\qii\ot(y\trl x\qi)-x\qii\ot\theta(b, x\qi)\\
&&+b\li\ot a b\lii+b\li\ot (x\ppr b\lii)+b\li\ot(x\trl b\lii)+b\li\ot\theta(a, b\lii)\\
&&+b\loi\ot a b\loo+b\loi\ot (x\ppr b\loo)+b\loi\ot(x\trl b\loo)+a\loi\ot\theta(a, b\loo)\\
&&+b\loo\ot (a\ppl b\lmi)+b\loo\ot\sigma(x, b\lmi)+b\loo\ot x b\lmi+b\loo\ot(a\trr b\lmi)\\
&&+b\ppi\ot(a\ppl b\pii)+b\ppi\ot\sigma(x, b\pii)+b\ppi\ot x b\pii+b\ppi\ot(a\trr b\pii)\\
&&+y\li\ot (a\ppl y\lii)+y\li\ot\sigma(x, y\lii)+y\li\ot x y\lii+y\li\ot(a\trr y\lii)\\
&&+y\boi\ot(a\ppl y\boo)+y\boi\ot\sigma(x, y\boo)+y\boi\ot x y\boo+y\boi\ot(a\trr y\boo)\\
&&+y\boo\ot a y\bi+y\boo\ot(x\ppr y\bi)+y\boo\ot(x\trl y\bi)+y\boo\ot\theta(a, y\bi)\\
&&+y\qi\ot a y\qii+y\qi\ot(x\ppr y\qii)+y\qi\ot(x\trl y\qii)+y\qi\ot\theta(a, y\qii)\\
&&+b\li\ot b\lii a+b\li\ot (b\lii\ppl x)+b\li\ot(b\lii\trr x)+b\li\ot\theta(b\lii, a)\\
&&-a b\li\ot b\lii-(x\ppr b\li)\ot b\lii-(x\trl b\li)\ot b\lii-\theta(a, b\li)\ot b\lii\\
&&+b\loi\ot b\loo a+b\loi\ot(b\loo\ppl x)+b\loi\ot(b\loo\trr x)+b\loi\ot\theta(b\loo, a)\\
&&-(a\ppl b\loi)\ot b\loo-\sigma(x, b\loi)\ot b\loo-x b\loi\ot b\loo-(a\trr b\loi)\ot b\loo\\
&&+b\loo\ot (b\lmi\ppr a)+b\loo\ot\sigma(b\lmi, x)+b\loo\ot b\lmi x+b\loo\ot(b\lmi\trl a)\\
&&-a b\loo\ot b\lmi-(x\ppr b\loo)\ot b\lmi-(x\trl b\loo)\ot b\lmi-\theta(a, b\loo)\ot b\lmi\\
&&+b\ppi\ot (b\pii\ppr a)+b\ppi\ot \sigma(b\pii , x)+b\ppi\ot b\pii x+b\ppi\ot (b\pii\trl a)\\
&&-(a\ppl b\ppi)\ot b\pii-\sigma(x, b\ppi)\ot b\pii-x b\ppi\ot b\pii-(a\trr b\ppi)\ot b\pii\\
&&+y\li\ot (y\lii\ppr a)+y\li\ot\sigma(y\lii, x)+y\li\ot y\lii x+y\li\ot(y\lii\trl a)\\
&&-(a\ppl y\li)\ot y\lii-\sigma(x, y\li)\ot y\lii-x y\li\ot y\lii-(a\trr y\li)\ot y\lii\\
&&+y\boi\ot(y\boo\ppr a)+y\boi\ot\sigma(y\boo, x)+y\boi\ot y\boo x+y\boi\ot(y\boo\trl a)\\
&&-a y\boi\ot y\boo-(x\ppr y\boi)\ot y\boo-(x\trl y\boi)\ot y\boo-\theta(a, y\boi)\ot y\boo\\
&&+y\boo\ot y\bi a+y\boo\ot(y\bi\ppl x)+y\boo\ot(y\bi\trr x)+y\boo\ot\theta(y\bi, a)\\
&&-(a\ppl y\boo)\ot y\bi-\sigma(x, y\boo)\ot y\bi-x y\boo\ot y\bi-(a\trr y\boo)\ot y\bi\\
&&+y\qi\ot y\qii a+y\qi\ot(y\qii\ppl x)+y\qi\ot(y\qii\trr x)+y\qi\ot\theta(y\qii, a)\\
&&-a y\qi\ot y\qii-(x\ppr y\qi)\ot y\qii-(x\trl y\qi)\ot y\qii-\theta(a, y\qi)\ot y\qii.
\end{eqnarray*}
If we compare both the two sides item by item, one will find all the  cocycle double matched pair conditions (CDM1)--(CDM14) in
Definition \ref{cocycledmp}.

Next, we check the second compatibility condition
$$
\begin{aligned}
&\Delta_{E}((b+y)(a+x))+\tau\Delta_{E}((b+y)(a+x))\\
=&(b+y)\bullet \Delta_{E} (a+x)+(b+y)\cdot\tau\Delta_{E} (a+x)\\
&+ \Delta_{E}(b+y)\bullet (a+x)+\tau \Delta_{E}(b+y)\cdot (a+x).
 \end{aligned}
 $$
The left hand side is equal to
\begin{eqnarray*}
&&\Delta_{E}((b+y)(a+x))+\tau\Delta_{E}((b+y)(a+x))\\
&=&\Delta_E( b a+y \ppr a+b \ppl x+\sigma(y, x)+  y x+y\trl a+b\trr x+\theta(b, a))\\
&&+\tau\Delta_E( b a+y \trr a+b \trl x+\sigma(y, x)+  y x+y\trl a+b\trr x+\theta(b, a))\\
&=&\Delta_A( ba)+\Delta_A(y \ppr a)+\Delta_A(b \ppl x)+\Delta_A(\sigma(y, x))+\phi(b a)+\phi(y \ppr a)\\
&&+\phi(b \ppl x)+\phi(\sigma(y, x))+\psi(b a)+\psi(y \ppr a)+\psi(b \ppl x)+\psi(\sigma(y, x))\\
&&+P(b a)+P(y\ppr a)+P(b\ppl x)+P(\sigma(y, x))+\Delta_{H}(yx)+\Delta_{H}(y\trl a)\\
&&+\Delta_{H}(b\trr x)+\Delta_{H}(\theta(b, a))+\rho(y x)+\rho(y\trl a)+\rho(b\trr x)+\rho(\theta(b, a))\\
&&+\gamma(y x)+\gamma(y\trl a)+\gamma(b\trr x)+\gamma(\theta(b, a))+Q(y x)+Q(y\trl a)\\
&&+Q(b\trr x)+Q(\theta(b, a))+\tau\Delta_A( ba)+\tau\Delta_A(y \ppr a)+\tau\Delta_A(b \ppl x)\\
&&+\tau\Delta_A(\sigma(y, x))+\tau\phi(b a)+\tau\phi(y \ppr a)+\tau\phi(b \ppl x)+\tau\phi(\sigma(y, x))\\
&&+\tau\psi(b a)+\tau\psi(y \ppr a)+\tau\psi(b \ppl x)+\tau\psi(\sigma(y, x))+\tau P(b a)\\
&&+\tau P(y\ppr a)+\tau P(b\ppl x)+\tau P(\sigma(y, x))+\tau\Delta_{H}(yx)+\tau\Delta_{H}(y\trl a)\\
&&+\tau\Delta_{H}(b\trr x)+\tau\Delta_{H}(\theta(b, a))+\tau\rho(y x)+\tau\rho(y\trl a)+\tau\rho(b\trr x)\\
&&+\tau\rho(\theta(b, a))+\tau\gamma(y x)+\tau\gamma(y\trl a)+\tau\gamma(b\trr x)+\tau\gamma(\theta(b, a))\\
&&+\tau Q(y x)+\tau Q(y\trl a)+\tau Q(b\trr x)+\tau Q(\theta(b, a)),
\end{eqnarray*}
and the right hand side is equal to
\begin{eqnarray*}
&&(b+y)\bullet \Delta_{E} (a+x)+(b+y)\cdot\tau\Delta_{E} (a+x)\\
&&+ \Delta_{E}(b+y)\bullet (a+x)+\tau \Delta_{E}(b+y)\cdot (a+x)\\
&=&(b+y)\bullet(a_{1} \otimes a_{2}+a_{(-1)} \otimes a_{(0)}+a_{(0)} \otimes a_{(1)}+a\ppi\ot a\pii+x_{1} \otimes x_{2}\\
&&+x_{[-1]} \otimes x_{[0]}+x_{[0]}\ot x_{[1]}+x\qi\ot x\qii)+(b+y)\cdot(a_{2} \otimes a_{1}+a_{(0)} \otimes a_{(-1)}\\
&&+a_{(1)} \otimes a_{(0)}+a\pii\ot a\ppi+x_{2} \otimes x_{1}+x\boo\ot x\boi+x\bi\ot x\boo+x\qii\ot x\qi)\\
&&+(b_{1} \otimes b_{2}+b_{(-1)} \otimes b_{(0)}+b_{(0)} \otimes b_{(1)}+b\ppi\ot b\pii+y_{1} \otimes y_{2}+y\boi\ot y\boo\\
&&+y\boo\ot y\bi+y\qi\ot y\qii)\bullet(a+x)+(b\lii\ot b\li+b\loo\ot b\loi+b\lmi\ot b\loo\\
&&+b\pii\ot b\ppi+y\lii\ot y\li+y\boo\ot y\boi+y\bi\ot y\boo+y\qii\ot y\qi)\cdot (a+x)\\
&=&a\li\ot b a\lii+a\li\ot(y\ppr a\lii)+a\li\ot(y\trl a\lii)+a\li\ot\theta(b, a\lii)\\
&&+a\loi\ot b a\loo+a\loi\ot (y\ppr a\loo)+a\loi\ot(y\trl a\loo)+a\loi\ot\theta(b, a\loo)\\
&&+a\loo\ot(b\ppl a\lmi)+a\loo\ot\sigma(y, a\lmi)+a\loo\ot y a\lmi+a\loo\ot (b\trr a\lmi)\\
&&+a\ppi\ot(b\ppl a\pii)+a\ppi\ot\sigma(y, a\pii)+a\ppi\ot y a\pii+a\ppi\ot(b\trr a\pii)\\
&&+x\li\ot(b\ppl x\lii)+x\li\ot\sigma(y, x\lii)+x\li\ot y x\lii+x\li\ot(b\trr x\lii)\\
&&+x\boi\ot(b\ppl x\boo)+x\boi\ot\sigma(y, x\boo)+x\boi\ot y x\boo+x\boi\ot(b\trr x\boo)\\
&&+x\boo\ot b x\bi+x\boo\ot(y\ppr x\bi)+x\boo\ot(y\trl x\bi)+x\boo\ot\theta(b, x\bi)\\
&&+x\qi\ot b x\qii+x\qi\ot(y\ppr x\qii)+x\qi\ot(y\trl x\qii)+x\qi\ot\theta(b, x\qii)\\
&&+b a\lii\ot a\li+(y\ppr a\lii)\ot a\li+(y\trl a\lii)\ot a\li+\theta(b, a\lii)\ot a\li\\
&&+b a\loo\ot a\loi+(y\ppr a\loo)\ot a\loi+(y\trl a\loo)\ot a\loi+\theta(b, a\loo)\ot a\loi\\
&&+(b\ppl a\lmi)\ot a\loo+\sigma(y, a\lmi)\ot a\loo+y a\lmi\ot a\loo+(b\trr a\lmi)\ot a\loo\\
&&+(b\ppl a\pii)\ot a\ppi+\sigma(y, a\pii)\ot a\ppi+y a\pii\ot a\ppi+(b\trr a\pii)\ot a\ppi\\
&&+(b\ppl x\lii)\ot x\li+\sigma(y, x\lii)\ot x\li+y x\lii\ot x\li+(b\trr x\lii)\ot x\li\\
&&+(b\ppl x\boo)\ot x\boi+\sigma(y, x\boo)\ot x\boi+y x\boo\ot x\boi+(b\trr x\boo)\ot x\boi\\
&&+b x\bi\ot x\boo+(y\ppr x\bi)\ot x\boo+(y\trl x\bi)\ot x\boo+\theta(b, x\bi)\ot x\boo\\
&&+b x\qii\ot x\qi+(y\ppr x\qii)\ot x\qi+(y\trl x\qii)\ot x\qi+\theta(b, x\qii)\ot x\qi\\
&&+b\li a\ot b\lii+(b\li\ppl x)\ot b\lii+(b\li\trr x)\ot b\lii+\theta(b\li, a)\ot b\lii\\
&&+(b\loi\ppr a)\ot b\loo+\sigma(b\loi, x)\ot b\loo+b\loi x\ot b\loo+(b\loi\trl a)\ot b\loo\\
&&+b\loo a\ot b\lmi+(b\loo\ppl x)\ot b\lmi+(b\loo\trr x)\ot b\lmi+\theta(b\loo, a)\ot b\lmi\\
&&+(b\ppi\ppr a)\ot b\pii+\sigma(b\ppi, x)\ot b\pii+b\ppi x\ot b\pii+(b\ppi\trl a)\ot b\pii\\
&&+(y\li\ppr a)\ot y\lii+\sigma(y\li, x)\ot y\lii+y\li x\ot y\lii+(y\li\trl a)\ot y\lii\\
&&+y\boi a\ot y\boo+(y\boi\ppl x)\ot y\boo+(y\boi\trr x)\ot y\boo+\theta(y\boi, a)\ot y\boo\\
&&+(y\boo\ppr a)\ot y\bi+\sigma(y\boo, x)\ot y\bi+y\boo x\ot y\bi+(y\boo\trl a)\ot y\bi\\
&&+y\qi a\ot y\qii+(y\qi\ppl x)\ot y\qii+(y\qi\trr x)\ot y\qii+\theta(y\qi, a)\ot y\qii\\
&&+b\lii\ot b\li a+b\lii\ot(b\li\ppl x)+b\lii\ot (b\li\trr x)+b\lii\ot \theta(b\li, a)\\
&&+b\loo\ot(b\loi\ppr a)+b\loo\ot\sigma(b\loi, x)+b\loo\ot b\loi x+b\loo\ot(b\loi\trl a)\\
&&+b\lmi\ot b\loo a+b\lmi\ot(b\loo\ppl x)+b\lmi\ot(b\loo\trr x)+b\lmi\ot\theta(b\loo, a)\\
&&+b\pii\ot(b\ppi\ppr a)+b\pii\ot\sigma(b\ppi, x)+b\pii\ot b\ppi x+b\pii\ot(b\ppi\trl a)\\
&&+y\lii\ot(y\li\ppr a)+y\lii\ot\sigma(y\li, x)+y\lii\ot y\li x+y\lii\ot(y\li\trl a)\\
&&+y\boo\ot y\boi a+y\boo\ot(y\boi\ppl x)+y\boo\ot(y\boi\trr x)+y\boo\ot\theta(y\boi, a)\\
&&+y\bi\ot(y\boo\ppr a)+y\bi\ot\sigma(y\boo, x)+y\bi\ot y\boo x+y\bi\ot(y\boo\trl a)\\
&&+y\qii\ot y\qi a+y\qii\ot(y\qi\ppl x)+y\qii\ot(y\qi\trr x)+y\qii\ot\theta(y\qi, a).
\end{eqnarray*}
If we compare both the two sides item by item, one obtain all the  cocycle double matched pair conditions (CDM15)--(CDM28) in
Definition \ref{cocycledmp}.

This complete the proof.
\end{proof}

\section{Extending structures for alternative bialgebras}
In this section, we will study the extending problem for      alternative  bialgebras.
We will find some special cases when the braided  alternative  bialgebra is deduced into an ordinary     alternative  bialgebra.
It is proved that the extending problem can be solved by using of the non-abelian cohomology theory based on our cocycle bicrossedproduct for braided alternative  bialgebras in last section.

\subsection{Extending structures for  alternative  algebras }
First we are going to study extending problem for  alternative  algebras.

There are two cases for $A$ to be an  alternative  algebra in the cocycle cross product system defined in last section, see condition (CC11)--(CC12). The first case is when we let $\ppr$, $\ppl$ to be trivial and $\theta\neq 0$,  then from condition (CP9) we get $\si(x, \theta(a, b))+\si(x, \theta(b, a))=0$, since $\theta\neq 0$ we assume $\sigma=0$ for simplicity, thus  we obtain the following type $(a1)$  unified product for  alternative algebras.

\begin{lemma}
Let ${A}$ be an  alternative  algebra and $V$ a vector space. An extending datum of ${A}$ by $V$ of type (a1)  is  $\Omega^{(1)}({A},V)=(\trr, \trl, \theta)$ consisting of bilinear maps
\begin{eqnarray*}
\trr: A\otimes V \to V, \quad \trl: V\otimes A \to V, \quad\theta: A\ot A\to V.
\end{eqnarray*}
Denote by $A_{}\#_{\theta}V$ the vector space $E={A}\oplus V$ together with the multiplication given by
\begin{eqnarray}
(a+ x)(b+ y)=ab+\big( xy+x\trl b+a\trr y+\theta(a, b)\big).
\end{eqnarray}
Then $A_{}\# {}_{\theta}V$ is an  alternative  algebra if and only if the following compatibility conditions hold for all $a$, $b$, $c\in {A}$, $x$, $y$, $z\in V$:
\begin{enumerate}
\item[(A1)] $(ab +ba)\trr x+(\theta(a, b) +\theta(b, a)) x=a\trr (b\trr x)+b\trr (a\trr x) $,
\item[(A2)] $(ab) \trr x+\theta(a, b) x+(a\trr x)\trl b=a\trr (b\trr x+x\trl b) $,
\item[(A3)] $x \trl (a b)+x\theta(a, b)+a \trr (x\trl b)=(a\trr x+x \trl a) \trl b $,
\item[(A4)] $x \trl (a b+b a)+x(\theta(b, a)+\theta(a, b))=(x \trl a) \trl b+(x \trl b) \trl a $,
\item[(A5)] $ a \trr (x y)+x(a\trr y) =(a \trr x+x\trl a) y $,
\item[(A6)] $ a \trr (x y+y x) =(a \trr x)y+(a \trr y)x$,
\item[(A7)] $(x y+yx)\trl a = x(y \trl a)+y(x \trl a)$,
\item[(A8)] $(x y) \trl a +(x\trl a)y= x(y \trl a+a\trr y)$,
\item[(A9)]$\theta(a b, c)+\theta(ba, c)+(\theta(a, b)+\theta(b,a) ) \triangleleft c=\theta(a, b c)+\theta(b, a c)+b\trr \theta(a, c)+a\trr \theta(b, c),$
\item[(A10)]$\theta(a b, c)+\theta(ac, b)+\theta(a, b) \triangleleft c+\theta(a,c) \triangleleft b=\theta(a, b c)+\theta(a, cb)+a\trr (\theta(c, b)+ \theta(b, c)),$
\item[(A11)] $(x y)z-x(yz)=-(yx)z+y(xz)$,
\item[(A12)] $(x y)z-x(yz)=-(xz)y+x(zy)$.
\end{enumerate}
\end{lemma}
Note that (A1)--(A8)  are deduced from (CP5)--(CP8) , (CP13)--(CP16) and by (A11)--(A12)  we obtain that $V$ is an  alternative  algebra. Furthermore, $V$ is in fact an alternative  subalgebra of $A_{}\#_{\theta}V$  but $A$ is not although $A$ is itself an alternative algebra.

Denote the set of all  algebraic extending datum of ${A}$ by $V$ of type (a1)  by $\mathcal{A}^{(1)}({A}, V)$.


In the following, we always assume that $A$ is a subspace of a vector space $E$, there exists a projection map $p: E \to{A}$ such that $p(a) = a$, for all $a \in {A}$.
Then the kernel space $V : = \ker(p)$ is also a subspace of $E$ and a complement of ${A}$ in $E$.

\begin{lemma}\label{lem:33-1}
Let ${A}$ be an  alternative   algebra and $E$ a vector space containing ${A}$ as a   subspace.
Suppose that there is an  alternative  algebra structure on $E$ such that $V$ is an  alternative  subalgebra of $E$
and the canonical projection map $p: E\to A$ is an  alternative  algebra homomorphism.
Then there exists an  alternative  algebraic extending datum $\Omega^{(1)}({A}, V)$ of ${A}$ by $V$ such that
$E\cong A_{}\#_{\theta}V$.
\end{lemma}

\begin{proof}
Since $V$ is an  alternative  subalgebra of $E$, we have $x\cdot_E y\in V$ for all $x, y\in V$.
We define the extending datum of ${A}$ through $V$ by the following formulas:
\begin{eqnarray*}
\trr: A\otimes {V} \to V, \qquad {a} \trr {x} &: =&{a}\cdot_E {x},\\
\trl: V\otimes {A} \to V, \qquad {x} \triangleleft {a} &: =&{x}\cdot_E {a},\\
\theta: A\otimes A \to V, \qquad \theta(a,b) &: =&a\cdot_E b-p \bigl(a\cdot_E b\bigl),\\
{\cdot_V}: V \otimes V \to V, \qquad {x}\cdot_V {y}&: =& {x}\cdot_E{y} .
\end{eqnarray*}
for any $a , b\in {A}$ and $x, y\in V$. It is easy to see that the above maps are  well defined and
$\Omega^{(1)}({A}, V)$ is an extending system of
${A}$ trough $V$ and
\begin{eqnarray*}
\varphi: A_{}\#_{\theta}V\to E, \qquad \varphi(a+ x) : = a+x
\end{eqnarray*}
is an isomorphism of   alternative  algebras.
\end{proof}

\begin{lemma}
Let $\Omega^{(1)}(A, V) = \bigl(\trr,  \trl, \theta, \cdot \bigl)$ and $\Omega'^{(1)}(A, V) = \bigl(\trr ', \, \trl', \theta', \cdot ' \bigl)$
be two algebraic extending datums of ${A}$ by $V$ of type (a1) and $A_{}\#_{\theta} V$, $A_{}\#_{\theta'} V$ be the corresponding unified products. Then there exists a bijection between the set of all homomorphisms of  alternative  algebras $\varphi: A_{\theta}\#_{\trr, \trl} V\to A_{\theta'}\#_{\trr', \trl'} V$ whose restriction on ${A}$ is the identity map and the set of pairs $(r, s)$, where $r: V\rightarrow {A}$ and $s: V\rightarrow V$ are two linear maps satisfying
\begin{eqnarray}
&&{r}(x\trl a)={r}(x)\cdot' a,\\
&&{r}(a\trr x)= a\cdot'{r}(y),\\
&&a\cdot' b=ab+r\theta(a, b),\\
&&{r}(xy)={r}(x)\cdot' {r}(y),\\
&&{s}(x)\trl' a+\theta'(r(x), a)={s}(x\trl a),\\
&&a\trr'{s}(y)+\theta'(a, r(y) )={s}(a\trr y),\\
&&\theta'(a, b)=s\theta(a, b),\\
&&{s}(xy)={s}(x)\cdot' {s}(y)+{s}(x)\trl'{r}(y)+{r}(x)\trr'{s}(y)+\theta'(r(x), r(y)),
\end{eqnarray}
for all $a, b\in{A}$ and $x$, $y\in V$.

Under the above bijection the homomorphism of   alternative  algebras $\varphi=\varphi_{r, s}: A_{}\#_{\theta}V\to A_{}\#_{\theta'} V$ to $(r, s)$ is given  by $\varphi(a+x)=(a+r(x))+ s(x)$ for all $a\in {A}$ and $x\in V$. Moreover, $\varphi=\varphi_{r, s}$ is an isomorphism if and only if $s: V\rightarrow V$ is a linear isomorphism.
\end{lemma}

\begin{proof}
Let $\varphi: A_{}\#_{\theta}V\to A_{}\#_{\theta'} V$  be an  alternative  algebra homomorphism  whose restriction on ${A}$ is the identity map. Then $\varphi$ is determined by two linear maps $r: V\rightarrow {A}$ and $s: V\rightarrow V$ such that
$\varphi(a+x)=(a+r(x))+s(x)$ for all $a\in {A}$ and $x\in V$.
In fact, we have to show
$$\varphi((a+ x)(b+ y))=\varphi(a+ x)\cdot'\varphi(b+ y).$$
The left hand side is equal to
\begin{eqnarray*}
&&\varphi((a+ x)(b+ y))\\
&=&\varphi\left({ab}+  x\trl b+a\trr y+{xy}+\theta(a, b)\right)\\
&=&{ab}+ r(x\trl b)+r(a\trr y)+r({xy})+r\theta(a, b)\\
&&+ s(x\trl b)+s(a\trr y)+s({xy})+s\theta(a, b),
\end{eqnarray*}
and the right hand side is equal to
\begin{eqnarray*}
&&\varphi(a+ x)\cdot' \varphi(b+ y)\\
&=&(a+r(x)+s(x))\cdot'  (b+r(y)+s(y))\\
&=&(a+r(x))\cdot' (b+r(y))+  s(x)\trl'(b+r(y))+(a+r(x))\trr's(y)\\
&& +s(x)\cdot' s(y)+\theta'(a+r(x), b+r(y)).
\end{eqnarray*}
Thus $\varphi$ is a homomorphism of alternative algebras if and only if the above conditions hold.
\end{proof}

The second case is when $\theta=0$,  we obtain the following type (a2)  unified product.
\begin{theorem}(\cite{Z5})
Let $A$ be an  alternative  algebra and $V$ a  vector space. An \textit{extending
datum of $A$ through $V$} of type (a2)  is a system $\Omega(A, V) =
\bigl(\triangleleft, \, \triangleright, \, \leftharpoonup, \,
\rightharpoonup, \, \sigma \bigl)$ consisting of linear maps
\begin{eqnarray*}
&&\ppr: V \otimes A \to A, \quad \ppl: A \otimes V \to A, \quad\triangleleft : V \otimes A \to V, \quad \trr: A \otimes V \to V ,
 \quad\sigma: V\otimes V \to A.
\end{eqnarray*}
Denote by $A_{}\# {}_{\sigma}V$ the vector space $E={A}\oplus V$ together with the multiplication
\begin{align}
(a+ x)(b+ y)=\big(ab+x\ppr b+a\ppl y+\sigma(x, y)\big)+\big( xy+x\trl b+a\trr y\big).
\end{align}
Then $A_{}\# {}_{\sigma}V$  is an alternative algebra if and only if the following compatibility conditions hold for any $a, b\in A$,
$x, y, z\in V$:
\begin{enumerate}
\item[(B1)] $x\ppr (ab)+a(x\ppr b)+a\ppl(x\trl b)=(x\ppr a+a\ppl x)b+(x\trl a+a\trr x)\ppr b $,
\item[(B2)] $x\ppr (ab+ ba)=(x\ppr a)b+(x\trl a)\ppr b+(x\ppr b)a+(x\trl b)\ppr a $,
\item[(B3)] $(ab) \ppl x+(a\ppl x)b+(a\trr x)\ppr b=a(b\ppl x+x\ppr b)+a\ppl(b\trr x+x\trl b) $,
\item[(B4)] $(ab +ba )\ppl x=a(b\ppl x)+a\ppl(b\trr x)+b(a\ppl x)+b\ppl(a\trr x)$,
\item[(B5)] $(xy)\ppr a+(x\ppr a)\ppl y+\sigma(x\trl a, y)+\sigma(x,y)a\\
=x\ppr(y\ppr a+a\ppl y)+\sigma(x, y\trl a)+\sigma(x, a\trr y)$,
\item[(B6)] $(xy+yx)\ppr a+(\sigma(x, y)+\sigma(y, x))a\\
=x\ppr(y\ppr a)+\sigma(x, y\trl a)+y\ppr(x\ppr a)+\sigma(y, x\trl a)$,
\item[(B7)] $a\ppl (x y)+x\ppr(a\ppl y)+\sigma(x,a\trr y)+a\sigma(x, y)\\
=(a\ppl x+x\ppr a)\ppl y+\sigma(a\trr x, y)+\sigma(x\trl a, y)$,
\item[(B8)] $a\ppl (x y+ y x)+a(\sigma(x, y)+\sigma(y, x))\\
=(a\ppl x)\ppl y+\sigma(a\trr x, y)+(a\ppl y)\ppl x+\sigma(a\trr y, x)$,
\item[(B9)] $x\trl (ab)+a\trr(x\trl b)=(x\trl a+a\trr x)\trl b  $,
\item[(B10)] $x\trl (ab+ba)=(x\trl a)\trl b+(x\trl b)\trl a  $,
\item[(B11)] $(ab)\trr x+(a\trr x)\trl b=a\trr(b\trr x+x\trl b) $,
\item[(B12)] $(ab+ba)\trr x=a\trr(b\trr x)+b\trr(a\trr x)$,
\item[(B13)] $(xy)\trl a+(x\trl a)y+(x\ppr a)\trr y\\
=x(y\trl a+a\trr y)+x\trl(y\ppr a+a\ppl y)$,
\item[(B14)] $(xy+yx)\trl a=x(y\trl a)+x\trl(y\ppr a)+y(x\trl a)+y\trl(x\ppr a)$,
\item[(B15)] $a\trr (xy)+x(a\trr y)+x\trl(a\ppl y)\\
=(a\ppl x+x\ppr a)\trr y+(a\trr x+x\trl a)y$,
\item[(B16)] $a\trr (xy+ yx)=(a\ppl x)\trr y+(a\trr x)y+(a\ppl y)\trr x+(a\trr y)x$,
\item[(B17)] $\sigma(x y, z)+(\sigma(y,x)+\sigma(x, y))\ppl z+\sigma(yx, z)\\
=\sigma(y,xz)+{\sigma}(x, y z)+y\ppr\sigma(x,z)+x \ppr \sigma(y, z),$
\item[(B18)] $\sigma(x y, z)+\sigma(x z, y)+\sigma(x, y)\ppl z+\sigma(x, z)\ppl y\\
={\sigma}(x, z y)+{\sigma}(x, y z)+x \ppr (\sigma(z, y)+ \sigma(y, z))$,
\item[(B19)] $(x y) z+(y x)z+(\sigma(x,y)+\sigma(y, x))\trr z\\
=x(y z)+y(x z)+y\trl\sigma(x, z)+x\trl\sigma(y, z),$
\item[(B20)] $(x y) z+(x z)y+\sigma(x,y)\trr z+\sigma(x, z)\trr y\\
=x(y z)+x(z y)+x\trl(\sigma(z, y)+\sigma(y, z))$.
\end{enumerate}
\end{theorem}

\begin{theorem}(\cite{Z5})
Let $A$ be an  alternative  algebra, $E$ a  vector space containing $A$ as a subspace.
If there is an  alternative  algebra structure on $E$ such that $A$ is an  alternative  subalgebra of $E$. Then there exists an  alternative  algebraic extending structure $\Omega(A, V) = \bigl(\triangleleft, \, \triangleright, \,
\leftharpoonup, \, \rightharpoonup, \, \sigma \bigl)$ of $A$ through $V$ such that there is an isomorphism of    alternative  algebras $E\cong A_{\sigma}\#_{}H$.
\end{theorem}

\begin{lemma}
Let $\Omega^{(1)}(A, V) = \bigl(\trr, \trl, \leftharpoonup,  \rightharpoonup,  \sigma,  \cdot \bigl)$ and
$\Omega'^{(1)}(A, V) = \bigl(\trr', \trl ', \leftharpoonup ',  \rightharpoonup ', \sigma ', \cdot ' \bigl)$
be two  algebraic extending structures of $A$ through $V$ and $A{}_{\sigma}\#_{}V$, $A{}_{\sigma'}\#_{}  V$ the  associated unified
products. Then there exists a bijection between the set of all
homomorphisms of algebras $\psi: A{}_{\sigma}\#_{}V\to A{}_{\sigma'}\#_{}  V$which
stabilize $A$ and the set of pairs $(r, s)$, where $r: V \to
A$, $s: V \to V$ are linear maps satisfying the following
compatibility conditions for any $a \in A$, $x$, $y \in V$:
\begin{enumerate}
\item[(M1)] $r(x \cdot y) = r(x)\cdot'r(y) + \sigma ' (s(x), s(y)) - \sigma(x, y) + r(x) \ppl' s(y) + s(x) \ppr' r(y)$,
\item[(M2)] $s(x \cdot y) = r(x) \trr ' s(y) + s(x)\trl ' r(y) + s(x) \cdot ' s(y)$,
 \item[(M3)] $r(x\trl  {a}) = r(x)\cdot' {a} - x \ppr {a} + s(x) \ppr' {a}$,
  \item[(M5)] $r({a} \trr x) = {a}\cdot'r(x) - {a}\ppl x + {a} \ppl' s(x)$,
 \item[(M4)] $s(x\trl {a}) = s(x)\trl' {a}$,
 \item[(M6)] $s({a}\trr x) = {a} \trr' s(x)$.
\end{enumerate}
Under the above bijection the homomorphism of algebras $\varphi =\varphi _{(r, s)}: A_{\sigma}\# {}_{}H \to A_{\sigma'}\# {}_{}H$ corresponding to
$(r, s)$ is given for any $a\in A$ and $x \in V$ by:
$$\varphi(a+ x) = (a + r(x))+ s(x).$$
Moreover, $\varphi  = \varphi _{(r, s)}$ is an isomorphism if and only if $s: V \to V$ is an isomorphism linear map.
\end{lemma}

Let ${A}$ be an  alternative  algebra and $V$ a vector space. Two algebraic extending systems $\Omega^{(i)}({A}, V)$ and ${\Omega'^{(i)}}({A}, V)$  are called equivalent if $\varphi_{r, s}$ is an isomorphism.  We denote it by $\Omega^{(i)}({A}, V)\equiv{\Omega'^{(i)}}({A}, V)$.
From the above lemmas, we obtain the following result.

\begin{theorem}\label{thm3-1}
Let ${A}$ be an  alternative  algebra, $E$ a vector space containing ${A}$ as a subspace and
$V$ be a complement of ${A}$ in $E$.
Denote $\mathcal{HA}(V, {A}):=\mathcal{A}^{(1)}({A}, V)\sqcup \mathcal{A}^{(2)}({A}, V) /\equiv$. Then the map
\begin{eqnarray}
\notag&&\Psi: \mathcal{HA}(V, {A})\rightarrow Extd(E,{A}),\\
&&\overline{\Omega^{(1)}({A}, V)}\mapsto A_{}\#_{\theta} V, \quad \overline{\Omega^{(2)}({A}, V)}\mapsto A_{\sigma}\# {}_{} V
\end{eqnarray}
is bijective, where $\overline{\Omega^{(i)}({A}, V)}$ is the equivalence class of $\Omega^{(i)}({A}, V)$ under $\equiv$.
\end{theorem}

\subsection{Extending structures for   alternative  coalgebras}

Next we consider the  alternative  coalgebra structures on $E=A^{P}\# {}^{Q}V$.

There are two cases for $(A, \Delta_A)$ to be an  alternative  coalgebra. The first case is  when $Q=0$,  then we obtain the following type (c1) unified coproduct for  alternative  coalgebras.
\begin{lemma}\label{cor02co}
Let $({A}, \Delta_A)$ be an  alternative  coalgebra and $V$ a vector space.
An  extending datum  of ${A}$ by $V$ of  type (c1) is  $\Omega^c({A}, V)=(\phi, {\psi}, \rho, \gamma, P, \Delta_V)$ with  linear maps
\begin{eqnarray*}
&&\phi: A \to V \otimes A, \quad  \psi: A \to A\otimes V,\\
&&\rho: V  \to A\otimes V, \quad  \gamma: V \to V \otimes A,\\
&& {P}: A\rightarrow {V}\otimes {V}, \quad\Delta_V: V\rightarrow V\otimes V.
\end{eqnarray*}
 Denote by $A^{P}\# {}^{} V$ the vector space $E={A}\oplus V$ with the linear map
$\Delta_E: E\rightarrow E\otimes E$ given by
$$\Delta_{E}(a)=(\Delta_{A}+\phi+\psi+P)(a),\quad \Delta_{E}(x)=(\Delta_{V}+\rho+\gamma)(x), $$
that is
$$\Delta_{E}(a)= a\li \ot a\lii+ a\moi \ot a\mo+a\mo\ot a\mi+a\ppi\ot a\pii,$$
$$\Delta_{E}(x)= x\li \ot x\lii+ x\boi \ot x\boo+x\boo \ot x\bi.$$
Then $A^{P}\# {}^{} V$  is an  alternative  coalgebra with the comultiplication given above if and only if the following compatibility conditions hold:
\begin{enumerate}

\item[(C1)] $\phi(a\li)\ot a\lii+\gamma(a\loi)\ot a\loo-a\loi\ot \Delta_{A}(a\loo)\\
=-\tau_{12}\big(\psi(a\li)\ot a\lii+\rho(a\loi)\ot a\loo-a\li\ot\phi(a\lii)-a\loo\ot\gamma(a\lmi)\big)$,

\item[(C2)] $P(a\li)\ot a\lii+\Delta_{V}(a\loi)\ot a\loo-a\loi\ot\phi(a\loo)-a\ppi\ot\gamma(a\pii)\\
=-\tau_{12}\big(P(a\li)\ot a\lii+\Delta_{V}(a\loi)\ot a\loo-a\loi\ot\phi(a\loo)-a\ppi\ot\gamma(a\pii)\big)$,

\item[(C3)] $\Delta_{A}(a\loo)\ot a\lmi-a\li\ot\psi(a\lii)-a\loo\ot\rho(a\lmi)\\
=-\tau_{12}\big(\Delta_{A}(a\loo)\ot a\lmi-a\li\ot\psi(a\lii)-a\loo\ot\rho(a\lmi)\big)$,

\item[(C4)] $\psi(a\loo)\ot a\lmi+\rho(a\ppi)\ot a\pii-a\li\ot P(a\lii)-a\loo\ot\Delta_{V}(a\lmi)\\
=-\tau_{12}\big(\phi(a\loo)\ot a\lmi+\gamma(a\ppi)\ot a\pii-a\loi\ot\psi(a\loo)-a\ppi\ot\rho(a\pii)\big)$,

\item[(C5)] $\gamma(x\boo)\ot x\bi-x\boo\ot\Delta_{A}(x\bi)=-\tau_{12}\big(\rho(x\boo)\ot x\bi-x\boi\ot\gamma(x\boo)\big)$,

\item[(C6)] $\Delta_{V}(x\boo)\ot x\bi-x\li\ot \gamma(x\lii)-x\boo\ot\phi(x\bi)\\
=-\tau_{12}\big(\Delta_{V}(x\boo)\ot x\bi-x\li\ot \gamma(x\lii)-x\boo\ot\phi(x\bi)\big)$,

\item[(C7)] $\Delta_{A}(x\boi)\ot x\boo-x\boi\ot \rho(x\boo)=-\tau_{12}\big(\Delta_{A}(x\boi)\ot x\boo-x\boi\ot \rho(x\boo)\big)$,

\item[(C8)] $\rho(x\li)\ot x\lii+\psi(x\boi)\ot x\boo-x\boi\ot\Delta_{V}(x\boo)\\
=-\tau_{12}\big(\gamma(x\li)\ot x\lii+\phi(x\boi)\ot x\boo-x\li\ot\rho(x\lii)-x\boo\ot\psi(x\bi)\big)$,

\item[(C9)] $\phi(a\li)\ot a\lii+\gamma(a\loi)\ot a\loo-a\loi\ot \Delta_{A}(a\loo)\\
=-\tau_{23}\big(\phi(a\li)\ot a\lii+\gamma(a\loi)\ot a\loo-a\loi\ot \Delta_{A}(a\loo)\big)$,

\item[(C10)] $P(a\li)\ot a\lii+\Delta_{V}(a\loi)\ot a\loo-a\loi\ot\phi(a\loo)-a\ppi\ot\gamma(a\pii)\\
=-\tau_{23}\big(\phi(a\loo)\ot a\lmi+\gamma(a\ppi)\ot a\pii-a\loi\ot\psi(a\loo)-a\ppi\ot\rho(a\pii)\big)$,

\item[(C11)] $\Delta_{A}(a\loo)\ot a\lmi-a\li\ot\psi(a\lii)-a\loo\ot\rho(a\lmi)\\
=-\tau_{23}\big(\psi(a\li)\ot a\lii+\rho(a\loi)\ot a\loo-a\li\ot\phi(a\lii)-a\loo\ot\gamma(a\lmi)\big)$,

\item[(C12)] $\psi(a\loo)\ot a\lmi+\rho(a\ppi)\ot a\pii-a\li\ot P(a\lii)-a\loo\ot\Delta_{V}(a\lmi)\\
=-\tau_{23}\big(\psi(a\loo)\ot a\lmi+\rho(a\ppi)\ot a\pii-a\li\ot P(a\lii)-a\loo\ot\Delta_{V}(a\lmi)\big)$,

\item[(C13)] $\gamma(x\boo)\ot x\bi-x\boo\ot\Delta_{A}(x\bi)=-\tau_{23}\big(\gamma(x\boo)\ot x\bi-x\boo\ot\Delta_{A}(x\bi)\big)$,

\item[(C14)] $\Delta_{V}(x\boo)\ot x\bi-x\li\ot \gamma(x\lii)-x\boo\ot\phi(x\bi)\\
=-\tau_{23}\big(\phi(x\boi)\ot x\boo+\gamma(x\li)\ot x\lii-x\li\ot\rho(x\lii)-x\boo\ot\psi(x\bi)\big)$,

\item[(C15)] $\Delta_{A}(x\boi)\ot x\boo-x\boi\ot \rho(x\boo)=-\tau_{23}\big(\rho(x\boo)\ot x\bi-x\boi\ot\gamma(x\boo)\big)$,

\item[(C16)] $\rho(x\li)\ot x\lii+\psi(x\boi)\ot x\boo-x\boi\ot\Delta_{V}(x\boo)\\
=-\tau_{23}\big(\rho(x\li)\ot x\lii+\psi(x\boi)\ot x\boo-x\boi\ot\Delta_{V}(x\boo)\big)$,

\item[(C17)]  $\Delta_V(a\ppi)\ot a\pii+P(a\lmoo)\ot a\mi-a\lmoi\ot P(a\lmoo)-a\ppi\ot \Delta_V(a\pii)\\
=-\tau_{12}\big(\Delta_V(a\ppi)\ot a\pii+P(a\lmoo)\ot a\mi-a\lmoi\ot P(a\lmoo)-a\ppi\ot \Delta_V(a\pii)\big)$,

\item[(C18)]  $\Delta_V(a\ppi)\ot a\pii+P(a\lmoo)\ot a\mi-a\lmoi\ot P(a\lmoo)-a\ppi\ot \Delta_V(a\pii)\\
=-\tau_{23}\big(\Delta_V(a\ppi)\ot a\pii+P(a\lmoo)\ot a\mi-a\lmoi\ot P(a\lmoo)-a\ppi\ot \Delta_V(a\pii)\big)$,

\item[(C19)] $\Delta_V(x\li)\ot x\lii+ P(x\boi) \ot x\boo-x_1\ot \Delta_V(x_2)-x\boo\ot P(x\bi)\\
=-\tau_{12}\big(\Delta_V(x\li)\ot x\lii+ P(x\boi) \ot x\boo-x_1\ot \Delta_V(x_2)-x\boo\ot P(x\bi)\big)$,

\item[(C20)] $\Delta_V(x\li)\ot x\lii+ P(x\boi) \ot x\boo-x_1\ot \Delta_V(x_2)-x\boo\ot P(x\bi)\\
=-\tau_{23}\big(\Delta_V(x\li)\ot x\lii+ P(x\boi) \ot x\boo-x_1\ot \Delta_V(x_2)-x\boo\ot P(x\bi)\big)$.
\end{enumerate}
\end{lemma}
Denote the set of all  coalgebraic extending datum of ${A}$ by $V$ of type (c1) by $\mathcal{C}^{(1)}({A},V)$.

\begin{lemma}\label{lem:33-3}
Let $({A}, \Delta_A)$ be an  alternative  coalgebra and $E$ a vector space containing ${A}$ as a subspace. Suppose that there is an   alternative   coalgebra structure $(E, \Delta_E)$ on $E$ such that  $p: E\to {A}$ is an   alternative   coalgebra homomorphism. Then there exists an alternative coalgebraic extending system $\Omega^c({A}, V)$ of $({A}, \Delta_A)$ by $V$ such that $(E, \Delta_E)\cong A^{P}\# {}^{} V$.
\end{lemma}

\begin{proof}
Let $p: E\to {A}$ and $\pi: E\to V$ be the projection map and $V=\ker({p})$.
Then the extending datum of $({A}, \Delta_A)$ by $V$ is defined as follows:
\begin{eqnarray*}
&&{\phi}: A\rightarrow V\ot {A},~~~~{\phi}(a)=(\pi\otimes {p})\Delta_E(a),\\
&&{\psi}: A\rightarrow A\ot V,~~~~{\psi}(a)=({p}\otimes \pi)\Delta_E(a),\\
&&{\rho}: V\rightarrow A\ot V,~~~~{\rho}(x)=({p}\otimes \pi)\Delta_E(x),\\
&&{\gamma}: V\rightarrow V\ot {A},~~~~{\gamma}(x)=(\pi\otimes {p})\Delta_E(x),\\
&&\Delta_V: V\rightarrow V\otimes V,~~~~\Delta_V(x)=(\pi\otimes \pi)\Delta_E(x),\\
&&Q: V\rightarrow {A}\otimes {A},~~~~Q(x)=({p}\otimes {p})\Delta_E(x),\\
&&P: A\rightarrow {V}\otimes {V},~~~~P(a)=({\pi}\otimes {\pi})\Delta_E(a).
\end{eqnarray*}
One check that  $\varphi: A^{P}\# {}^{} V\to E$ given by $\varphi(a+x)=a+x$ for all $a\in A, x\in V$ is an   alternative   coalgebra isomorphism.
\end{proof}

\begin{lemma}\label{lem-c1}
Let $\Omega^{(1)}({A}, V)=(\phi, {\psi}, \rho, \gamma, P, \Delta_V)$ and ${\Omega'^{(1)}}({A}, V)=(\phi', {\psi'}, \rho', \gamma',  P', \Delta'_V)$ be two  alternative   coalgebraic extending datums of $({A}, \Delta_A)$ by $V$. Then there exists a bijection between the set of   alternative    coalgebra homomorphisms $\varphi: A^{P}\# {}^{} V\rightarrow A^{P'}\# {}^{} V$ whose restriction on ${A}$ is the identity map and the set of pairs $(r, s)$, where $r: V\rightarrow {A}$ and $s:V\rightarrow V$ are two linear maps satisfying
\begin{eqnarray}
\label{comorph11}&&P'(a)=s(a\ppi)\ot s(a\pii),\\
\label{comorph121}&&\phi'(a)={s}(a\lmoi)\ot a\lmo+s(a\ppi)\ot r(a\pii),\\
\label{comorph122}&&\psi'(a)=a\lmo\ot {s}(a\mi) +r(a\ppi)\ot s(a\pii),\\
\label{comorph13}&&\Delta'_A(a)=\Delta_A(a)+{r}(a\lmoi)\ot a\lmo+a\lmo\ot {r}(a\mi)+r(a\ppi)\ot r(a\pii),\\
\label{comorph21}&&\Delta_V'({s}(x))=({s}\otimes {s})\Delta_V(x),\\
\label{comorph221}&&{\rho}'({s}(x))+\psi'(r(x))=r(x\li)\ot s(x\lii)+x\boi\ot s(x\boo),\\
\label{comorph222}&&{\gamma}'({s}(x))+\phi'(r(x))=s(x\li)\ot r(x\lii)+s(x\boo)\ot x\bi,\\
\label{comorph23}&&\Delta'_A({r}(x))+P'(r(x))=r(x\li)\ot r(x\lii)+x\boi\ot r(x\boo)+r(x\boo)\ot x\bi.
\end{eqnarray}
Under the above bijection the   alternative   coalgebra homomorphism $\varphi=\varphi_{r, s}: A^{P}\# {}^{} V\rightarrow A^{P'}\# {}^{} V$ to $(r, s)$ is given by $\varphi(a+x)=(a+r(x))+s(x)$ for all $a\in {A}$ and $x\in V$. Moreover, $\varphi=\varphi_{r, s}$ is an isomorphism if and only if $s: V\rightarrow V$ is a linear isomorphism.
\end{lemma}
\begin{proof}
Let $\varphi: A^{P}\# {}^{} V\rightarrow A^{P'}\# {}^{} V$  be an    alternative  coalgebra homomorphism  whose restriction on ${A}$ is the identity map. Then $\varphi$ is determined by two linear maps $r: V\rightarrow {A}$ and $s: V\rightarrow V$ such that
$\varphi(a+x)=(a+r(x))+s(x)$ for all $a\in {A}$ and $x\in V$. We will prove that
$\varphi$ is a homomorphism of   alternative   coalgebras if and only if the above conditions hold.
First   it is easy to see that  $\Delta'_E\varphi(a)=(\varphi\otimes \varphi)\Delta_E(a)$ for all $a\in {A}$.
\begin{eqnarray*}
\Delta'_E\varphi(a)&=&\Delta'_E(a)=\Delta'_A(a)+\phi'(a)+\psi'(a)+P'(a),
\end{eqnarray*}
and
\begin{eqnarray*}
&&(\varphi\otimes \varphi)\Delta_E(a)\\
&=&(\varphi\otimes \varphi)\left(\Delta_A(a)+\phi(a)+\psi(a)+P(a)\right)\\
&=&\Delta_A(a)+{r}(a\lmoi)\ot a\lmo+{s}(a\lmoi)\ot a\lmo+a\lmo\ot {r}(a\mi) +a\lmo\ot {s}(a\mi)\\
&&+r(a\ppi)\ot r(a\pii)+r(a\ppi)\ot s(a\pii)+s(a\ppi)\ot r(a\pii)+s(a\ppi)\ot s(a\pii).
\end{eqnarray*}
Thus we obtain that $\Delta'_E\varphi(a)=(\varphi\otimes \varphi)\Delta_E(a)$  if and only if the conditions \eqref{comorph11}, \eqref{comorph121}, \eqref{comorph122} and \eqref{comorph13} hold.
Then we consider that $\Delta'_E\varphi(x)=(\varphi\otimes \varphi)\Delta_E(x)$ for all $x\in V$.
\begin{eqnarray*}
\Delta'_E\varphi(x)&=&\Delta'_E({r}(x)+{s}(x))=\Delta'_E({r}(x))+\Delta'_E({s}(x))\\
&=&\Delta'_A({r}(x))+\phi'(r(x))+\psi'(r(x))+P(r(x))+\Delta'_V({s}(x))+{\rho}'({s}(x))+{\gamma}'({s}(x))),
\end{eqnarray*}
and
\begin{eqnarray*}
&&(\varphi\otimes \varphi)\Delta_E(x)\\
&=&(\varphi\otimes \varphi)(\Delta_V(x)+{\rho}(x)+{\gamma}(x))\\
&=&(\varphi\otimes \varphi)(x\li\ot x\lii+x\boi\ot x\boo+x\boo\ot x\bi)\\
&=&r(x\li)\ot r(x\lii)+r(x\li)\ot s(x\lii)+s(x\li)\ot r(x\lii)+s(x\li)\ot s(x\lii)\\
&&+x\boi\ot r(x\boo)+x\boi\ot s(x\boo)+r(x\boo)\ot x\bi+s(x\boo)\ot x\bi.
\end{eqnarray*}
Thus we obtain that $\Delta'_E\varphi(x)=(\varphi\otimes \varphi)\Delta_E(x)$ if and only if the conditions  \eqref{comorph21},  \eqref{comorph221},  \eqref{comorph222} and \eqref{comorph23}  hold. By definition, we obtain that $\varphi=\varphi_{r, s}$ is an isomorphism if and only if $s: V\rightarrow V$ is a linear isomorphism.
\end{proof}

The second case is $\phi=0$ and  $\psi=0$, we obtain  the following type (c2) unified coproduct for  coalgebras.
\begin{lemma}\label{cor02}
Let $({A}, \Delta_A)$ be an   alternative   coalgebra and $V$ a vector space.
An  extending datum  of $({A}, \Delta_A)$ by $V$ of type (c2)  is  $\Omega^{(2)}({A},V)=(\rho, \gamma, {Q}, \Delta_V)$ with  linear maps
\begin{eqnarray*}
&&\rho: V  \to A\otimes V, \quad  \gamma: V \to V \otimes A, \quad \Delta_{V}: V \to V\otimes V, \quad Q: V \to A\otimes A.
\end{eqnarray*}
 Denote by $A^{}\# {}^{Q} V$ the vector space $E={A}\oplus V$ with the comultiplication
$\Delta_E: E\rightarrow E\otimes E$ given by
\begin{eqnarray}
\Delta_{E}(a)&=&\Delta_{A}(a), \quad \Delta_{E}(x)=(\Delta_{V}+\rho+\gamma+Q)(x), \\
\Delta_{E}(a)&=& a\li \ot a\lii, \quad \Delta_{E}(x)= x\li \ot x\lii+ x\boi \ot x\boo+x\boo \ot x\bi+x\qi\ot x\qii.
\end{eqnarray}
Then $A^{}\# {}^{Q} V$  is an  alternative    coalgebra with the comultiplication given above if and only if the following compatibility conditions hold:

\begin{enumerate}
\item[(D1)] $\gamma(x\boo)\ot x\bi-x\li\ot Q(x\lii)-x\boo\ot\Delta_{A}(x\bi)=-\tau_{12}\big(\rho(x\boo)\ot x\bi-x\boi\ot\gamma(x\boo)\big)$,

\item[(D2)] $\Delta_{V}(x\boo)\ot x\bi-x\li\ot \gamma(x\lii)=-\tau_{12}\big(\Delta_{V}(x\boo)\ot x\bi-x\li\ot \gamma(x\lii)\big)$,

\item[(D3)] $Q(x\li)\ot x\lii+\Delta_{A}(x\boi)\ot x\boo-x\boi\ot \rho(x\boo)\\
=-\tau_{12}\big(Q(x\li)\ot x\lii+\Delta_{A}(x\boi)\ot x\boo-x\boi\ot \rho(x\boo)\big)$,

\item[(D4)] $\rho(x\li)\ot x\lii-x\boi\ot\Delta_{V}(x\boo)=-\tau_{12}\big(\gamma(x\li)\ot x\lii-x\li\ot\rho(x\lii)\big)$,

\item[(D5)] $\gamma(x\boo)\ot x\bi-x\li\ot Q(x\lii)-x\boo\ot\Delta_{A}(x\bi)\\
=-\tau_{23}\big(\gamma(x\boo)\ot x\bi-x\li\ot Q(x\lii)-x\boo\ot\Delta_{A}(x\bi)\big)$,

\item[(D6)] $\Delta_{V}(x\boo)\ot x\bi-x\li\ot \gamma(x\lii)=-\tau_{23}\big(\gamma(x\li)\ot x\lii-x\li\ot\rho(x\lii)\big)$,

\item[(D7)] $Q(x\li)\ot x\lii+\Delta_{A}(x\boi)\ot x\boo-x\boi\ot \rho(x\boo)=-\tau_{23}\big(\rho(x\boo)\ot x\bi-x\boi\ot\gamma(x\boo)\big)$,

\item[(D8)] $\rho(x\li)\ot x\lii-x\boi\ot\Delta_{V}(x\boo)=-\tau_{23}\big(\rho(x\li)\ot x\lii-x\boi\ot\Delta_{V}(x\boo)\big)$,

\item[(D9)]  $ \Delta_A(x\qi)\ot x\qii+Q(x\boo)\ot x\bi-x\qi\ot \Delta_A(x\qii)-x\boi\ot Q(x\boo)\\
=-\tau_{12}\big(\Delta_A(x\qi)\ot x\qii+Q(x\boo)\ot x\bi-x\qi\ot \Delta_A(x\qii)-x\boi\ot Q(x\boo)\big)$,

\item[(D10)]  $ \Delta_A(x\qi)\ot x\qii+Q(x\boo)\ot x\bi-x\qi\ot \Delta_A(x\qii)-x\boi\ot Q(x\boo)\\
=-\tau_{23}\big(\Delta_A(x\qi)\ot x\qii+Q(x\boo)\ot x\bi-x\qi\ot \Delta_A(x\qii)-x\boi\ot Q(x\boo)\big)$,

\item[(D11)] $\Delta_V(x\li)\ot x\lii-x_1\ot \Delta_V(x_2)=-\tau_{12}\big(\Delta_V(x\li)\ot x\lii-x_1\ot \Delta_V(x_2)\big)$,

\item[(D12)] $\Delta_V(x\li)\ot x\lii-x_1\ot \Delta_V(x_2)=-\tau_{23}\big(\Delta_V(x\li)\ot x\lii-x_1\ot \Delta_V(x_2)\big)$.
\end{enumerate}
\end{lemma}
Note that in this case $(V, \Delta_V)$ is an  alternative  coalgebra.

Denote the set of all  alternative  coalgebraic extending datum of ${A}$ by $V$ of type (c2) by $\mathcal{C}^{(2)}({A}, V)$.

Similar to the  alternative   algebra case,  one  show that any   alternative   coalgebra structure on $E$ containing ${A}$ as an   alternative   subcoalgebra is isomorphic to such a unified coproduct.
\begin{lemma}\label{lem:33-4}
Let $({A}, \Delta_A)$ be an    alternative  coalgebra and $E$ a vector space containing ${A}$ as a subspace. Suppose that there is an   alternative   coalgebra structure $(E, \Delta_E)$ on $E$ such that  $({A}, \Delta_A)$ is an   alternative  subcoalgebra of $E$. Then there exists an   alternative   coalgebraic extending system $\Omega^{(2)}({A}, V)$ of $({A}, \Delta_A)$ by $V$ such that $(E, \Delta_E)\cong A^{}\# {}^{Q} V$.
\end{lemma}

\begin{proof}
Let $p: E\to {A}$ and $\pi: E\to V$ be the projection map and $V=ker({p})$.
Then the extending datum of $({A}, \Delta_A)$ by $V$ is defined as follows:
\begin{eqnarray*}
&&{\rho}: V\rightarrow A\ot V,~~~~{\phi}(x)=(p\otimes {\pi})\Delta_E(x),\\
&&{\gamma}: V\rightarrow V\ot {A},~~~~{\phi}(x)=(\pi\otimes {p})\Delta_E(x),\\
&&\Delta_V: V\rightarrow V\otimes V,~~~~\Delta_V(x)=(\pi\otimes \pi)\Delta_E(x),\\
&&Q: V\rightarrow {A}\otimes {A},~~~~Q(x)=({p}\otimes {p})\Delta_E(x).
\end{eqnarray*}
One check that  $\varphi: A^{}\# {}^{Q} V\to E$ given by $\varphi(a+x)=a+x$ for all $a\in A, x\in V$ is an   alternative   coalgebra isomorphism.
\end{proof}

\begin{lemma}\label{lem-c2}
Let $\Omega^{(2)}({A}, V)=(\rho, \gamma, {Q}, \Delta_V)$ and ${\Omega'^{(2)}}({A}, V)=(\rho', \gamma', {Q'}, \Delta'_V)$ be two   alternative   coalgebraic extending datums of $({A}, \Delta_A)$ by $V$. Then there exists a bijection between the set of   alternative    coalgebra homomorphisms $\varphi: A \# {}^{Q} V\rightarrow A \# {}^{Q'} V$ whose restriction on ${A}$ is the identity map and the set of pairs $(r, s)$, where $r: V\rightarrow {A}$ and $s:V\rightarrow V$ are two linear maps satisfying
\begin{eqnarray}
\label{comorph1}&&{\rho}'({s}(x))=r(x\li)\ot s(x\lii)+x\boi\ot s(x\boo),\\
\label{comorph2}&&{\gamma}'({s}(x))=s(x\li)\ot r(x\lii)+s(x\boo)\ot x\bi,\\
\label{comorph3}&&\Delta_V'({s}(x))=({s}\otimes {s})\Delta_V(x),\\
\label{comorph4}&&\Delta'_A({r}(x))+{Q'}({s}(x))=r(x\li)\ot r(x\lii)+x\boi\ot r(x\boo)+r(x\boo)\ot x\bi+{Q}(x).
\end{eqnarray}
Under the above bijection the    alternative  coalgebra homomorphism $\varphi=\varphi_{r, s}: A^{ }\# {}^{Q} V\rightarrow A^{ }\# {}^{Q'} V$ to $(r, s)$ is given by $\varphi(a+x)=(a+r(x))+s(x)$ for all $a\in {A}$ and $x\in V$. Moreover, $\varphi=\varphi_{r, s}$ is an isomorphism if and only if $s: V\rightarrow V$ is a linear isomorphism.
\end{lemma}
\begin{proof} The proof is similar as the proof of Lemma \ref{lem-c1}.
Let $\varphi: A^{ }\# {}^{Q} V\rightarrow A^{}\# {}^{Q'} V$  be an   alternative   coalgebra homomorphism  whose restriction on ${A}$ is the identity map.
First  it is easy to see that  $\Delta'_E\varphi(a)=(\varphi\otimes \varphi)\Delta_E(a)$ for all $a\in {A}$.
Then we consider that $\Delta'_E\varphi(x)=(\varphi\otimes \varphi)\Delta_E(x)$ for all $x\in V$.
\begin{eqnarray*}
\Delta'_E\varphi(x)&=&\Delta'_E({r}(x)+{s}(x))=\Delta'_E({r}(x))+\Delta'_E({s}(x))\\
&=&\Delta'_A({r}(x))+\Delta'_V({s}(x))+{\rho}'({s}(x))+{\gamma}'({s}(x))+{Q}'({s}(x)),
\end{eqnarray*}
and
\begin{eqnarray*}
&&(\varphi\otimes \varphi)\Delta_E(x)\\
&=&(\varphi\otimes \varphi)(\Delta_V(x)+{\rho}(x)+{\gamma}(x)+{Q}(x))\\
&=&(\varphi\otimes \varphi)(x\li\ot x\lii+x\boi\ot x\boo+x\boo\ot x\bi+{Q}(x))\\
&=&r(x\li)\ot r(x\lii)+r(x\li)\ot s(x\lii)+s(x\li)\ot r(x\lii)+s(x\li)\ot s(x\lii)\\
&&+x\boi\ot r(x\boo)+x\boi\ot s(x\boo)+r(x\boo)\ot x\bi+s(x\boo)\ot x\bi+{Q}(x).
\end{eqnarray*}
Thus we obtain that $\Delta'_E\varphi(x)=(\varphi\otimes \varphi)\Delta_E(x)$ if and only if the conditions \eqref{comorph1},  \eqref{comorph2},  \eqref{comorph3} and \eqref{comorph4} hold. By definition, we obtain that $\varphi=\varphi_{r, s}$ is an isomorphism if and only if $s: V\rightarrow V$ is a linear isomorphism.
\end{proof}

Let $({A},\Delta_A)$ be an    alternative  coalgebra and $V$ be a vector space. Two   alternative   coalgebraic extending systems $\Omega^{(i)}({A}, V)$ and ${\Omega'^{(i)}}({A}, V)$  are called equivalent if $\varphi_{r, s}$ is an isomorphism.  We denote it by $\Omega^{(i)}({A}, V)\equiv{\Omega'^{(i)}}({A}, V)$.
From the above lemmas, we obtain the following result.
\begin{theorem}\label{thm3-2}
Let $({A}, \Delta_A)$ be an   alternative   coalgebra, $E$ a vector space containing ${A}$ as a subspace and
$V$ be a ${A}$-complement in $E$. Denote $\mathcal{HC}(V,{A}):=\mathcal{C}^{(1)}({A}, V)\sqcup\mathcal{C}^{(2)}({A}, V) /\equiv$. Then the map
\begin{eqnarray*}
&&\Psi: \mathcal{HC}_{{A}}^2(V, {A})\rightarrow CExtd(E,{A}),\\
&&\overline{\Omega^{(1)}({A}, V)}\mapsto A^{P}\# {}^{} V,
 \quad \overline{\Omega^{(2)}({A}, V)}\mapsto A^{}\# {}^{Q} V
\end{eqnarray*}
is bijective, where $\overline{\Omega^{(i)}({A}, V)}$ is the equivalence class of $\Omega^{(i)}({A}, V)$ under $\equiv$.
\end{theorem}

\subsection{Extending structures for     alternative  bialgebras}
Let $(A, \cdot, \Delta_A)$ be an     alternative  bialgebra. From (CBB1) and (CBB2)  we have the following two cases.

The first case is that we assume $Q=0$ and $\ppr, \ppl$ to be trivial. Then by the above Theorem \ref{main2}, we obtain the following result.

\begin{theorem}\label{thm-41}
Let $(A, \cdot, \Delta_A)$ be an     alternative  bialgebra and $V$ a vector space.
An extending datum of ${A}$ by $V$ of type (I) is  $\Omega^{(1)}({A}, V)=(\trr, \trl, \phi, \psi, \rho, \gamma, \theta, P, \cdot_V, \Delta_V)$ consisting of  linear maps
\begin{eqnarray*}
\trr: V\otimes {A}\rightarrow V,~~~~\trl: A\otimes V\rightarrow V,~~~~\theta:  A\otimes A \rightarrow {V},~~~~\cdot_V: V\otimes V \rightarrow V,\\
 \phi : A \to V\otimes A, \quad{\psi}: V\to  V\otimes A,~~~~{P}: A\rightarrow {V}\otimes {V},~~~~\Delta_V: V\rightarrow V\otimes V,\\
 \rho: V\to A \otimes V,~~~~\gamma : V\to V \otimes A.
\end{eqnarray*}
Then the unified product $A^{P}_{}\# {}^{}_{\theta}\, V$ with bracket
\begin{align}
(a+ x) (b+ y): =ab+( xy+ a\trr y+x\trl b+\theta(a, b))
\end{align}
and comultiplication
\begin{eqnarray}
\Delta_E(a)=\Delta_A(a)+{\phi}(a)+{\psi}(a)+P(a), \quad \Delta_E(x)=\Delta_V(x)+{\rho}(x)+{\gamma}(x)
\end{eqnarray}
forms an     alternative  bialgebra if and only if $A_{}\# {}_{\theta} V$ forms an  alternative  algebra, $A^{P}\# {}^{} \, V$ forms an   alternative   coalgebra and the following conditions are satisfied:
\begin{enumerate}
\item[(E1)]  $\phi(ab)+\gamma(\theta(a, b))\\
 =-a\lmi\ot b a\loo+b\loi\ot a b\loo+b\loi\ot b\loo a+(a\loi\trl b)\ot a\loo\\
+(a\lmi\trl b)\ot a\loo-(a\trr b\loi)\ot b\loo+\theta(a\li, b)\ot a\lii+\theta(a\lii, b)\ot a\li-\theta(a, b\li)\ot b\lii$,

\item[(E2)] $\psi(a b)+\rho(\theta(a, b))\\
=a\loo b\ot a\lmi+a\loo b\ot a\loi-a b\loo\ot b\lmi-a\loo\ot (b\trr a\loi)+b\loo\ot(a\trr b\lmi)\\
+b\loo\ot(b\lmi\trl a)+b\li\ot\theta(a, b\lii)+b\li\ot\theta(b\lii, a)-a\lii\ot\theta(b, a\li)$,

\item[(E3)] $\rho(x y)=-x_{[1]} \otimes y x_{[0]} +y\boi\ot x y\boo+y\boi\ot y\boo x$,

\item[(E4)] $\gamma(x y)=x\boo y\ot x\bi+x\boo y\ot x\boi-xy_{[0]}\otimes y_{[1]}$,

\item[(E5)] $\Delta_{V}(a \trr y)$\\
$=(a\loo\trr y)\ot a\lmi+(a\loo\trr y)\ot a\loi-a\lmi\ot(y\trl a\loo)+y\li\ot(a\trr y\lii)\\
+y\li\ot(y\lii\trl a)-\left(a \trr y_{1}\right) \otimes y_{2}+y\boo\ot\theta(a, y\bi)+y\boo\ot\theta(y\bi, a)\\
-\theta(a, y\boi)\ot y\boo+a\ppi y\ot a\pii+a\pii y\ot a\ppi-a\pii\ot y a\ppi$,

\item[(E6)] $\Delta_{V}(x \trl b)$\\
$=(x\li\trl b)\ot x\lii+(x\lii\trl b)\ot x\li-x\lii\ot(b\trr x\li)+b\loi\ot(x\trl b\loo)\\
+b\loi\ot(b\loo\trr x)-\left(x\trl b_{(0)}\right) \otimes b_{(1)}+\theta(x\boi, b)\ot x\boo+\theta(x\bi, b)\ot x\boo\\
-x\boo\ot\theta(b, x\boi)+b\ppi\ot x b\pii+b\ppi\ot b\pii x-x b\ppi\ot b\pii$,

\item[(E7)]$\Delta_{V}(\theta(a,b))+P(a, b)$\\
$=\theta(a\loo, b)\ot a\lmi+\theta(a\loo, b)\ot a\loi-\theta(a,b\loo)\ot b\lmi+b\loi\ot\theta(a, b\loo)\\
+b\loi\ot\theta(b\loo, a)-a\lmi\ot\theta(b, a\loo)+(a\ppi\trl b)\ot a\pii+(a\pii\trl b)\ot a\ppi\\
-(a\trr b\ppi)\ot b\pii+b\ppi\ot(a\trr b\pii)+b\ppi\ot(b\pii\trl a)-a\pii\ot(b\trr a\ppi)$,

\item[(E8)]
 $\gamma(x\trl b)=(x\boo\trl b)\ot x\bi+(x\boo\trl b)\ot x\boi-x\boo\ot b x\boi-\left(x\trl b_{1}\right) \otimes b_{2}-x b_{(-1)} \otimes b_{(0)}$,

\item[(E9)]
$\rho(a \trr y)=-a\lii\ot(y\trl a\li)-a\loo\ot y a\loi+y\boi\ot(a\trr y\boo)+y\boi\ot(y\boo\trl a)-a y_{[-1]} \otimes y_{[0]}$,

\item[(E10)]
$\rho(x\trl b)=x\boi b\ot x\boo+x\bi b\ot x\boo-x\bi\ot(b\trr x\boo)\\
+b\li\ot(x\trl b\lii)+b\loo\ot x b\lmi+b\li\ot(b\lii\trr x)+b\loo\ot b\lmi x$,

\item[(E11)]
$\gamma(a \trr y)=(a\li\trr y)\ot a\lii+a\loi y\ot a\loo+(a\lii\trr y)\ot a\li\\
+a\lmi y\ot a\loo+y\boo\ot a y\bi+y\boo\ot y\bi a-\left(a\trr y_{[0]}\right)\otimes y_{[1]}$,

\item[(E12)]  $\phi(b a)+\gamma(\theta(b, a))+\tau\psi(b a)+\tau\rho(\theta(b, a))\\
=a\loi\ot b a\loo+(b\trr a\lmi)\ot a\loo+(b\loi\trl a)\ot b\loo+b\lmi\ot b\loo a\\
+\theta(b, a\lii)\ot a\li+\theta(b\li, a)\ot b\lii$,

\item[(E13)] $\psi(b a)+\rho(\theta(b, a))+\tau\phi(b a)+\tau\gamma(\theta(b, a))\\
=a\loo\ot(b\trr a\lmi)+b a\loo \ot a\loi+ b\loo a\ot b\lmi+b\loo\ot(b\loi\trl a)\\
+a\li\ot\theta(b, a\lii)+b\lii\ot\theta(b\li, a)$,

\item[(E14)] $\rho(y x)+\tau\gamma(y x)=x\boi\ot y x\boo+y\bi\ot y\boo x$,

\item[(E15)] $\gamma(y x)+\tau\rho(y x)=y x\boo\ot x\boi+y\boo x\ot y\bi$,

\item[(E16)] $\Delta_{V}(y\trl a)+\tau\Delta_{V}(y\trl a)\\
=a\loi\ot(y\trl a\loo)+(y\trl a\loo)\ot a\loi+(y\li\trl a)\ot y\lii+y\lii\ot(y\li\trl a)\\
+a\ppi\ot y a\pii+y a\pii\ot a\ppi+\theta(y\boi, a)\ot y\boo+y\boo\ot\theta(y\boi, a)$,

\item[(E17)] $\Delta_{V}(b \trr x)+\tau\Delta_{V}(b \trr x)\\
=x\li\ot(b\trr x\lii)+(b\trr x\lii)\ot x\li+(b\loo\trr x)\ot b\lmi+b\lmi\ot (b\loo\trr x)\\
+x\boo\ot\theta(b, x\bi)+\theta(b, x\bi)\ot x\boo+b\ppi x\ot b\pii+b\pii\ot b\ppi x$,

\item[(E18)]
$\gamma(y\trl a)+\tau\rho(y\trl a)=(y\trl a\lii)\ot a\li+y a\lmi\ot a\loo+y\boo\ot y\boi a+(y\boo\trl a)\ot y\bi$,

\item[(E19)]
$\gamma(b\trr x)+\tau\rho(b\trr x)=x\boo\ot b x\bi+(b\trr x\boo)\ot x\boi+(b\li\trr x)\ot b\lii+b\loi x\ot b\loo$,

\item[(E20)]
$\rho(y\trl a)+\tau\gamma(y\trl a)=a\li\ot(y\trl a\lii)+a\loo\ot y a\lmi+y\boi a\ot y\boo+y\bi\ot(y\boo\trl a)$,

\item[(E21)]
$\rho(b\trr x)+\tau\gamma(b\trr x)=x\boi\ot(b\trr x\boo)+b x\bi\ot x\boo+b\lii\ot(b\li\trr x)+b\loo\ot b\loi x$,

\item[(E22)]
$P(b,a)+\Delta_{V}(\theta(b, a))+\tau P(b, a)+\tau\Delta_{V}(\theta(b, a))\\
=a\loi\ot\theta(b, a\loo)+\theta(b, a\loo)\ot a\loi+a\ppi\ot(b\trr a\pii)+(b\trr a\pii)\ot a\ppi\\
+\theta(b\loo, a)\ot b\lmi+b\lmi\ot\theta(b\loo, a)+(b\ppi\trl a)\ot b\pii+b\pii\ot(b\ppi\trl a)$,

\item[(E23)] $\Delta_{V}(xy)\\
=x\li y\ot x\lii+x\lii y\ot x\li-x\lii\ot y x\li+y\li\ot x y\lii\\
+y\li\ot y\lii x-x y\li\ot y\lii+(x\boi\trr y)\ot x\boo+(x\bi\trr y)\ot x\boo\\
-x\boo\ot(y\trl x\boi)+y\boo\ot(x\trl y\bi)+y\boo\ot(y\bi\trr x)-(x\trl y\boi)\ot y\boo,$

\item[(E24)] $\Delta_{V}(yx)+\tau\Delta_{V}(yx)\\
=x\li \ot y x\lii+y x\lii \ot x\li+y\li x\ot y\lii+y\lii\ot y\li x\\
+x\boo\ot(y\trl x\bi)+(y\trl x\bi)\ot x\boo+(y\boi\trr x)\ot y\boo+y\boo\ot(y\boi\trr x).$
\end{enumerate}

Conversely, any     alternative  bialgebra structure on $E$ with the canonical projection map $p: E\to A$ both an  alternative  algebra homomorphism and an   alternative   coalgebra homomorphism is of this form.
\end{theorem}
Note that in this case, $(V, \cdot, \Delta_V)$ is a  braided     alternative  bialgebra. Although $(A, \cdot, \Delta_A)$ is not an  alternative   sub-bialgebra of $E=A^{P}_{}\# {}^{}_{\theta}\, V$, but it is indeed an     alternative  bialgebra and a subspace $E$.
Denote the set of all     alternative  bialgebraic extending datum of type (I) by $\mathcal{IB}^{(1)}({A}, V)$.

The second case is that we assume $P=0, \theta=0$ and $\phi, \psi$ to be trivial. Then by the above Theorem \ref{main2}, we obtain the following result.

\begin{theorem}\label{thm-42}
Let $A$ be an     alternative  bialgebra and $V$ a vector space.
An extending datum of ${A}$ by $V$ of type (II) is  $\Omega^{(2)}({A}, V)=(\ppr, \ppl, \trr, \trl, \sigma, \rho, \gamma, Q,  \cdot_V, \Delta_V)$ consisting of  linear maps
\begin{eqnarray*}
\trl: V\otimes {A}\rightarrow {V},~~~~\trr: A\otimes {V}\rightarrow V,~~~~\sigma:  V\otimes V \rightarrow {A},~~~\cdot_V: V\otimes V \rightarrow V,\\
{\rho}: V\to  A\otimes V,~~~~{\gamma}: V\to  V\otimes A,~~~~{Q}: V\rightarrow {A}\otimes {A},~~~~\Delta_V: V\rightarrow V\otimes V,\\
\ppr: V\otimes A \to A,~~~~ \ppl: A\otimes V \to A.
\end{eqnarray*}
Then the unified product $A^{}_{\sigma}\# {}^{Q}_{}\, V$ with bracket
\begin{align}
(a+ x)(b+ y): =\big(ab+x\ppr b+a\ppl y+\sigma(x, y))+( xy+x\trl b+a\trr y\big)
\end{align}
and comultiplication
\begin{eqnarray}
\Delta_E(a)=\Delta_A(a), \quad \Delta_E(x)=\Delta_V(x)+{\rho}(x)+{\gamma}(x)+Q(x)
\end{eqnarray}
forms an     alternative  bialgebra if and only if $A_{\sigma}\# {}_{} V$ forms an alternative algebra, $A^{}\# {}^{Q}V$ forms an alternative  coalgebra and the following conditions are satisfied:

\begin{enumerate}
\item[(F1)] $\rho(x y)\\
=(x\bi\ppl y)\ot x\boo+(x\boi\ppl y)\ot x\boo-x_{[1]} \otimes y x_{[0]}-\left(x \ppr y_{[-1]}\right) \otimes y_{[0]}  \\
+y\boi\ot x y\boo+y\boi\ot y\boo x+\sigma(x\li,y)\ot x\lii+\sigma(x\lii, y)\ot x\li-\sigma(x, y\li)\ot y\lii\\
+y\qi\ot (y\qii\trr x)+y\qi\ot (x\trl y\qii)-x\qii\ot(y\trl x\qi)$,

\item[(F2)] $\gamma(x y)\\
=x\boo y\ot x\bi+x\boo y\ot x\boi-x_{[0]}\otimes (y\ppr x\boi)-xy_{[0]}\otimes y_{[1]}+y\boo\ot (x\ppr y\bi)\\
+y\boo\ot(y\bi\ppl x)+(x\qi\trr y)\ot x\qii+(x\qii\trr y)\ot x\qi-x\lii\ot\sigma(y, x\li)\\
+y\li\ot\sigma(x, y\lii)+y\li\ot\sigma(y\lii, x)-(x\trl y\qi)\ot y\qii$,

\item[(F3)] $\Delta_{A}(x \ppr b)+Q(x\trl  b)$ \\
$=(x\boo\ppr b)\ot x\bi+(x\boo\ppr b)\ot x\boi-x\bi\ot(b\ppl x\boo)-\left(x \ppr b_{1}\right) \otimes b_{2}\\
+b\li\ot(x\ppr b\lii)+b\li\ot(b\lii\ppl x)+x\qi b\ot x\qii+x\qii b\ot x\qi-x\qii\ot b x\qi$,

\item[(F4)] $\Delta_{A}(a\ppl y)+Q(a\trr y)$\\
$=(a\li\ppl y)\ot a\lii+\left(a_{2} \ppl y\right)\ot a\li-a\lii\ot(y\ppr a\li)-\left(a\ppl y_{[0]}\right) \otimes y_{[1]}\\
+y\boi\ot(a\ppl y\boo)+y\boi\ot(y\boo\ppr a)+y\qi\ot a y\qii+y\qi\ot y\qii a-a y\qi\ot y\qii$,

\item[(F5)] $\Delta_{V}(a \trr y)=y\li\ot(a\trr y\lii)+y\li\ot(y\lii\trl a)-\left(a \trr y_{1}\right) \otimes y_{2}$,

\item[(F6)] $\Delta_{V}(x \trl b)=(x\li\trl b)\ot x\lii+(x\lii\trl b)\ot x\li-x\lii\ot(b\trr x\li)$,

\item[(F7)]$\Delta_{A}(\sigma(x, y))+Q(x, y)$\\
$=\sigma(x\boo, y)\ot x\bi+\sigma(x\boo, y)\ot x\boi-\sigma(x, y\boo)\ot y\bi+y\boi\ot\sigma(x, y\boo)\\
+y\boi\ot\sigma(y\boo, x)-x\bi\ot\sigma(y, x\boo)+(x\qi\ppl y)\ot x\qii+(x\qii\ppl y)\ot x\qi\\
-x\qii\ot(y\ppr x\qi)+y\qi\ot(x\ppr y\qii)+y\qi\ot(y\qii\ppl x)-(x\ppr y\qi)\ot y\qii$,

\item[(F8)]
 $\gamma(x\trl b)=(x\boo\trl b)\ot x\bi+(x\boo\trl b)\ot x\boi-x\lii\ot(b\ppl x\li)-x\boo\ot b x\boi-\left(x\trl b_{1}\right) \otimes b_{2}$,

\item[(F9)]
$\gamma(a \trr y)=(a\li\trr y)\ot a\lii+(a\lii\trr y)\ot a\li+y\li\ot(a\ppl y\lii)\\
+y\boo\ot a y\bi+y\li\ot(y\lii\ppr a)+y\boo\ot y\bi a-\left(a\trr y_{[0]}\right)\otimes y_{[1]}$,

\item[(F10)]
$\rho(a \trr y)=-a\lii\ot(y\trl a\li)+y\boi\ot(a\trr y\boo)-\left(a\ppl y_{1}\right) \otimes y_{2}+y\boi\ot(y\boo\trl a)-a y_{[-1]} \otimes y_{[0]}$,

\item[(F11)]
$\rho(x\trl b)=(x\li\ppr b)\ot x\lii+x\boi b\ot x\boo+(x\lii\ppr b)\ot x\li\\
+x\bi b\ot x\boo-x\bi\ot(b\trr x\boo)+b\li\ot(x\trl b\lii)+b\li\ot(b\lii\trr x)$,

\item[(F12)] $\rho(y x)+\tau\gamma(y x)\\
=x\boi\ot y x\boo+(y\ppr x\bi)\ot x\boo+(y\boi\ppl x)\ot y\boo+y\bi\ot y\boo x\\
+x\qi\ot(y\trl x\qii)+y\qii\ot(y\qi\trr x)+\sigma(y, x\lii)\ot x\li+\sigma(y\li, x)\ot y\lii$,

\item[(F13)] $\gamma(y x)+\tau\rho(y x)\\
=y x\boo\ot x\boi+y\boo x\ot y\bi+y\boo\ot(y\boi\ppl x)+x\boo\ot(y\ppr x\bi)\\
+x\li\ot\sigma(y, x\lii)+y\lii\ot\sigma(y\li, x)+(y\trl x\qii)\ot x\qi+(y\qi\trr x)\ot y\qii$,

\item[(F14)] $\Delta_{A}(b \ppl x)+Q(b\trr x)+\tau\Delta_{A}(b \ppl x)+\tau Q(b\trr x)\\
=x\boi\ot(b\ppl x\boo)+(b\ppl x\boo)\ot x\boi+(b\li\ppl x)\ot b\lii+b\lii\ot(b\li\ppl x)\\
+x\qi\ot b x\qii+b x\qii\ot x\qi$,

\item[(F15)] $\Delta_{A}(y\ppr a)+Q(y\trl a)+\tau\Delta_{A}(y\ppr a)+\tau Q(y\trl a)\\
=a\li\ot(y\ppr a\lii)+(y\ppr a\lii)\ot a\li+(y\boo\ppr a)\ot y\bi+y\bi\ot(y\boo\ppr a)\\
+y\qi a\ot y\qii+y\qii\ot y\qi a$,

\item[(F16)] $\Delta_{V}(y\trl a)+\tau\Delta_{V}(y\trl a)=(y\li\trl a)\ot y\lii+y\lii\ot(y\li\trl a)$,

\item[(F17)] $\Delta_{V}(b \trr x)+\tau\Delta_{V}(b \trr x)=x\li\ot(b\trr x\lii)+(b\trr x\lii)\ot x\li$,

\item[(F18)]
$\gamma(y\trl a)+\tau\rho(y\trl a)\\
=(y\trl a\lii)\ot a\li+(y\boo\trl a)\ot y\bi+y\lii\ot(y\li\ppr a)+y\boo\ot y\boi a$,

\item[(F19)]
$\gamma(b\trr x)+\tau\rho(b\trr x)\\
=x\li\ot(b\ppl x\lii)+x\boo\ot b x\bi+(b\trr x\boo)\ot x\boi+(b\li\trr x)\ot b\lii$,

\item[(F20)]
$\rho(y\trl a)+\tau\gamma(y\trl a)\\
=a\li\ot(y\trl a\lii)+(y\li\ppr a)\ot y\lii+y\boi a\ot y\boo+y\bi\ot(y\boo\trl a)$,

\item[(F21)]
$\rho(b\trr x)+\tau\gamma(b\trr x)\\
=x\boi\ot(b\trr x\boo)+(b\ppl x\lii)\ot x\li+b x\bi\ot x\boo+b\lii\ot(b\li\trr x)$,

\item[(F22)]
$\Delta_{A}(\sigma(y, x))+Q(y, x)+\tau\Delta_{A}(\sigma(y, x))+\tau Q(y, x)\\
=x\boi\ot\sigma(y, x\boo)+\sigma(y, x\boo)\ot x\boi+x\qi\ot(y\ppr x\qii)+(y\ppr x\qii)\ot x\qi\\
+\sigma(y\boo, x)\ot y\bi+y\bi\ot\sigma(y\boo, x)+(y\qi\ppl x)\ot y\qii+y\qii\ot(y\qi\ppl x)$,

\item[(F23)] $\Delta_{V}(xy)\\
=x\li y\ot x\lii+x\lii y\ot x\li-x\lii\ot y x\li+y\li\ot x y\lii\\
+y\li\ot y\lii x-x y\li\ot y\lii+(x\boi\trr y)\ot x\boo+(x\bi\trr y)\ot x\boo\\
-x\boo\ot(y\trl x\boi)+y\boo\ot(x\trl y\bi)+y\boo\ot(y\bi\trr x)-(x\trl y\boi)\ot y\boo,$

\item[(F24)] $\Delta_{V}(yx)+\tau\Delta_{V}(yx)\\
=x\li \ot y x\lii+y x\lii \ot x\li+y\li x\ot y\lii+y\lii\ot y\li x\\
+x\boo\ot(y\trl x\bi)+(y\trl x\bi)\ot x\boo+(y\boi\trr x)\ot y\boo+y\boo\ot(y\boi\trr x).$
\end{enumerate}

Conversely, any     alternative  bialgebra structure on $E$ with the canonical injection map $i: A\to E$ both an   alternative   algebra homomorphism and an  alternative    coalgebra homomorphism is of this form.
\end{theorem}
Note that in this case, $(A, \cdot, \Delta_A)$ is an   alternative   sub-bialgebra of $E=A^{}_{\sigma}\# {}^{Q}_{}\, V$ and $(V, \cdot, \Delta_V)$ is a  braided    alternative  bialgebra.
Denote the set of all      alternative  bialgebraic extending datum of type (II) by $\mathcal{IB}^{(2)}({A},V)$.

In the above two cases, we find that  the braided     alternative  bialgebra $V$ play a special role in the extending problem of    alternative  bialgebra $A$.
Note that $A^{P}_{}\# {}^{}_{\theta}\, V$ and $A^{}_{\sigma}\# {}^{Q}_{}\, V$ are all     alternative  bialgebra structures on $E$.
Conversely,  any     alternative  bialgebra extending system $E$ of ${A}$  through $V$ is isomorphic to such two types.
Now from Theorem \ref{thm-41}, Theorem \ref{thm-42} we obtain the main result of in this section,
which solve the extending problem for     alternative  bialgebra.

\begin{theorem}\label{bim1}
Let $({A}, \cdot, \Delta_A)$ be an     alternative  bialgebra, $E$ a vector space containing ${A}$ as a subspace and $V$ be a complement of ${A}$ in $E$.
Denote by
$$\mathcal{HLB}(V, {A}): =\mathcal{IB}^{(1)}({A}, V)\sqcup\mathcal{IB}^{(2)}({A}, V)/\equiv.$$
Then the map
\begin{eqnarray}
&&\Upsilon: \mathcal{HLB}(V, {A})\rightarrow BExtd(E, {A}),\\
&&\overline{\Omega^{(1)}({A}, V)}\mapsto A^{P}_{}\# {}^{}_{\theta}\, V, \quad   \overline{\Omega^{(2)}({A}, V)}\mapsto A^{}_{\sigma}\# {}^{Q}_{}\, V
\end{eqnarray}
is bijective, where $\overline{\Omega^{(i)}({A}, V)}$ is the equivalence class of $\Omega^{(i)}({A}, V)$ under $\equiv$.
\end{theorem}

A very special case is that when $\ppr$ and $\ppl$ are trivial in the above Theorem \ref{thm-42}.    We obtain the following result.
\begin{corollary}\label{thm4}
Let $A$ be an     alternative  bialgebra and $V$ a vector space.
An extending datum of ${A}$ by $V$ is  $\Omega({A}, V)=(\trr, \trl, \sigma, \cdot_V, \rho, \gamma, Q, \Delta_V)$ consisting of  eight linear maps
\begin{eqnarray*}
\trl: V\otimes {A}\rightarrow {V},~~~~\trr: A\otimes {V}\rightarrow V,~~~~\sigma:  V\otimes V \rightarrow {A},~~~\cdot_V: V\otimes V \rightarrow V,\\
{\rho}: V\to  A\otimes V,~~~~{\gamma}: V\to  V\otimes A,~~~~{Q}: V\rightarrow {A}\otimes {A},~~~~\Delta_V: V\rightarrow V\otimes V.
\end{eqnarray*}
Then the unified product $A^{}_{\sigma}\# {}^{Q}_{}\, V$ with bracket
\begin{align}
(a+ x) (b+ y): =(ab+\sigma(x, y))+  (xy+x\trl b+a\trr y)
\end{align}
and comultiplication
\begin{eqnarray}
\Delta_E(a)=\Delta_A(a), \quad \Delta_E(x)=\Delta_V(x)+{\rho}(x)+{\gamma}(x)+Q(x)
\end{eqnarray}
forms an     alternative  bialgebra if and only if $A_{\sigma}\# {}_{} V$ forms an alternative algebra, $A^{}\# {}^{Q} \, V$ forms an alternative  coalgebra and the following conditions are satisfied:
\begin{enumerate}
\item[(G1)] $\rho(x y)\\
=-x_{[1]} \otimes y x_{[0]}+y\boi\ot x y\boo+y\boi\ot y\boo x+\sigma(x\li, y)\ot x\lii+\sigma(x\lii, y)\ot x\li\\
-\sigma(x, y\li)\ot y\lii+y\qi\ot (y\qii\trr x)+y\qi\ot (x\trl y\qii)-x\qii\ot(y\trl x\qi)$,

\item[(G2)] $\gamma(x y)\\
=x\boo y\ot x\bi+x\boo y\ot x\boi-xy_{[0]}\otimes y_{[1]}+(x\qi\trr y)\ot x\qii+(x\qii\trr y)\ot x\qi\\
-x\lii\ot\sigma(y, x\li)+y\li\ot\sigma(x, y\lii)+y\li\ot\sigma(y\lii, x)-(x\trl y\qi)\ot y\qii$,

\item[(G3)] $Q(x\trl  b)=x\qi b\ot x\qii+x\qii b\ot x\qi-x\qii\ot b x\qi$,

\item[(G4)] $Q(a\trr y)=y\qi\ot a y\qii+y\qi\ot y\qii a-a y\qi\ot y\qii$,

\item[(G5)] $\Delta_{V}(a \trr y)=y\li\ot(a\trr y\lii)+y\li\ot(y\lii\trl a)-\left(a \trr y_{1}\right) \otimes y_{2}$,

\item[(G6)] $\Delta_{V}(x \trl b)=(x\li\trl b)\ot x\lii+(x\lii\trl b)\ot x\li-x\lii\ot(b\trr x\li)$,

\item[(G7)]$\Delta_{A}(\sigma(x,y))+Q(x, y)$\\
$=\sigma(x\boo, y)\ot x\bi+\sigma(x\boo, y)\ot x\boi-\sigma(x, y\boo)\ot y\bi+y\boi\ot\sigma(x, y\boo)\\
+y\boi\ot\sigma(y\boo, x)-x\bi\ot\sigma(y, x\boo)$,

\item[(G8)]
 $\gamma(x\trl b)=(x\boo\trl b)\ot x\bi+(x\boo\trl b)\ot x\boi-x\boo\ot b x\boi-\left(x\trl b_{1}\right) \otimes b_{2}$,

\item[(G9)]
$\rho(a \trr y)=-a\lii\ot(y\trl a\li)+y\boi\ot(a\trr y\boo)+y\boi\ot(y\boo\trl a)-a y_{[-1]} \otimes y_{[0]}$,

\item[(G10)]
$\rho(x\trl b)=x\boi b\ot x\boo+x\bi b\ot x\boo-x\bi\ot(b\trr x\boo)+b\li\ot(x\trl b\lii)+b\li\ot(b\lii\trr x)$,

\item[(G11)]
$\gamma(a \trr y)=(a\li\trr y)\ot a\lii+(a\lii\trr y)\ot a\li+y\boo\ot a y\bi+y\boo\ot y\bi a-\left(a\trr y_{[0]}\right)\otimes y_{[1]}$,

\item[(G12)] $\rho(y x)+\tau\gamma(y x)\\
=x\boi\ot y x\boo+y\bi\ot y\boo x+x\qi\ot(y\trl x\qii)\\
+y\qii\ot(y\qi\trr x)+\sigma(y, x\lii)\ot x\li+\sigma(y\li, x)\ot y\lii$,

\item[(G13)] $\gamma(y x)+\tau\rho(y x)\\
=y x\boo\ot x\boi+y\boo x\ot y\bi+x\li\ot\sigma(y, x\lii)\\
+y\lii\ot\sigma(y\li, x)+(y\trl x\qii)\ot x\qi+(y\qi\trr x)\ot y\qii$,

\item[(G14)] $Q(b\trr x)+\tau Q(b\trr x)=x\qi\ot b x\qii+b x\qii\ot x\qi$,

\item[(G15)] $Q(y\trl a)+\tau Q(y\trl a)=y\qi a\ot y\qii+y\qii\ot y\qi a$,

\item[(G16)] $\Delta_{V}(y\trl a)+\tau\Delta_{V}(y\trl a)=(y\li\trl a)\ot y\lii+y\lii\ot(y\li\trl a)$,

\item[(G17)] $\Delta_{V}(b \trr x)+\tau\Delta_{V}(b \trr x)=x\li\ot(b\trr x\lii)+(b\trr x\lii)\ot x\li$,

\item[(G18)]
$\gamma(y\trl a)+\tau\rho(y\trl a)=(y\trl a\lii)\ot a\li+(y\boo\trl a)\ot y\bi+y\boo\ot y\boi a$,

\item[(G19)]
$\gamma(b\trr x)+\tau\rho(b\trr x)=x\boo\ot b x\bi+(b\trr x\boo)\ot x\boi+(b\li\trr x)\ot b\lii$,

\item[(G20)]
$\rho(y\trl a)+\tau\gamma(y\trl a)=a\li\ot(y\trl a\lii)+y\boi a\ot y\boo+y\bi\ot(y\boo\trl a)$,

\item[(G21)]
$\rho(b\trr x)+\tau\gamma(b\trr x)=x\boi\ot(b\trr x\boo)+b x\bi\ot x\boo+b\lii\ot(b\li\trr x)$,

\item[(G22)]
$\Delta_{A}(\sigma(y, x))+Q(y, x)+\tau\Delta_{A}(\sigma(y, x))+\tau Q(y, x)\\
=x\boi\ot\sigma(y, x\boo)+\sigma(y,x\boo)\ot x\boi+\sigma(y\boo, x)\ot y\bi+y\bi\ot\sigma(y\boo, x)$,

\item[(G23)] $\Delta_{V}(xy)\\
=x\li y\ot x\lii+x\lii y\ot x\li-x\lii\ot y x\li+y\li\ot x y\lii\\
+y\li\ot y\lii x-x y\li\ot y\lii+(x\boi\trr y)\ot x\boo+(x\bi\trr y)\ot x\boo\\
-x\boo\ot(y\trl x\boi)+y\boo\ot(x\trl y\bi)+y\boo\ot(y\bi\trr x)-(x\trl y\boi)\ot y\boo,$

\item[(G24)] $\Delta_{V}(yx)+\tau\Delta_{V}(yx)\\
=x\li \ot y x\lii+y x\lii \ot x\li+y\li x\ot y\lii+y\lii\ot y\li x\\
+x\boo\ot(y\trl x\bi)+(y\trl x\bi)\ot x\boo+(y\boi\trr x)\ot y\boo+y\boo\ot(y\boi\trr x).$
\end{enumerate}
\end{corollary}

\section*{Acknowledgements}
This is a primary edition, something should be modified in the future.

\vskip7pt
\footnotesize{
\noindent Tao Zhang\\
College of Mathematics and Information Science,\\
Henan Normal University, Xinxiang 453007, P. R. China;\\
 E-mail address: \texttt{{zhangtao@htu.edu.cn}}

 \vskip7pt
\footnotesize{
\noindent Fang Yang\\
College of Mathematics and Information Science,\\
Henan Normal University, Xinxiang 453007, P. R. China;\\
 E-mail address: \texttt{{htuyangfang@163.com}}

\end{document}